\documentclass[11pt,reqno]{amsart}

\usepackage{amssymb,amsmath,amsthm,amsfonts}
\usepackage{hyperref, multirow, bm, color, marginnote, graphicx, wrapfig, subcaption}
\usepackage{diagbox}

\makeatletter
\renewcommand{\tocsection}[3]{%
  \indentlabel{\@ifnotempty{#2}{\bfseries\ignorespaces#1 #2\quad}}\bfseries#3}
\renewcommand{\tocsubsection}[3]{%
  \indentlabel{\@ifnotempty{#2}{\ignorespaces#1 #2\quad}}#3}

\newcommand\@dotsep{4.5}
\def\@tocline#1#2#3#4#5#6#7{\relax
  \ifnum #1>\c@tocdepth 
  \else
    \par \addpenalty\@secpenalty\addvspace{#2}%
    \begingroup \hyphenpenalty\@M
    \@ifempty{#4}{%
      \@tempdima\csname r@tocindent\number#1\endcsname\relax
    }{%
      \@tempdima#4\relax
    }%
    \parindent\z@ \leftskip#3\relax \advance\leftskip\@tempdima\relax
    \rightskip\@pnumwidth plus1em \parfillskip-\@pnumwidth
    #5\leavevmode\hskip-\@tempdima{#6}\nobreak
    \leaders\hbox{$\m@th\mkern \@dotsep mu\hbox{.}\mkern \@dotsep mu$}\hfill
    \nobreak
    \hbox to\@pnumwidth{\@tocpagenum{\ifnum#1=1\bfseries\fi#7}}\par
    \nobreak
    \endgroup
  \fi}
\AtBeginDocument{%
\expandafter\renewcommand\csname r@tocindent0\endcsname{0pt}
}
\def\l@subsection{\@tocline{2}{0pt}{2.5pc}{5pc}{}}
\makeatother

\usepackage[margin=1in]{geometry}

\newcommand{\R}{\mathbb{R}}

\newcommand{\N}{\mathbb{N}}

\newcommand{\bars}{\overline s}

\newcommand{\X}{\bm{X}}
\newcommand{\be}{\bm{e}}

\newcommand{\bv}{\bm{v}}
\newcommand{\p}{\partial}

\newcommand{\abs}[1]{\left\lvert #1 \right\rvert}
\newcommand{\norm}[1]{\left\lVert #1 \right\rVert}
\newcommand{\wh}[1]{\widehat{#1}}
\newcommand{\wt}[1]{\widetilde{#1}}
\newcommand{\mc}[1]{\mathcal{#1}}

\newtheorem{theorem}{Theorem}[section]
\newtheorem{lemma}[theorem]{Lemma}

\theoremstyle{definition}

\allowdisplaybreaks

\graphicspath{{figures/}}

\begin{document}
 \title[Viscoelastic resistive force theory]{Well-posedness of a viscoelastic resistive force theory and applications to swimming}

 \author{Laurel Ohm}
\address[L. Ohm]{Department of Mathematics, University of Wisconsin - Madison, Madison, WI 53706}
\email{lohm2@wisc.edu}

\begin{abstract} 
We propose and analyze a simple model for the evolution of an immersed, inextensible filament which incorporates linear viscoelastic effects of the surrounding fluid. The model is a closed-form system of equations along the curve only which includes a `memory' term due to viscoelasticity.
For a planar filament, given a forcing in the form of a preferred curvature, we prove well-posedness of the fiber evolution as well as the existence of a unique time-periodic solution in the case of time-periodic forcing. Moreover, we obtain an expression for the swimming speed of the filament in terms of the preferred curvature. The swimming speed depends in a complicated way on the viscoelastic parameters corresponding to the fluid relaxation time and additional polymeric viscosity. We study this expression in detail, accompanied by numerical simulations, and show that this simple model can capture complex effects of viscoelasticity on swimming. In particular, the viscoelastic swimmer is shown to be faster than its Newtonian counterpart in some situations and slower in others. Strikingly, we even find an example where viscoelastic effects may lead to a reversal in swimming direction from the Newtonian setting, although this occurs when the displacement for both the Newtonian and viscoelastic swimmers is practically negligible. 
\end{abstract}

\maketitle

\tableofcontents

\section{Introduction}
The effect of fluid viscoelasticity on swimming microorganisms is a subject of great interest in the biofluids community. 
A major focus of studies on swimming in viscoelastic media is on the complexity of the impacts that viscoelasticity can have on swimming speeds, depending on the situation. 
In experimental settings and in computational models, viscoelastic effects have been shown to hinder swimming \cite{shen2011undulatory},
enhance swimming \cite{spagnolie2013locomotion, keim2012fluid, espinosa2013fluid, riley2014enhanced,riley2015small}, or both hinder and enhance swimming depending on factors such as the kinematics of the swimmer \cite{godinez2015complex, elfring2016effect, angeles2021front,thomases2014mechanisms,li2021microswimming}. Many studies emphasize the non-monotonic dependence of swimming speed on parameters relating to fluid viscoelasticity \cite{martinez2014flagellated,teran2010viscoelastic,liu2011force,thomases2017role,salazar2016numerical}.
Much of this prior work is either experimental, e.g. \cite{shen2011undulatory}, or based on computational models which couple an equation for curve evolution with bulk viscoelastic fluid equations (such as Oldroyd-B) via, e.g. the immersed boundary method \cite{li2017flagellar,li2019orientation,thomases2014mechanisms,thomases2017role}.\\

Here we present a simple model for the evolution of an immersed, inextensible curve which incorporates linear viscoelastic effects of the surrounding fluid in a closed-form system of equations along the curve only. 
The model is derived from the linear viscoelastic resistive force theory described in \cite{fu2008beating,thomases2017role}, but requires some additional interpretation to yield a well-posed curve evolution. The resulting fiber evolution equations look like classical resistive force theory plus Euler beam theory \cite{gray1955propulsion, camalet2000generic,hines1978bend,tornberg2004simulating}, but incorporate the evolution of an additional variable corresponding to a memory of the fiber curvature at previous times. This system of PDEs satisfies a very natural energy identity: in the absence of forcing, the bending energy of the filament plus an energy corresponding to the memory term are non-increasing in time. We prescribe a time-periodic forcing along the filament in the form of a preferred curvature and consider swimming as an emergent property. We show that this simple model, which is not coupled to any equations in the bulk, can capture the complexity of viscoelastic effects on swimming, including slowdowns and speedups relative to a Newtonian swimmer, depending on the form of the preferred curvature and the size of two viscoelastic parameters. \\

Although the viscoelastic effect is linear, the system of PDEs we obtain is nonlinear due to the fiber inextensibility constraint. For a planar filament, we prove well-posedness of the fiber evolution problem, including global existence and uniqueness for small data and local existence and uniqueness for large data. The model inherits many of the features of the Newtonian problem, studied in detail from a PDE perspective in our previous work \cite{RFTpaper}, but is fundamentally different in that the analysis now includes an additional ODE for the memory variable. As in \cite{RFTpaper}, we prove that given a (small) time-periodic forcing along the fiber in the form of a preferred curvature, there exists a unique periodic solution to the filament evolution equations. Furthermore, we show that the periodic solution to the viscoelastic PDE converges to the unique periodic solution in the Newtonian setting as a parameter corresponding to the fluid relaxation time goes to zero. These analysis questions are interesting in their own right and continue to develop the PDE theory of the hydrodynamics of slender filaments initiated in \cite{closed_loop,free_ends,rigid,inverse,regularized}. \\

Finally, we calculate an expression for the fiber swimming speed in terms of the prescribed preferred curvature. The expression depends in a complicated way on the viscoelastic parameters corresponding to the fluid relaxation time and the additional (polymer) viscosity of the fluid. Nevertheless, we are able to make a few predictions about the swimming speed, which we test via numerical simulations. The numerical method that we use is a natural extension of the method we proposed in \cite{RFTpaper}, which is based on a combination of the methods used in \cite{moreau2018asymptotic,maxian2021integral}. We show that varying two viscoelastic parameters corresponding to the fluid relaxation time and the additional polymeric stress of the fluid can have complex effects on the fiber swimming speed, including both speedups and slowdowns relative to the Newtonian setting. In addition, we numerically find a scenario in which viscoelastic effects may cause the swimmer to reverse direction, although the displacement for both the Newtonian and viscoelastic swimmers in this case is practically negligible. Our results are for a small set of parameter values and two choices of time-periodic preferred curvature, meaning that much remains to be explored. \\

We note that prior asymptotic calculations have indicated that linear viscoelasticity does not affect the swimming speed of a filament to leading order in small amplitude deformations, and nonlinear viscoelastic effects are needed to see changes in the swimming speed from a Newtonian swimmer \cite{fu2007theory,fu2009swimming,fulford1998swimming,lauga2007propulsion}. These results rely on an expression for the swimming speed in terms of the actual fiber shape rather than a given forcing. Here we show that prescribing the same preferred curvature along the filament results in differences in the emerging shape and hence in the swimming speed. We also note that our analysis accounts for effects of boundary conditions on finite fibers and applies to fibers with small curvatures rather than small amplitude deformations. \\

We further note that the numerical results reported here are qualitative rather than quantitative in the sense that the scaling of the model as presented here is not physically realistic. In particular, our choice of timescale for the fiber evolution removes the dependence of the dynamics on the swimmer slenderness and its bending stiffness. This makes the analysis and comparison to the Newtonian setting in \cite{RFTpaper} more convenient, but makes direct comparison with experiments or biological models less convenient. These physical considerations are expected to play a role: indeed, the important role of fiber flexibility on swimming has been emphasized in, e.g., \cite{salazar2016numerical,thomases2017role}. 
From an applications perspective, our main aim is to show that we can indeed get complex swimming behaviors from the simple model presented here. 
A more physical rescaling of this model may even amplify these differences.


\subsection{The model}
Let $\X:[0,1]\times[0,T]\to \R^3$ denote the centerline of an inextensible elastic filament. Throughout, we will used the notation $I=[0,1]$ to denote the unit interval, $s\in I$ to denote the arclength parameter along $\X$, and subscript $(\cdot)_s$ to denote $\frac{\p}{\p s}$. At each $s\in I$, the unit tangent vector to $\X$ is given by $\be_{\rm t}(s,t)=\X_s/\abs{\X_s}=\X_s$ due to inextensibility. Here we will consider a fiber undergoing planar deformations only; in particular, at each point $s\in I$ we may define an in-plane unit normal vector $\be_{\rm n}\perp \be_{\rm t}$ to the filament.\\

The motion of the filament is driven by a prescribed active forcing in the form of a preferred curvature $\kappa_0(s,t)$ along the fiber. Given $\kappa_0$, the filament evolves according to 
\begin{align}
\frac{\p\X}{\p t} &= -(1+\mu)({\bf I}+\gamma\X_s\X_s^{\rm T})\big(\X_{sss}-\tau\X_s-(\kappa_0)_s\be_{\rm n}- \frac{\mu}{1+\mu}\xi_s\be_{\rm n}\big)_s  \label{eq:VEoriginal1}\\
\delta\frac{\p \xi}{\p t} &= -\xi + \kappa-\kappa_0  \label{eq:VEoriginal2}\\
\abs{\X_s}^2 &= 1\,,  \label{eq:VEoriginal3}
\end{align}
with boundary conditions 
\begin{equation}\label{eq:BCs_original}
(\X_{ss}-\kappa_0\be_{\rm n})\big|_{s=0,1} =0\,, \quad (\X_{sss}-\tau\X_s-(\kappa_0)_s\be_{\rm n})\big|_{s=0,1} =0\,, \quad \xi\big|_{s=0,1}=\xi_s\big|_{s=0,1}=0\,.
\end{equation}

Here the matrix $({\bf I}+\gamma\X_s\X_s^{\rm T})$ in \eqref{eq:VEoriginal1} is the resistive force theory approximation relating the hydrodynamic force along a slender filament in a Stokes (Newtonian) fluid to its velocity \cite{gray1955propulsion, pironneau1974optimal,lauga2020fluid}. The parameter $\gamma$ is a shape factor which depends on the aspect ratio of the filament; for a very slender filament, $\gamma\approx 1$.\\

The first three components $\big(\X_{sss}-\tau\X_s-(\kappa_0)_s\be_{\rm n}\big)_s$ of the forcing term in \eqref{eq:VEoriginal1} are identical to the Newtonian setting. The term $\X_{ssss}$ is the elastic response of the filament to deformations and may be rewritten as $(\kappa_s\be_{\rm n}-\kappa^2\be_{\rm t})_s$, where $\kappa(s,t)$ is the filament curvature. The function $\tau(s,t)$ plays the role of the (unknown) filament tension and enforces the inextensibility constraint \eqref{eq:VEoriginal3}. As noted, $\kappa_0$ is the prescribed preferred curvature of the filament and serves as an active forcing along the fiber.\\

The leading order effects of the filament response to a viscoelastic fluid are encoded in the variable $\xi$. As seen from equation \eqref{eq:VEoriginal1}, $\xi$ modifies the evolution of $\X$ in the same way as the preferred curvature $\kappa_0$, and, from equation \eqref{eq:VEoriginal2}, may be interpreted as the `memory' of the curvature difference $\kappa-\kappa_0$ at previous times. 
The parameters $\delta\ge0$ and $\mu\ge 0$ are associated with viscoelasticity: $\delta$ is the additional relaxation time of the filament due to viscoelastic  effects of the fluid, and $\mu$ relates to the additional (polymer) viscosity of the viscoelastic medium. 
Note that if $\mu=0$ or if $\delta=0$, we recover the classical Newtonian formulation (see \cite{camalet2000generic,hines1978bend, wiggins1998flexive, wiggins1998trapping,tornberg2004simulating,RFTpaper})
\begin{equation}\label{eq:newt}
\frac{\p\X}{\p t} = -({\bf I}+\gamma\X_s\X_s^{\rm T})\big(\X_{sss}-\tau\X_s-(\kappa_0)_s\be_{\rm n}\big)_s
\end{equation}
(up to a rescaling of the unknown tension in the $\delta=0$ case). However, showing convergence of solutions of \eqref{eq:VEoriginal1}-\eqref{eq:VEoriginal2} to solutions of \eqref{eq:newt} as $\delta\to 0$ is more subtle than simply verifying that $\delta=0$ yields the Newtonian formulation, since $\delta>0$ is a singular perturbation. We address the $\delta\to 0$ limit in the time-periodic setting in Theorem \ref{thm:per_deb}.  \\

The model \eqref{eq:VEoriginal1}-\eqref{eq:BCs_original} is inspired by the following (perhaps more familiar) framing of linear viscoelastic effects on filament evolution. We consider
\begin{align}
\frac{\p\X}{\p t} &= -({\bf I}+\gamma\X_s\X_s^{\rm T})(\bm{\sigma}^{\rm vis})_s \label{eq:RFT} \\
(1+\mu)\delta\frac{\p\bm{\sigma}^{\rm ve}}{\p t} + \bm{\sigma}^{\rm ve} &= \delta\frac{\p\bm{\sigma}^{\rm vis}}{\p t} + \bm{\sigma}^{\rm vis}  \label{eq:oldroyd}\\ 
(\bm{\sigma}^{\rm ve})_s &= \big(\X_{sss}-\tau\X_s -(\kappa_0)_s\be_{\rm n}\big)_s  \label{eq:balance}\\
\abs{\X_s}^2&=1  \label{eq:inext}\,.
\end{align}

Here $\bm{\sigma}^{\rm ve}(s,t)$ and $\bm{\sigma}^{\rm vis}(s,t)$ are both vectors along the filament $\X$. Equation \eqref{eq:RFT} relates the filament velocity $\frac{\p\X}{\p t}$ to the viscous drag $\bm{f}^{\rm vis}=(\bm{\sigma}^{\rm vis})_s$ along the fiber via resistive force theory. Equation \eqref{eq:oldroyd} has the form of a linearized Oldroyd-B model \cite{fu2008beating,thomases2017role}, which relates the viscoelastic stresses in the fluid to the viscous strain rate. 
Equation \eqref{eq:oldroyd} is restricted to the filament only; $\bm{\sigma}^{\rm ve}(s,t)$ and $\bm{\sigma}^{\rm vis}(s,t)$ are both vectors along the fiber. When $s$ and $t$ derivatives commute (i.e. for a straight filament), equation \eqref{eq:oldroyd} agrees with the linear viscoelastic resistive force theory derived in \cite{fu2008beating,thomases2017role} in terms of $\bm{f}^{\rm ve}$ and $\bm{f}^{\rm vis}$ instead. 
Here again $\delta$ is the fluid relaxation time and $1+\mu$ is the total viscosity of the medium. Note that rescaling $\bm{\sigma}^{\rm vis}$ by $\frac{1}{1+\mu}$ and $\delta$ by $1+\mu$, we may rewrite \eqref{eq:oldroyd} in the (perhaps more usual) form
$\delta\frac{\p\bm{\sigma}^{\rm ve}}{\p t} + \bm{\sigma}^{\rm ve} = \delta\frac{\p\bm{\sigma}^{\rm vis}}{\p t} + (1+\mu)\bm{\sigma}^{\rm vis}$, but we will use the form \eqref{eq:oldroyd} for analysis.  \\

Equation \eqref{eq:balance} is the force balance between the viscoelastic forces in the fluid $\bm{f}^{\rm ve}=(\bm{\sigma}^{\rm ve})_s$ and the elastic forces $\big(\X_{sss}-\tau\X_s -(\kappa_0)_s\be_{\rm n}\big)_s$ along the rod, which are subject to the inextensibility constraint \eqref{eq:inext}. See \cite{camalet2000generic,thomases2017role,RFTpaper} for a variational derivation of the elastic forces along the fiber; note that the boundary conditions \eqref{eq:BCs_original} come from this variational derivation. \\
 
Let $\sigma^{\rm vis}_{\rm n}=\bm{\sigma}^{\rm vis}\cdot\be_{\rm n}$, $\sigma^{\rm vis}_{\rm t}=\bm{\sigma}^{\rm vis}\cdot\be_{\rm t}$, and likewise for $\sigma^{\rm ve}_{\rm n}$, $\sigma^{\rm ve}_{\rm t}$. 
Now, because the filament is inextensible, the stresses in the tangential direction along the fiber are unknown and are grouped into the filament tension (see \eqref{eq:balance}). Up to a redefinition of the (unknown) filament tension, we will use the same form of stress in the tangential direction as in the Newtonian setting, i.e. $\sigma^{\rm vis}_{\rm t}=\tau\X_s$ for some unknown function $\tau$.
We will thus consider equation \eqref{eq:oldroyd} as an equation holding along the normal direction of the fiber only:
\begin{align*}
(1+\mu)\delta\frac{\p\sigma^{\rm ve}_{\rm n}}{\p t} + \sigma^{\rm ve}_{\rm n} &= \delta\frac{\p\sigma^{\rm vis}_{\rm n}}{\p t} + \sigma^{\rm vis}_{\rm n}\,.
\end{align*}
We can then solve for $\sigma^{\rm vis}_{\rm n}$ in terms of $\sigma^{\rm ve}_{\rm n}$:
\begin{align*}
\sigma^{\rm vis}_{\rm n} &= (1+\mu)\sigma^{\rm ve}_{\rm n} - \mu\delta^{-1}\int_0^t e^{-(t-t')/\delta}\sigma^{\rm ve}_{\rm n}\big|_{t=t'}\, dt' + e^{-t/\delta}\big(\sigma^{\rm vis,in}_{\rm n}-(1+\mu)\sigma^{\rm ve,in}_{\rm n}\big) \\
&= (1+\mu)\sigma^{\rm ve}_{\rm n} - \mu\delta^{-1}\int_0^t e^{-(t-t')/\delta}(\kappa-\kappa_0)_s\, dt' + e^{-t/\delta}\big(\sigma^{\rm vis,in}_{\rm n}-(1+\mu)\sigma^{\rm ve,in}_{\rm n}\big)\,,
\end{align*}
where $\sigma^{\rm vis,in}_{\rm n}=\sigma^{\rm vis}_{\rm n}\big|_{t=0}$, $\sigma^{\rm ve,in}_{\rm n}=\sigma^{\rm ve}_{\rm n}\big|_{t=0}$.
Taking $\xi=\delta^{-1}\int_0^t e^{-(t-t')/\delta}(\kappa-\kappa_0)\, dt' + \mu^{-1}e^{-t/\delta}\int_0^s\big(\sigma^{\rm vis,in}_{\rm n}-(1+\mu)\sigma^{\rm ve,in}_{\rm n}\big)ds'$, we have that $\sigma^{\rm vis}_{\rm n}= (1+\mu)\sigma^{\rm ve}_{\rm n} - \mu\xi_s$ and $\delta\dot\xi=-\xi+\kappa-\kappa_0$, yielding the system \eqref{eq:VEoriginal1}-\eqref{eq:BCs_original}.\\

The model \eqref{eq:VEoriginal1}-\eqref{eq:BCs_original} has advantages over other potential ways of incorporating linear viscoelastic effects of the surrounding fluid due to both its simplicity and because it has an associated energy. In particular, taking $\kappa_0=0$ (no internal forcing), equations \eqref{eq:VEoriginal1} and \eqref{eq:VEoriginal2} reduce to
\begin{align*}
\frac{\p\X}{\p t} &= -(1+\mu)({\bf I}+\gamma\X_s\X_s^{\rm T})\big(\X_{sss}-\tau\X_s- \frac{\mu}{1+\mu}\xi_s\be_{\rm n}\big)_s \\
\delta\frac{\p \xi}{\p t} &= -\xi + \kappa\,.
\end{align*}
Multiplying both sides of the first equation by $(\X_{sss}-\tau\X_s- \frac{\mu}{1+\mu}\xi_s\be_{\rm n})_s$ and integrating in $s$, on the right hand side we obtain the negative quantity 
\begin{align*}
-R^2(t)&:=-(1+\mu)\int_0^1\bigg(\big|\big(\X_{sss}-\tau\X_s- \frac{\mu}{1+\mu}\xi_s\be_{\rm n}\big)_s\big|^2 +\gamma \bigg(\X_s\cdot\big(\X_{sss}-\tau\X_s- \frac{\mu}{1+\mu}\xi_s\be_{\rm n}\big)_s\bigg)^2\bigg)\,ds.
\end{align*}

On the left hand side, using that $\X_s=\be_{\rm t}$ and $\X_{ss}=\kappa\be_{\rm n}$, we have
\begin{align*}
\int_0^1 \frac{\p\X}{\p t}\cdot(\X_{sss}-\tau\X_s- \frac{\mu}{1+\mu}\xi_s\be_{\rm n})_s \, ds &= -\int_0^1\frac{\p\X_s}{\p t}\cdot\big(\X_{sss}-\tau\X_s- \frac{\mu}{1+\mu}\xi_s\be_{\rm n}\big) \, ds \\
&= \int_0^1\dot\kappa\big(\kappa- \frac{\mu}{1+\mu}\xi\big) \, ds\,,
\end{align*}
where we use $\dot\kappa$ to denote $\frac{\p\kappa}{\p t}$. Now, $\dot\kappa \xi$ may be rewritten as
\begin{align*}
\dot\kappa \xi &= \p_t(\kappa \xi)-\kappa\dot\xi = \p_t(\kappa\xi)-\dot\xi(\xi+\delta\dot \xi) = \p_t(\kappa\xi)-\frac{1}{2}\p_t(\xi^2) -\delta\dot\xi^2\,,
\end{align*}
so the left hand side becomes
\begin{align*}
\int_0^1\dot\kappa\big(\kappa- \frac{\mu}{1+\mu}\xi\big) \, ds &= \int_0^1\bigg(\frac{1}{2}\p_t(\kappa^2)- \frac{\mu}{1+\mu}\big(\p_t(\kappa \xi)-\frac{1}{2}\p_t(\xi^2) -\delta \dot \xi^2 \big)\bigg) \, ds\,.
\end{align*}

We thus obtain the energy equality
\begin{equation}\label{eq:energy}
\frac{1}{2}\p_t \int_0^1\bigg(\kappa^2+\mu(\kappa-\xi)^2 \bigg) \, ds  
= -\delta\mu\int_0^1\dot \xi^2\,ds - (1+\mu)R^2(t) \,;
\end{equation}
in particular, $\norm{\kappa}_{L^2}^2+\mu\norm{\kappa-\xi}_{L^2}^2$ is a monotone quantity. Notice that $\kappa-\xi=\delta\dot\xi$, so this quantity may be rewritten as $\norm{\kappa}_{L^2}^2+\mu\delta\|\dot\xi\|_{L^2}^2$.

\subsection{Analytical setup}
Rather than working directly with the formulation \eqref{eq:VEoriginal1}-\eqref{eq:BCs_original}, we will use the inextensibility of the filament and the planarity of its deformation to write the tangent vector along the filament as
\begin{equation}\label{eq:tangent}
\X_s = \be_{\rm t}= \begin{pmatrix}
\cos\theta\\
\sin\theta
\end{pmatrix} \,, 
\end{equation}
where $\theta(s,t)$ is the angle between $\be_{\rm t}(s,t)$ and $\be_{\rm t}(0,0)$. Differentiating \eqref{eq:VEoriginal1} in $s$, we may then obtain a system of three equations: two evolution equations for $\theta$ and $\xi$, and one elliptic equation for the tension $\tau$, given by 
\begin{equation}\label{eq:theta_eqn}
\begin{aligned}
\dot\theta &=(1+\mu)\bigg(-\theta_{ssss}+(2+\gamma)(\theta_s^3)_s +(2+\gamma)\tau_s\theta_s+\tau\theta_{ss} + (\kappa_0)_{sss}\\
&\qquad -(1+\gamma)\theta_s^2(\kappa_0)_s \bigg) + \mu\bigg( \xi_{sss}-(1+\gamma)\theta_s^2\xi_s\bigg)   \\
\delta\dot\xi &= -\xi+\theta_s-\kappa_0\\
(1+\gamma)\tau_{ss} &=(\theta_s)^2\tau+ (\theta_s)^4+\theta_{ss}^2-(4+3\gamma)(\theta_{ss}\theta_s)_s + (2+\gamma)(\kappa_0)_{ss}\theta_s +(1+\gamma)\theta_{ss}(\kappa_0)_s \\
&\qquad +\frac{\mu}{1+\mu}\bigg( (2+\gamma)(\theta_s\xi_s)_s - \theta_{ss}\xi_s \bigg) \\
 (\theta_s-\kappa_0)\big|_{s=0,1}&=0\,,\quad (\theta_{ss}-(\kappa_0)_s)\big|_{s=0,1}=0\,, \quad (\tau+\kappa_0^2)\big|_{s=0,1}=0\, \quad \xi\big|_{s=0,1}=\xi_s\big|_{s=0,1}=0\,.
\end{aligned}
\end{equation}
However, it will be more useful to consider the filament evolution in term of the fiber curvature $\kappa=\theta_s$ rather than $\theta$, as is done in \cite{goldstein1995nonlinear,thomases2017role,RFTpaper}. Furthermore, due to the boundary conditions in \eqref{eq:theta_eqn}, it will be most useful to recast the system \eqref{eq:theta_eqn} in terms of $\overline\kappa=\kappa-\kappa_0$ and $\overline\tau=\tau+\kappa_0^2$. The equations \eqref{eq:theta_eqn} may be written as 
\begin{align}
\dot{\overline\kappa} &= -(1+\mu)\overline\kappa_{ssss} +\mu\xi_{ssss}-\dot\kappa_0 + (1+\mu)\big(\mc{N}[\overline\kappa,\kappa_0] \big)_s  - \mu(1+\gamma)\big((\overline\kappa+\kappa_0)^2\xi_s \big)_s \label{eq:kappadot}\\
\delta\dot\xi &= \overline\kappa-\xi \label{eq:xidot}\\
(1+\gamma)\overline\tau_{ss} &=(\overline\kappa+\kappa_0)^2\overline\tau+\mc{T}[\overline\kappa,\kappa_0] +\frac{\mu}{1+\mu}\bigg( (2+\gamma)((\overline\kappa+\kappa_0)\xi_s)_s - (\overline\kappa+\kappa_0)_s\xi_s \bigg) \label{eq:taueq}\\
\overline\kappa\big|_{s=0,1} &=\overline\kappa_s\big|_{s=0,1}=0\,, \quad \xi\big|_{s=0,1}=\xi_s\big|_{s=0,1} =0\,, \quad \overline\tau\big|_{s=0,1} =0\,. \label{eq:BCs}
\end{align}
This may be compared with the curvature formulation in the Newtonian case \cite{RFTpaper}, where the evolution is given by
\begin{equation}\label{eq:NewtonianKappa}
\begin{aligned}
\dot{\overline\kappa}^{\rm nw} &= -\overline\kappa^{\rm nw}_{ssss} -\dot\kappa_0 + \big(\mc{N}[\overline\kappa^{\rm nw},\kappa_0] \big)_s  \\
(1+\gamma)\overline\tau^{\rm nw}_{ss} &=(\overline\kappa^{\rm nw}+\kappa_0)^2\overline\tau^{\rm nw}+\mc{T}[\overline\kappa^{\rm nw},\kappa_0] \,.
\end{aligned}
\end{equation}

Here we use the superscript $(\cdot)^{\rm nw}$ to distinguish the solution to the Newtonian PDE \eqref{eq:NewtonianKappa} from the viscoelastic $\overline\kappa$.
The nonlinear terms $\mc{N}$ and $\mc{T}$ have the same form in both the viscoelastic and Newtonian cases and are given by
\begin{equation}\label{eq:NandT}
\begin{aligned}
\mc{N}[\overline\kappa,\kappa_0] &:= 3(2+\gamma)\overline\kappa(\overline\kappa+2\kappa_0)\overline\kappa_s+(5+3\gamma)\kappa_0^2\overline\kappa_s+ (5+2\gamma)\overline\kappa^2(\kappa_0)_s \nonumber\\
&\qquad +2(3+\gamma)\overline\kappa\kappa_0(\kappa_0)_s +(2+\gamma)\overline\tau_s(\overline\kappa+\kappa_0)+\overline\tau(\overline\kappa+\kappa_0)_s \\
\mc{T}[\overline\kappa,\kappa_0]&:= \overline\kappa(\overline\kappa+\kappa_0)^2(\overline\kappa+2\kappa_0) + (\overline\kappa+\kappa_0)_s\overline\kappa_s \nonumber\\
&\qquad-(1+\gamma)\big(\overline\kappa(\overline\kappa+2\kappa_0)\big)_{ss} -(2+\gamma)\big(\overline\kappa_s(\overline\kappa+\kappa_0)\big)_s\,.
\end{aligned}
\end{equation}
Note that $\overline\tau$ appears in $\mc{N}$, but since $\overline\tau=\overline\tau(\overline\kappa,\kappa_0)$, we will not denote this $\overline\tau$ dependence in our notation. The formulation \eqref{eq:kappadot}-\eqref{eq:BCs} will serve as the basis for our analysis. \\

Given $\theta\big|_{t=0}$ and $(\overline\kappa,\xi)$ solving \eqref{eq:NandT}, we may recover $\theta(s,t)$ via
\begin{equation}\label{eq:theta_recover}
\dot\theta = -(1+\mu)\overline\kappa_{sss}-\dot\kappa_0 + (1+\mu)\mc{N}[\overline\kappa,\kappa_0] +\mu\xi_{sss}- \mu(1+\gamma)(\overline\kappa+\kappa_0)^2\xi_s\,, 
\end{equation}
and the evolution of the fiber frame $(\be_{\rm t},\be_{\rm n})$ via 
\begin{equation}\label{eq:frame_ev}
\dot\be_{\rm t}(s,t) = \dot\theta(s,t) \be_{\rm n}(s,t)\,, \qquad \dot\be_{\rm n}(s,t)=-\dot\theta(s,t) \be_{\rm t}(s,t)\,.
\end{equation}
Using \eqref{eq:frame_ev}, we may then obtain the full fiber evolution by 
\begin{equation}\label{eq:fiber_move}
\begin{aligned}
\frac{\p\X}{\p t}(s,t) &= \bigg(-(1+\mu)\overline\kappa_{ss}+\mu\xi_{ss}+(1+\mu)(\overline\kappa+\kappa_0)(\overline\tau-\overline\kappa(\overline\kappa+2\kappa_0)) \bigg)\be_{\rm n} \\
&\qquad +(1+\gamma)\bigg((\overline\kappa+\kappa_0)\big((1+\mu)\overline\kappa_s-\mu\xi_s \big)+(1+\mu)\big(\overline\tau_s-(\overline\kappa(\overline\kappa+2\kappa_0))_s \big)\bigg)\be_{\rm t}\,.
\end{aligned}
\end{equation}

To analyze the system \eqref{eq:kappadot}-\eqref{eq:BCs}, as in \cite{RFTpaper}, we will begin by defining the linear operator $\mc{L}$ by
\begin{equation}\label{eq:L_op}
 \mc{L}[\psi]:= \p_{ssss}\psi\,, \qquad \psi(0)=\psi(1)=0\,, \; \psi_s(0)=\psi_s(1)=0\,.
\end{equation}
From \cite{landau1986theory, wiggins1998flexive, wiggins1998trapping},
we have that the eigenfunctions $\psi_k$ and eigenvalues $\lambda_k$ of the operator $\mc{L}$ \eqref{eq:L_op} are given by 
\begin{equation}\label{eq:L_eigs}
\begin{aligned}
\psi_k(s) &= \frac{\wh\psi_k(s)}{\|\wh\psi_k\|_{L^2}}\,, \qquad \lambda_k = \alpha_k^4\,, \qquad k=1,2,\dots\\
&\text{where } \cos(\alpha_k)\cosh(\alpha_k)=1\,, \quad \alpha_0=0 \\
&\text{and } \wh\psi_k(s) = \left(\cos(\alpha_k)-\cosh(\alpha_k)\right) \left(\cos(\alpha_k s)- \cosh(\alpha_k s) \right) \\
&\hspace{2cm} + \left(\sin(\alpha_k) +\sinh(\alpha_k)\right) \left(\sin(\alpha_k s)-\sinh(\alpha_k s)\right)\, .
\end{aligned}
\end{equation}
We note that $\alpha_k\to\frac{(2k+1)\pi}{2}$ as $k\to\infty$, and that the smallest eigenvalue of $\mc{L}$ is given by $\lambda_1\approx(4.73)^4\approx 500$. We further note that $\psi_k(s)$ is even about $s=\frac{1}{2}$ for odd $k$, and odd about $s=\frac{1}{2}$ for even $k$. \\

We may consider the expansion of any $u\in L^2(I)$ in eigenfunctions of $\mc{L}$:
\begin{equation}\label{eq:wt_expand}
u = \sum_{k=1}^\infty \wt{u}_k\psi_k\,, \quad \wt{u}_k = \int_0^1u(s)\psi_k(s)ds\,.
\end{equation}
The domain of $\mc{L}^r$, $0\le r\le1$, may then be defined by  
\begin{equation}\label{eq:domainLr}
D(\mc{L}^r) = \bigg\{ u\in L^2(I) \;:\; \sum_{k=1}^\infty\lambda_k^{2r}\wt{u}_k^2 <\infty \bigg\}\,.
\end{equation}
Note that $D(\mc{L}^r)\subseteq H^{4r}(I)$ for $0\le r \le 1$, and $D(\mc{L}^0)= L^2(I)$.\\


With $\mc{L}$ as defined in \eqref{eq:L_op}, we may define a mild solution $(\overline\kappa,\xi)$ to the system \eqref{eq:kappadot}-\eqref{eq:BCs} by the Duhamel formula
\begin{equation}\label{eq:duhamel}
\begin{aligned}
\begin{pmatrix}
\overline\kappa \\
\xi 
\end{pmatrix} &= e^{\mc{A}t}\begin{pmatrix}
\overline\kappa^{\rm in} \\
\xi^{\rm in} 
\end{pmatrix} - \int_0^te^{\mc{A}(t-t')}\begin{pmatrix}
 \dot\kappa_0 \\
0
\end{pmatrix}\, dt'+ (1+\mu)\int_0^te^{\mc{A}(t-t')}\begin{pmatrix}
 \big(\mc{N}[\overline\kappa,\kappa_0] \big)_s \\
0
\end{pmatrix}\, dt' \\
&\hspace{3cm}
 - \mu(1+\gamma)\int_0^te^{\mc{A}(t-t')}\begin{pmatrix}
\big((\overline\kappa+\kappa_0)^2\xi_s \big)_s \\
0
\end{pmatrix}\, dt'\,,
\end{aligned}
\end{equation}

where $\mc{A}$ denotes the operator
\begin{equation}\label{eq:Aop}
\mc{A} = \begin{pmatrix}
-(1+\mu)\mc{L} & \mu\mc{L} \\
\delta^{-1} & -\delta^{-1}
\end{pmatrix}\,.
\end{equation}

\subsection{Statement of results}\label{subsec:statements}
Our first result is well-posedness for the system \eqref{eq:kappadot}-\eqref{eq:BCs} in the case of small $\kappa_0$ and either short time or small initial data $\overline\kappa^{\rm in}$. We note that the viscoelastic system inherits all of the subtleties of the Newtonian case \eqref{eq:NewtonianKappa} that make global well-posedness difficult for large initial data. In particular, the behavior of the filament tension $\overline\tau$, particularly its dependence on powers of $\overline\kappa_s$, limits what we can show in terms of well-posedness. See \cite{RFTpaper} for a deeper discussion of these issues in the Newtonian setting. \\

Here and throughout, we use the notation
\begin{align*}
\norm{\begin{pmatrix}
u\\
\phi
\end{pmatrix}}_{\dot H^m\times \dot H^m} :=\norm{u}_{\dot H^m}+\norm{\phi}_{\dot H^m} \,.
\end{align*}
The well-posedness results for the system \eqref{eq:kappadot}-\eqref{eq:BCs} may be stated as follows.

\begin{theorem}[Well-posedness]\label{thm:wellposed}
There exist constants $\varepsilon>0,\varepsilon_1\ge0,\varepsilon_2\ge0$ such that, given $\kappa_0\in C^1([0,T];H^1(I))$ satisfying
\begin{align*}
\sup_{t\in[0,T]}\norm{\kappa_0}_{H^1(I)}=\varepsilon_1\le \varepsilon\,, \quad \sup_{t\in[0,T]}\norm{\dot\kappa_0}_{L^2(I)}=\varepsilon_2\le \varepsilon\,,
\end{align*}
there exist
\begin{enumerate}
\item A time $T_\varepsilon>0$ depending on $(\overline\kappa^{\rm in},\xi^{\rm in})$ such that the system \eqref{eq:kappadot}--\eqref{eq:BCs} admits a unique mild solution $(\overline\kappa,\xi)\in C([0,T_\varepsilon];L^2(I)\times L^2(I))\cap C((0,T_\varepsilon];\dot H^1(I)\times \dot H^1(I))$. \\

\item A constant $\varepsilon_3>0$ such that if 
\begin{align*}
\norm{\begin{pmatrix}
\overline\kappa^{\rm in} \\
\xi^{\rm in}
\end{pmatrix}}_{L^2(I)\times L^2(I)} = \varepsilon_3 \le \varepsilon\,,
\end{align*}
then, for any $T>0$, the system \eqref{eq:kappadot}--\eqref{eq:BCs} admits a unique mild solution $(\overline\kappa,\xi)\in C([0,T]; L^2(I)\times L^2(I))\cap C((0,T];\dot H^1(I)\times \dot H^1(I))$ satisfying
\begin{equation}\label{est:withforcing}
\sup_{t\in[0,T]}\bigg(\norm{\begin{pmatrix}
\overline\kappa \\
\xi
\end{pmatrix}}_{L^2\times L^2} + \min\{t^{1/4},1\}\norm{\begin{pmatrix}
\overline\kappa \\
\xi
\end{pmatrix}}_{\dot H^1\times\dot H^1} \bigg)
 \le c\,(\varepsilon_1+\varepsilon_2+\varepsilon_3)\,.
\end{equation}
\end{enumerate}

In case (2), in the absence of an internal forcing ($\kappa_0\equiv0$), we may obtain the bound 
\begin{equation}\label{est:noforcing}
\norm{\begin{pmatrix}
\kappa \\
\xi
\end{pmatrix}}_{L^2(I)\times L^2(I)} + \min\{t^{1/4},1\}\norm{\begin{pmatrix}
\kappa \\
\xi
\end{pmatrix}}_{\dot H^1(I)\times\dot H^1(I)} \le c\,e^{-t\Lambda} \norm{\begin{pmatrix}
\kappa^{\rm in}\\
\xi^{\rm in}
\end{pmatrix}}_{L^2(I)}\,,
\end{equation}
where $\Lambda=\min\{\lambda_1,\frac{1}{\delta(1+\mu)}\}$ for $\lambda_1$ as in \eqref{eq:L_eigs}. 
\end{theorem}
As in the Newtonian case, in the absence of internal fiber forcing, the straight filament is nonlinearly stable to small perturbations. 
However, the fiber relaxation time $\Lambda$ in \eqref{est:noforcing} depends on the viscoelastic parameters $\delta$ and $\mu$. In particular, if either $\delta$ or $\mu$ is large (even if, e.g., $\delta=\mu=1$), the decay rate of perturbations to the straight filament is much slower than in the Newtonian setting (where $\Lambda=\lambda_1\approx 500$).
Finally, similar to the Newtonian setting, the quantity $\norm{(\overline\kappa^{\rm in},\xi^{\rm in})}_{L^2(I)\times L^2(I)}$ in case (2) corresponds to the initial viscoelastic bending energy \eqref{eq:energy} of the filament; in particular, a small initial energy leads to global existence. \\

For applications to undulatory swimming, we are most interested in prescribing a time-periodic preferred curvature $\kappa_0$ and understanding properties of the resulting time-periodic solution. Given a $T$-periodic $\kappa_0$, we prove the existence of a unique $T$-periodic solution $(\overline\kappa,\xi)$ to \eqref{eq:kappadot}-\eqref{eq:BCs}. Moreover, we show that since the prescribed $\kappa_0$ is small, the unique periodic solution $(\overline\kappa,\xi)$ is close to the solution to the linearized version of \eqref{eq:kappadot}-\eqref{eq:BCs}. This will be useful for computing an expression for the fiber swimming speed. Finally, we show that as the viscoelastic relaxation time $\delta\to 0$, the unique periodic solution converges to the unique periodic solution of the Newtonian PDE \eqref{eq:NewtonianKappa}, studied in detail in \cite{RFTpaper}. Recall that $\delta>0$ is a singular perturbation of the $\delta=0$ case (see \eqref{eq:xidot}).
\begin{theorem}[Periodic solutions and properties]\label{thm:per_deb}
There exists a constant $\varepsilon>0$ such that, given a $T$-periodic $\kappa_0\in C^1([0,T];H^1(I))$ satisfying
\begin{equation}\label{est:kapp0bds}
 \sup_{t\in[0,T]} \norm{\dot\kappa_0}_{L^2(I)} = \varepsilon_1 \le \varepsilon\,, \qquad \sup_{t\in[0,T]} \norm{\kappa_0}_{H^1(I)} = \varepsilon_2 \le \varepsilon\,,
 \end{equation} 
 \begin{enumerate}
  \item[(a)]
 There exists a unique $T$-periodic solution $(\overline\kappa,\xi)$ to the system \eqref{eq:kappadot}-\eqref{eq:BCs} satisfying 
 \begin{equation}\label{est:periodic}
  \sup_{t\in[0,T]} \norm{\begin{pmatrix}
  \overline\kappa\\
  \xi
  \end{pmatrix}}_{H^1(I)\times H^1(I)} \le c\,\big(\varepsilon_1 + \varepsilon_2\big)\,.
 \end{equation}
\item[(b)] Defining $(\overline\kappa^{\rm lin},\xi^{\rm lin})$ to be the unique periodic solution to the linear PDE
\begin{equation}\label{eq:lin_def}
\begin{aligned}
\dot{\overline\kappa}^{\rm lin}&= -(1+\mu)\overline\kappa^{\rm lin}_{ssss}+\mu\xi^{\rm lin}_{ssss} - \dot\kappa_0 \\
\delta\dot\xi^{\rm lin} &= -\xi^{\rm lin}+\overline\kappa^{\rm lin} \\
\overline\kappa^{\rm lin}\big|_{s=0,1}&= \xi^{\rm lin}\big|_{s=0,1}=0\,,\;  \overline\kappa^{\rm lin}_s\big|_{s=0,1}= \xi^{\rm lin}_s\big|_{s=0,1}=0\,, 
\end{aligned}
\end{equation}
the periodic solution $(\overline\kappa,\xi)$ of \eqref{est:periodic} satisfies 
\begin{equation}\label{est:lin_per_bd}
\sup_{t\in [0,T]}\norm{\begin{pmatrix}
\overline \kappa- \overline\kappa^{\rm lin}\\
\xi - \xi^{\rm lin}
\end{pmatrix}}_{H^1(I)\times H^1(I)} \le c\,\varepsilon^3\,.
\end{equation} 
\item[(c)] In the limit $\delta\to 0$, the periodic solution $(\overline\kappa,\xi)$ satisfies 
\begin{equation}\label{est:small_delta}
\norm{\overline\kappa-\xi}_{H^1(I)} \le c\,\delta^{1/2}\big(\sup_{t\in[0,T]}\norm{\dot\kappa_0}_{L^2} +\sup_{t\in[0,T]}\norm{\kappa_0}_{H^1} \big)\,.
\end{equation}
In particular, as $\delta\to 0$, $(\overline\kappa,\xi)$ converges to the solution to the Newtonian PDE \eqref{eq:NewtonianKappa} with the same forcing $\kappa_0$. 
\end{enumerate}
\end{theorem}
The proof of Theorem \ref{thm:per_deb} appears in two parts: parts (a) and (b) are shown in section \ref{subsec:periodic}, and part (c) is shown in section \ref{subsec:deborah}. \\

Given a $T$-periodic $\kappa_0$ and the corresponding $T$-periodic solution guaranteed by Theorem \ref{thm:per_deb}, we may now study the actual swimming speed of the filament. We first calculate an expression for the swimming velocity $\bm{V}(t)=\int_0^1\frac{\p\X}{\p t}\,ds$. The expression involves $\kappa_0$, $\overline\kappa$, and $\xi$. To better understand the swimming speed of the filament, especially in relation to the Newtonian setting, we need an expression in terms of $\kappa_0$ only. We use the closeness of $(\overline\kappa,\xi)$ to $(\overline\kappa^{\rm lin},\xi^{\rm lin})$ from Theorem \ref{thm:per_deb} (b) to obtain an expression for the average filament swimming speed at leading order in terms of $\kappa_0$ only. Here and throughout, for $u=u(t)$, we will denote the time average over one period by 
\begin{equation}\label{eq:langlerangle}
\langle u \rangle := \frac{1}{T}\int_0^T u \,dt\,.
\end{equation}

The fiber swimming velocity $\bm{V}(t)$ and average speed in direction $\be_{\rm t}(0,0)$ are given as follows.
\begin{theorem}\label{thm:VEswimming}
For $\varepsilon$ as in Theorem \ref{thm:per_deb}, given a $T$-periodic $\kappa_0\in C^1([0,T];H^3(I))$ satisfying
\begin{equation}\label{est:kapp0bds2}
 \sup_{t\in[0,T]} \norm{\dot\kappa_0}_{L^2} = \varepsilon_1 \le \varepsilon\,, \qquad \sup_{t\in[0,T]} \norm{\kappa_0}_{H^1} = \varepsilon_2 \le \varepsilon\,,
 \end{equation} 
a filament satisfying equations \eqref{eq:kappadot}-\eqref{eq:BCs} swims with velocity
\begin{equation}\label{est:swimspeed1}
\bm{V}(t)= U(t)\be_{\rm t}(0,0) + \bm{r}_{\rm v}(t)\,,
\end{equation}
where $\sup_{t\in[0,T]}\abs{U}\le c\varepsilon^2$, $\sup_{t\in[0,T]}\abs{\bm{r}_{\rm v}}\le c\varepsilon^3$, and
\begin{equation}\label{est:Ut}
U(t) =  -\gamma\int_0^1 (\kappa_0)_s\overline\kappa\, ds -\gamma\mu\int_0^1(\overline\kappa-\xi)(\overline\kappa+\kappa_0)_s \, ds\,.
\end{equation}

Moreover, expanding $\kappa_0$ as $\kappa_0(s,t)=\sum_{m,k=1}^\infty\big(a_{m,k}\cos(\omega m t)-b_{m,k}\sin(\omega m t)\big)\psi_k(s)$, where $\psi_k$ are the eigenfunctions \eqref{eq:L_eigs} of the operator $\mc{L}$ and $\omega=\frac{2\pi}{T}$, the average swimming speed $\langle U \rangle$ over the course of one time period is given by 
\begin{equation}\label{est:swimspeed2}
\begin{aligned}
\langle U \rangle &=\frac{\gamma}{2}\sum_{m,k,\ell=1}^\infty\bigg( W_{1,m\ell k}\,(a_{m,k}b_{m,\ell}-a_{m,\ell}b_{m,k}) \\
&\hspace{2cm} + W_{2,m\ell k}\,(a_{m,k}a_{m,\ell}+b_{m,k}b_{m,\ell})\bigg)
%
 \int_0^1\psi_k(\psi_\ell)_s\,ds\, + r_{\rm u}\,.
\end{aligned}
\end{equation}
Here coefficients $W_{j,m\ell k}$ are explicitly computable (see equation \eqref{eq:Wexpr}) and the remainder term $r_{\rm u}$ satisfies $\abs{r_{\rm u}}\le c\varepsilon^4$.
\end{theorem}

The proof of Theorem \ref{thm:VEswimming} is given in section \ref{sec:VEswim}, and various observations and applications of the swimming expression \eqref{est:swimspeed2} are explored immediately in the following section \ref{sec:numerics}.

\section{Applications and numerical results}\label{sec:numerics}
Here we detail some of the main takeaways and applications of Theorem \ref{thm:VEswimming}, interspersed with numerical simulations.
For the numerical simulations, we will rely on a reformulation of \eqref{eq:VEoriginal1}-\eqref{eq:BCs_original} which avoids the need to solve for the fiber tension $\tau$. We propose a natural extension of the numerical method in \cite{RFTpaper}, which is itself based on a combination of the formulations of \cite{moreau2018asymptotic} and \cite{maxian2021integral}. The method is described in detail in Appendix \ref{app:nummeth}. \\

We begin with some observations about the viscoelastic swimming speed expression \eqref{est:swimspeed2}. We first note the forms of the coefficients $W_{j,m\ell k}$, which are calculated in section \ref{sec:VEswim}. We have
\begin{equation}\label{eq:Wexpr}
\begin{aligned}
W_{1,m\ell k} &= Q_{m,k}-\frac{\mu\delta\omega m}{1+(\delta\omega m)^2}
\bigg(Q_{m,\ell}Q_{m,k}-\delta\omega m(1- H_{m,\ell})Q_{m,k}\\
&\hspace{4cm} -\delta\omega m H_{m,k}Q_{m,\ell}-(1-H_{m,\ell})H_{m,k}\bigg)\\
W_{2,m\ell k} &= H_{m,k} -\frac{\mu\delta\omega m}{1+(\delta\omega m)^2}
\bigg(\delta\omega mQ_{m,\ell}Q_{m,k}+(1- H_{m,\ell})Q_{m,k} \\
&\hspace{4cm} +H_{m,k}Q_{m,\ell}-\delta\omega m(1-H_{m,\ell})H_{m,k}\bigg)\,,
%
\end{aligned}
\end{equation}
where
\begin{equation}\label{eq:QandH}
\begin{aligned}
Q_{m,k}&= \frac{\lambda_k\omega m(1+(1+\mu)(\delta\omega m)^2)}{\lambda_k^2(1+(1+\mu)^2(\delta\omega m)^2)+\omega^2 m^2(2\mu\delta\lambda_k+1+(\delta\omega m)^2)}\,,
\\
H_{m,k}&=\frac{\omega^2 m^2(\mu\delta\lambda_k+ 1+(\delta\omega m)^2)}{\lambda_k^2(1+(1+\mu)^2(\delta\omega m)^2)+\omega^2 m^2(2\mu\delta\lambda_k+1+(\delta\omega m)^2)}\,.
\end{aligned}
\end{equation}

We will be comparing the coefficients \eqref{eq:Wexpr} with the Newtonian swimmer for the same preferred curvature $\kappa_0$. In \cite{RFTpaper}, the Newtonian swimmer satisfying \eqref{eq:NewtonianKappa} was shown to swim with speed $U^{\rm nw}$ of the form 
\begin{equation}\label{eq:NewtSpeed1}
U^{\rm nw}(t) = -\gamma\int_0^1 (\kappa_0)_s\overline\kappa^{\rm nw} \, ds\,,
\end{equation}
which, again averaging over one period, may be written to leading order in $\varepsilon$ as
\begin{equation}\label{eq:NewtSpeed2}
\begin{aligned}
\langle U^{\rm nw}\rangle &= \frac{\gamma}{2}\sum_{m,k,\ell=1}^\infty\frac{\omega^2 m^2}{\omega^2m^2+\lambda_k^2} \bigg(\frac{\lambda_k}{\omega m}\big(a_{m,k}b_{m,\ell}-b_{m,k}a_{m,\ell} \big) \\
&\hspace{5cm}+ a_{m,k}a_{m,\ell}+b_{m,k}b_{m,\ell} \bigg) \int_0^1\psi_k(\psi_\ell)_s\,ds \,.
\end{aligned}
\end{equation}

Some things that we can immediately note in comparing \eqref{est:swimspeed1} and \eqref{est:swimspeed2} to \eqref{eq:NewtSpeed1} and \eqref{eq:NewtSpeed2} include:
\begin{enumerate}
\item[0.] The viscoelastic swimming speed has a complicated dependence on the parameters $\mu$, $\delta$, and $\omega$ and it is not immediately clear how it compares to the Newtonian swimming speed.

\item[1.] The velocity expression \eqref{est:swimspeed1} has the form of the Newtonian swimming speed \eqref{eq:NewtSpeed1} plus a correction term proportional to $\mu(\overline\kappa-\xi)$, but note that $\overline\kappa\neq\overline\kappa^{\rm nw}$ in general.  

\item[2.] If either $\mu=0$ or $\delta=0$, the viscoelastic expression \eqref{est:swimspeed2} reduces to the Newtonian expression \eqref{eq:NewtSpeed2}. In particular, we obtain $W_{1,m\ell k} = Q_{m,k}$ where $Q_{m,k}=\frac{\lambda_k\omega m}{\lambda_k^2+\omega^2m^2}$ and $W_{2,m\ell k} = H_{m,k}$ where $H_{m,k}=\frac{\omega^2m^2}{\lambda_k^2+\omega^2m^2}$.

\item[3.] Due to the form of the eigenfunctions $\psi_k$ of the operator $\mc{L}$ (see \eqref{eq:L_eigs}), in both the Newtonian and viscoelastic cases, if the preferred curvature $\kappa_0(s,t)$ is always odd or always even about the fiber midpoint $s=\frac{1}{2}$, the swimming speed will vanish. This is because $\psi_{2k}(s)$ is odd about $s=\frac{1}{2}$ and $\psi_{2k-1}(s)$ is even for each $k=1,2,\dots$, and thus $\int_0^1\psi_{2k}(\psi_{2\ell})_s\,ds=0$ and $\int_0^1\psi_{2k-1}(\psi_{2\ell-1})_s\,ds=0$ for each $k,\ell$. In particular, if $\kappa_0(s,t)$ can be written purely in terms of either $\psi_{2k}$ or $\psi_{2k-1}$, the swimming speed \eqref{est:swimspeed2} will vanish. 
\end{enumerate}

Otherwise, owing to the complicated nature of the expression \eqref{est:swimspeed2} and particularly of the coefficients \eqref{eq:Wexpr}, it is difficult to say much in general about the viscoelastic swimmer, but we can make some predictions in certain scenarios. For the following, we will consider $\kappa_0$ of the form
\begin{equation}\label{eq:kappa0}
\kappa_0(s,t) = F_1(s) \cos(\omega t)+F_2(s)\sin(\omega t)\,,
\end{equation}
i.e. we will force only a single mode in time. As such, we will drop all dependence on the temporal mode $m$ in our notation. \\

We will begin by considering the case $F_1=F_2$, in which case the swimming speed expression \eqref{est:swimspeed2} reduces to  
\begin{equation}\label{eq:F1equalF2}
\begin{aligned}
\langle U \rangle &=\gamma\sum_{k,\ell=1}^\infty W_{2,\ell k} \, a_ka_\ell
 \int_0^1\psi_k(\psi_\ell)_s\,ds \,.
\end{aligned}
\end{equation}
We referred to these swimmers as `bad swimmers' in the Newtonian setting \cite{RFTpaper} because the coefficient $W_{2,\ell k}$ is given by $\omega^2/(\lambda_k^2+\omega^2)$ and thus decays very rapidly as $k$ increases. Even $\lambda_1^2\approx 500^2$ is extremely large, and thus to see noticeable displacement for the swimmer over one period, we need to take $\omega$ very large as well. For numerical simulations, we will take $\omega=32\pi$.\\

Since $\lambda_k$ is so large, we consider the leading order behavior of $W_{2,\ell k}$ in $1/\lambda_k$, given by
\begin{equation}\label{eq:F1F2coeff_leading}
W_{2,\ell k} \approx H_k + \frac{\mu(\delta\omega)^2}{1+(\delta\omega)^2}H_k -\frac{\mu\delta\omega}{1+(\delta\omega)^2}Q_k\,.
\end{equation}


In Figure \ref{fig:W2k_expr}, we plot the approximate expression \eqref{eq:F1F2coeff_leading} for $W_{2,\ell k}$ for $\lambda_k=\lambda_1$, $\omega=32\pi$, and various values of $\mu$ and $\delta$. \\

\begin{figure}[!ht]
\centering
  \begin{subfigure}[b]{0.45\textwidth}
    \includegraphics[scale=0.5]{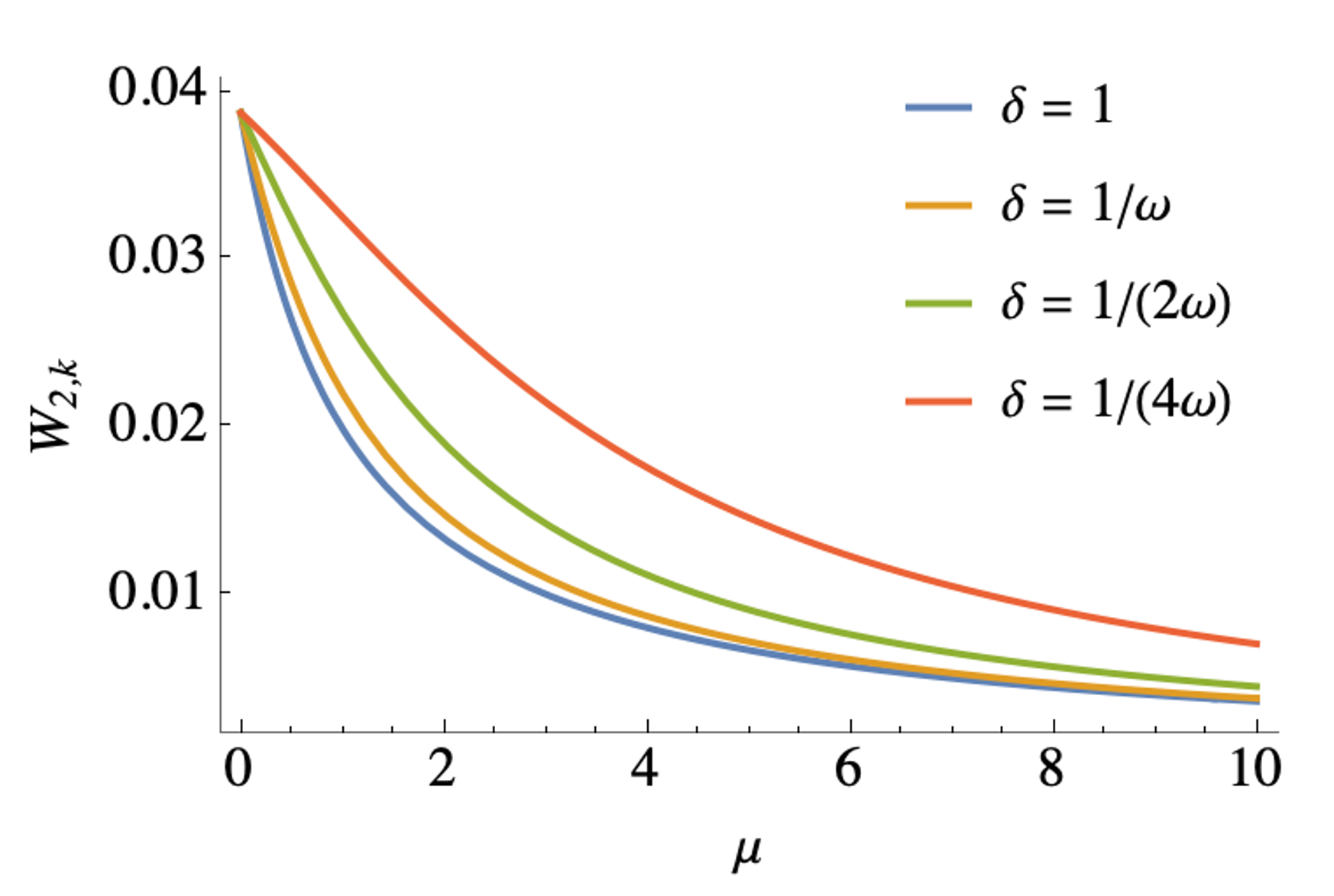}
    \caption{}
  \end{subfigure} 
  \begin{subfigure}[b]{0.45\textwidth}
    \includegraphics[scale=0.5]{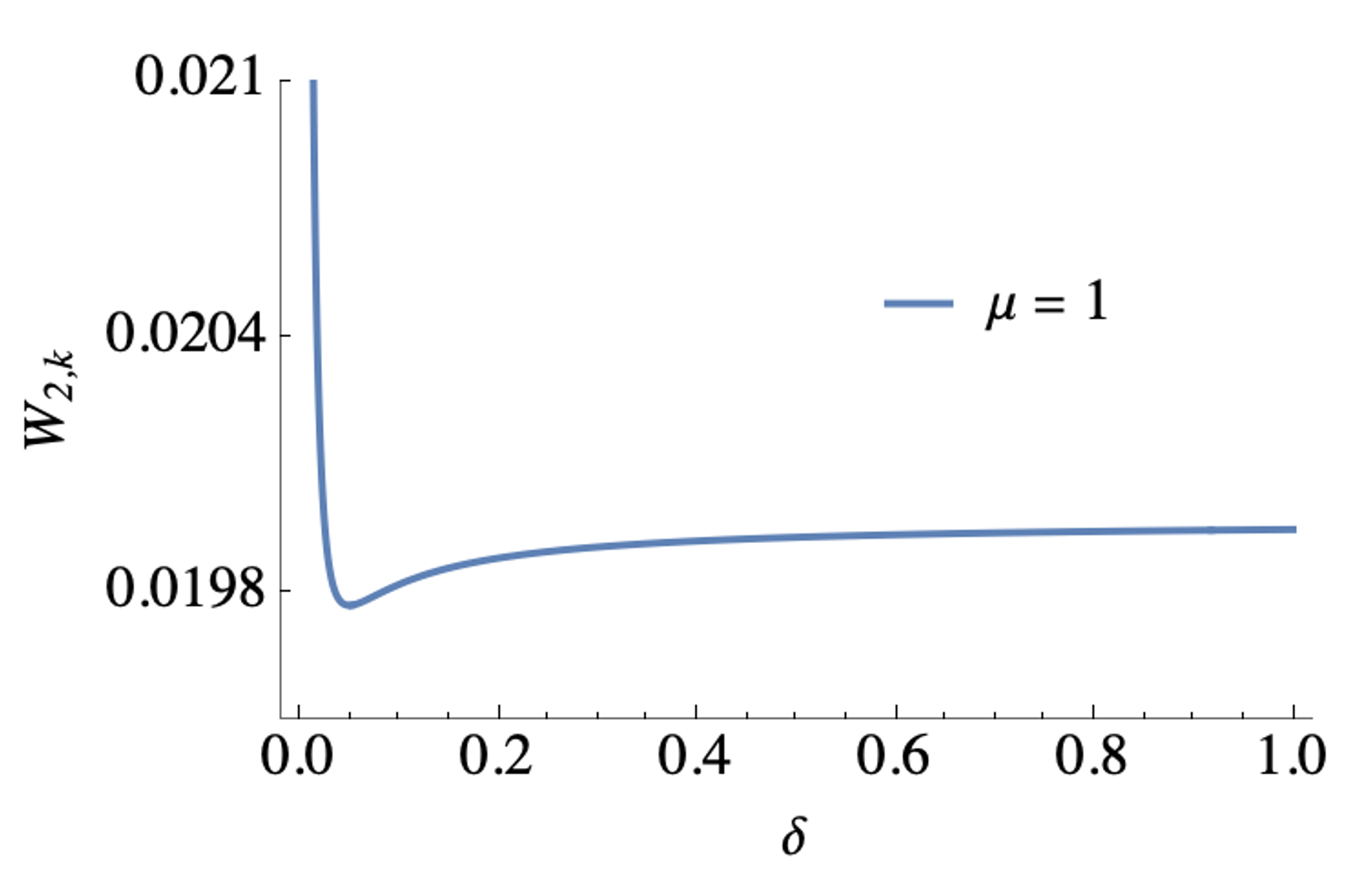}
    \caption{}
  \end{subfigure}
\caption{(a) Plot of the coefficient $W_{2,\ell k}$ \eqref{eq:F1F2coeff_leading} for $\mu\in[0,10]$ for $\lambda_k=\lambda_1$, $\omega=32\pi$, and four different fixed values of $\delta$. (b) Plot of $W_{2,\ell k}$ for $\delta\in[0,1]$ for $\lambda_k=\lambda_1$, $\omega=32\pi$, and $\mu=1$.}
\label{fig:W2k_expr}
\end{figure}

For fixed $\delta>0$, we note that the coefficient $W_{2,\ell k}$ for $k=1$ is monotone decreasing in $\mu$, with a steeper initial decrease for larger values of $\delta$. The coefficient appears to approach 0 as $\mu\to\infty$ for all values of $\delta>0$. The behavior is similar for higher modes $k=2,3,...$, although the magnitude of the coefficient is much smaller due to the $\lambda_k^{-1}$ scaling of $W_{2,\ell k}$. 
For fixed $\mu>0$, the coefficient displays slight non-monotonic behavior in $\delta$. Taking $\delta>0$ results in a sharp initial decline from the $\delta=0$ value, but after this sharp initial drop, the coefficient is very slightly increasing in $\delta$.   \\

Given the behavior displayed in Figure \ref{fig:W2k_expr}, we predict that for any fixed $\delta>0$, we will see a slowdown in the swimmer as $\mu$ is increased. For fixed $\mu>0$, we will see very little change in the swimming speed as $\delta\gtrsim 1$ is increased. 

To test these predictions, we choose a preferred curvature of the form \eqref{eq:kappa0} with $F_1=(s-1)^2$ and $F_2=F_1$ (note that we normalize such that $\norm{F_1}_{L^2}=\norm{F_2}_{L^2}=1$). We take $\omega=32\pi$ and simulate the fiber motion until $t=2$ beginning from a straight line from $x=0$ to $x=1$ along the $x$-axis. We record the fiber's displacement $\int_0^1 \X(s,t_2)\,ds-\int_0^1 \X(s,t_1)\,ds$ between times $t_1=1$ and $t_2=2$ to ensure that the periodic solution has been reached.\\

\begin{figure}[!ht]
\centering
  \begin{subfigure}[b]{0.45\textwidth}
    \includegraphics[scale=0.2]{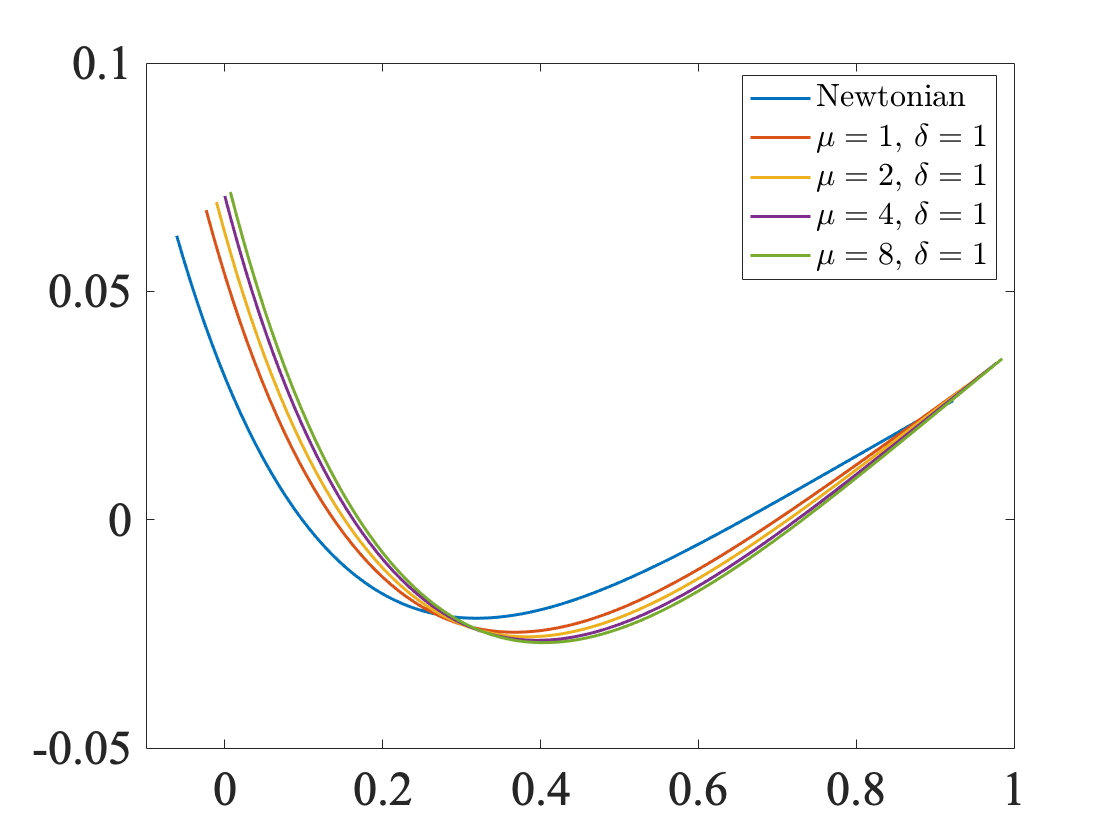}
    \caption{}
  \end{subfigure} 
  \begin{subfigure}[b]{0.45\textwidth}
    \includegraphics[scale=0.2]{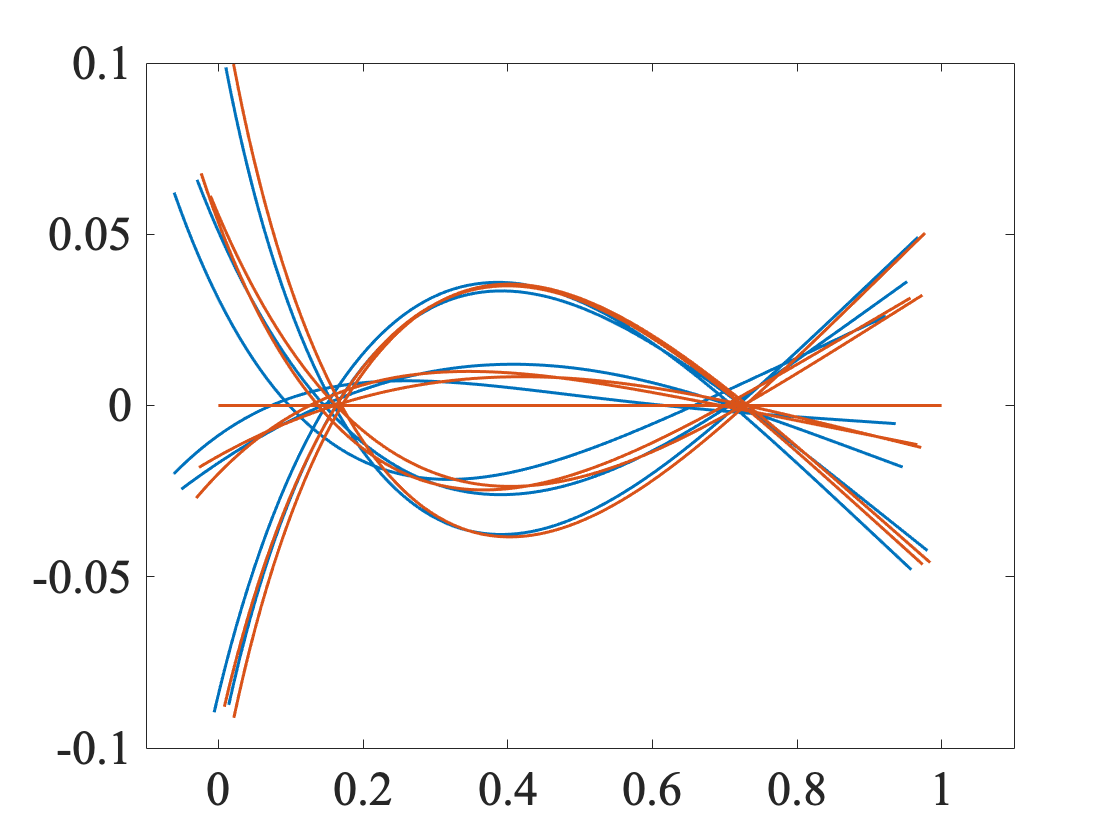}
    \caption{}
  \end{subfigure}
\caption{(a) Location of the fiber at time $t=2$ for fixed $\delta=1$ and five different values of $\mu$. All swimmers swim poorly, but the Newtonian swimmer (blue) is noticeably faster. (b) Comparison of swimmer shapes at ten different snapshots in time for the Newtonian (blue) and $\mu=\delta=1$ (orange) swimmers.}
\label{fig:BadSwim_BigDelta}
\end{figure}

We first fix $\delta=1$ and compare the swimming displacement for 5 different values of $\mu$ (see Figure \ref{fig:BadSwim_BigDelta}a). We compare $\mu=0$ (Newtonian) against $\mu=1,2,4,8$. From Figure \ref{fig:BadSwim_BigDelta}a, we can see that the Newtonian swimmer swims the farthest, although none of the swimmers swim very well. Between $t=1$ and $t=2$, the Newtonian swimmer's displacement is $-0.036$. For $\mu=1,2,4,8$, respectively, the displacement is $-0.018$, $-0.012$, $-0.0069$, $-0.0035$. As predicted, the distance decreases with increasing $\mu$. 
Furthermore, when $\mu$ is fixed at $\mu=1$ and $\delta\gtrsim1$ is varied, there is very little difference in the swimming displacement versus the $\delta=1$ swimmer. When $\delta=1,2,4,8$, respectively, the swimming displacement is still $-0.018$.\\

In Figure \ref{fig:BadSwim_BigDelta}b, snapshots of the location of the swimmer with $\mu=\delta=1$ at ten different points in time between $t=0$ and $t=2$ are plotted against the same points in time for the Newtonian swimmer. \\

\begin{figure}[!ht]
\centering
  \begin{subfigure}[b]{0.45\textwidth}
    \includegraphics[scale=0.18]{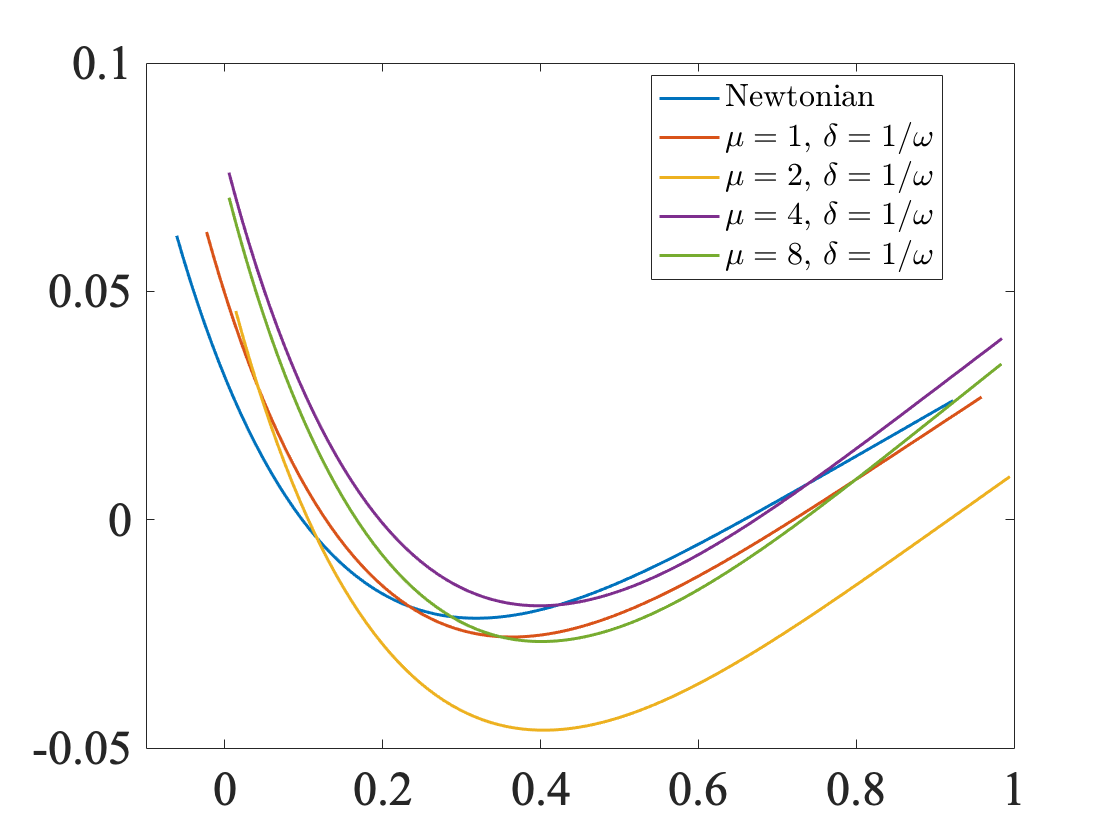}
    \caption{}
  \end{subfigure} 
  \begin{subfigure}[b]{0.45\textwidth}
    \includegraphics[scale=0.47]{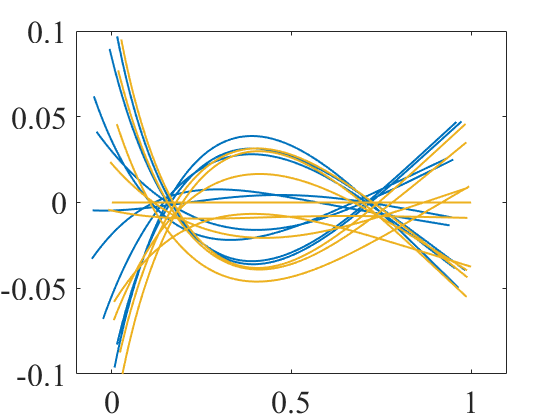}
    \caption{}
  \end{subfigure}
\caption{(a) Location of the fiber at time $t=2$ for fixed $\delta=1/\omega$ and five different values of $\mu$. Notice that each of the viscoelastic swimmers appear to have moved to the right initially, but except for the $\mu=2$ case, after an initial adjustment the swimmers do move leftward. 
(b) Comparison of swimmer shapes at ten different snapshots in time for the Newtonian (blue) and $\mu=2,\delta=1/\omega$ (yellow) swimmers. Note that the yellow swimmer moves backward. }
\label{fig:BadSwim_SmallDelta}
\end{figure}

For small $\delta\sim1/\omega$, things appear to be a bit more complicated than predicted. We fix $\delta=1/\omega$ and simulate the swimmer until $t=2$ using different values of $\mu$. We again calculate the swimmer's displacement between $t=1$ and $t=2$. When $\mu=1$, the displacement is $-0.019$, i.e. roughly the same as when $\delta=1$. However, when $\mu=2$, the swimmer's displacement is $+0.0062$, in particular, we find that the swimmer moves in the opposite direction (see Figure \ref{fig:BadSwim_SmallDelta}). The behavior of the swimmer subsequently becomes more complicated as $\mu$ increases: for $\mu=4$ and $\mu=8$, we observe a displacement of $-0.0030$ and $-0.0040$, respectively. The swimmer now moves in the same direction as for large $\delta$, but instead of losing speed as $\mu$ increases, it appears to gain a bit of speed. For $\mu=8$, the $\delta=1/\omega$ swimmer even swims a bit further than the $\delta=1$ swimmer. 
This discrepancy from the prediction may be due to the already very small nature of displacements when $F_1=F_2$ in the preferred curvature \eqref{eq:kappa0} (indeed, these are the `bad swimmers' in the Newtonian setting \cite{RFTpaper}. It is possible that nonlinear effects or effects of additional terms in the full expression \eqref{eq:Wexpr} for $W_{2,\ell k}$ may be enough to alter the swimming behavior.\\


We next consider the more complicated scenario of $F_1\neq F_2$ in the preferred curvature equation \eqref{eq:kappa0}. Now all terms are present in the swimming expression \eqref{est:swimspeed2}. To leading order in $\frac{1}{\lambda_k}$, the additional coefficients $W_{1,\ell k}$ of the swimming expression \eqref{est:swimspeed2} are given by
\begin{equation}\label{eq:W_reduced}
\begin{aligned}
W_{1,\ell k} &\approx Q_{k}+\frac{\mu(\delta\omega)^2}{1+(\delta\omega)^2}Q_{k} +\frac{\mu\delta\omega}{1+(\delta\omega)^2}H_k\,.
\end{aligned}
\end{equation}
In Figure \ref{fig:W1k_expr}, we plot the coefficient $W_{1,\ell k}$ for $\lambda_k=\lambda_1$, $\omega=32\pi$, and various values of $\delta$ and $\mu$.\\

\begin{figure}[!ht]
\centering
  \begin{subfigure}[b]{0.49\textwidth}
    \includegraphics[scale=0.47]{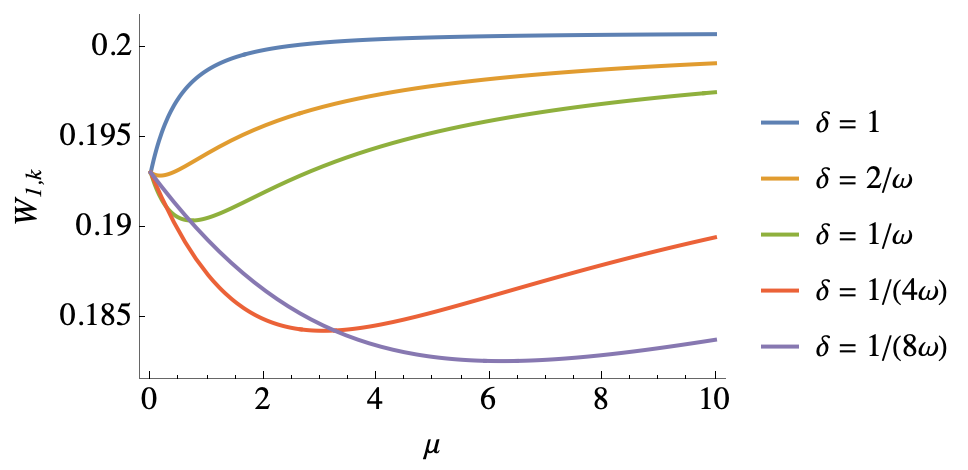}
    \caption{}
  \end{subfigure} 
  \begin{subfigure}[b]{0.49\textwidth}
    \includegraphics[scale=0.47]{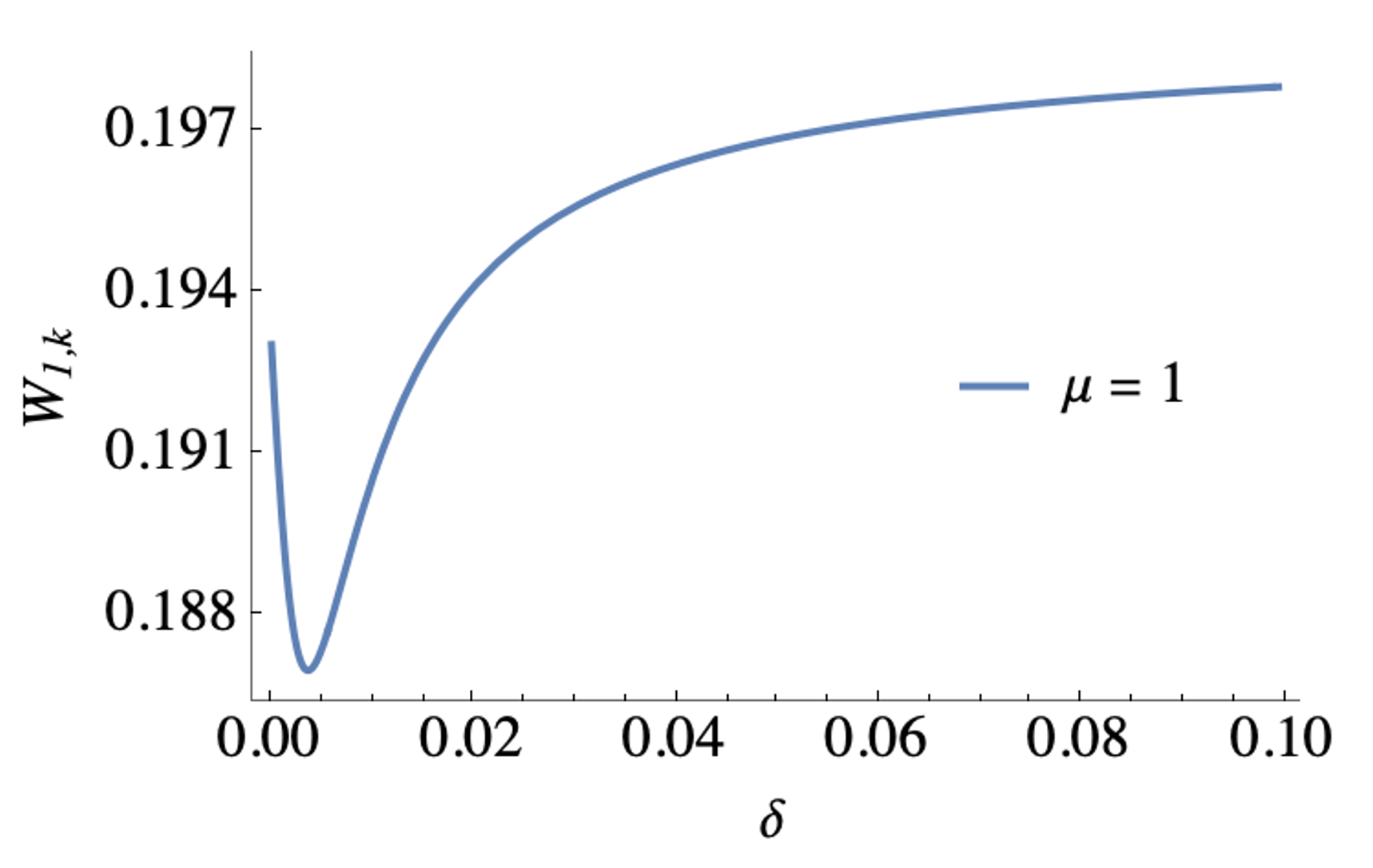}
    \caption{}
  \end{subfigure}
\caption{(a) Plot of the coefficient $W_{1,\ell k}$ \eqref{eq:W_reduced} for $\mu\in[0,10]$ for $\lambda_k=\lambda_1$, $\omega=32\pi$, and five different fixed values of $\delta$. (b) Plot of $W_{1,\ell k}$ for $\delta\in[0,0.1]$ for $\lambda_k=\lambda_1$, $\omega=32\pi$, and $\mu=1$.}
\label{fig:W1k_expr}
\end{figure}

Compared to $W_{2,\ell k}$, the coefficient $W_{1,\ell k}$, in addition to being significantly larger in magnitude, displays much more interesting non-monotonic behavior. For large fixed $\delta\gtrsim 1$, the coefficient $W_{1,\ell k}$ is monotone increasing in $\mu$, whereas for smaller $\delta\sim\frac{1}{\omega}$, the coefficient is initially decreasing for small $\mu$ and then increasing for large $\mu$. For all values of $\delta$, the coefficient appears to approach the value 0.2 asymptotically as $\mu\to\infty$. For fixed $\mu$, we additionally see non-monotonic behavior in $\delta$ for very small $\delta\sim\frac{1}{\omega}$.  \\

The behavior of the coefficient $W_{1,\ell k}$ in Figure \ref{fig:W1k_expr} prompts us to make the following predictions about the viscoelastic swimmer behavior.
\begin{enumerate}
\item[a.] For large $\delta\gtrsim 1$, the viscoelastic swimmer will swim faster as $\mu$ is increased. For fixed $\mu>0$, we will again see very little change in the swimming speed as $\delta\gtrsim 1$ is increased. 

\item[b.] For small $\delta\sim\frac{1}{\omega}$ and small $\mu>0$, we may expect the viscoelastic swimmer to be slower than both the Newtonian ($\delta=0$) and $\delta\gtrsim 1$ swimmers. As $\mu$ increases, we may expect to see the viscoelastic swimmer catch back up to the Newtonian swimmer and eventually surpass it as $\mu$ continues to increase. 
\end{enumerate}

\begin{figure}[!ht]
\centering
    \includegraphics[scale=0.18]{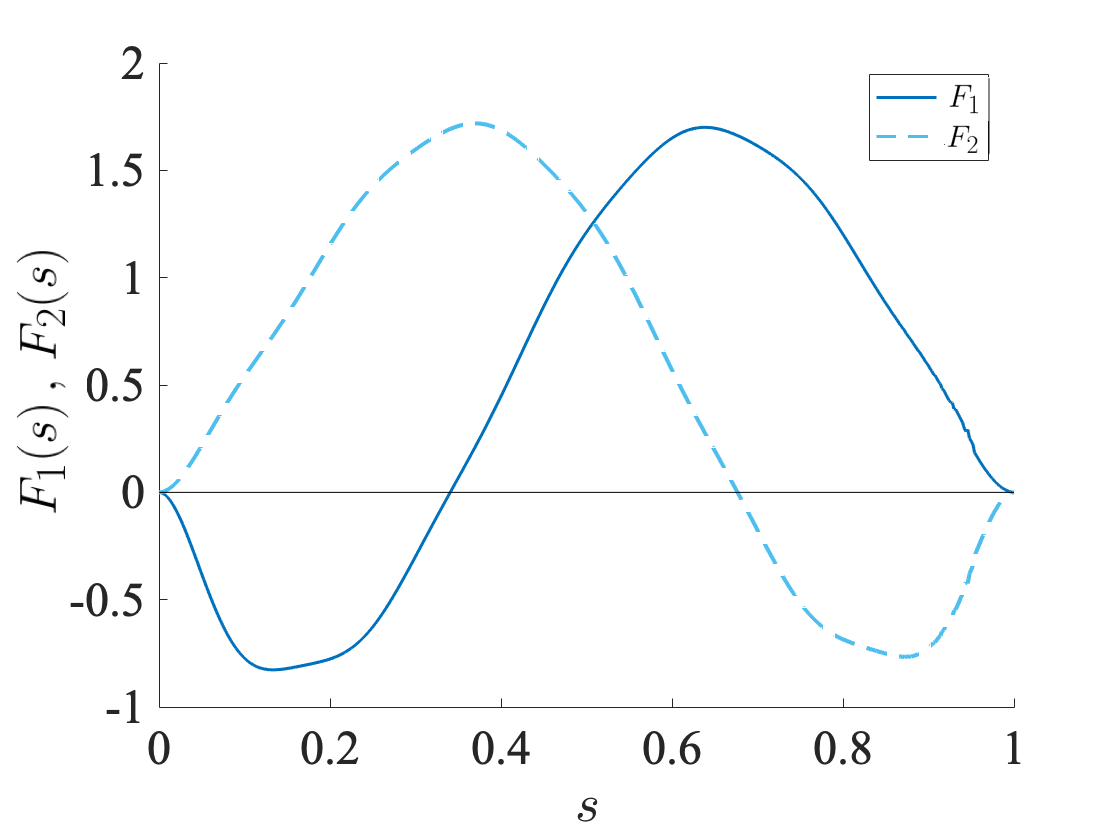}
\caption{The form of $F_1$ and $F_2$ in the preferred curvature \eqref{eq:kappa0} used in numerical tests. This $F_1$ and $F_2$ came from a small numerical optimization of the Newtonian swimming speed \eqref{eq:NewtSpeed2} in \cite{RFTpaper}. }
\label{fig:GoodForcing}
\end{figure}

To test our predictions, we use the preferred curvature components $F_1$ and $F_2$ pictured in Figure \ref{fig:GoodForcing}. These $F_1$ and $F_2$ were computed in \cite{RFTpaper} as the `optimal' preferred curvature $\kappa_0$ of the form \eqref{eq:kappa0} resulting in the greatest average swimming speed \eqref{eq:NewtSpeed2} in the Newtonian setting. The optimization of \eqref{eq:NewtSpeed2} was performed over the first 12 spatial modes $k$ of $\kappa_0$ and thus may not exactly represent the true optimal $\kappa_0$ for the Newtonian swimmer. However, we note that in the Newtonian setting, the combination of $F_1$ and $F_2$ plotted in Figure \ref{fig:GoodForcing} does outperform the classical traveling wave forcing $F_1=\sin(\omega s)$, $F_2=\cos(\omega s)$.\\

As before, we take $\omega=32\pi$ and simulate the swimmer until $t=2$. The swimmer begins as a straight line along the $x$-axis from $x=0$ to $x=1$. Again, we keep track of the displacement of the swimmer $\int_0^1 \X(s,t_2)\,ds-\int_0^1 \X(s,t_1)\,ds$ between times $t_1=1$ and $t_2=2$.\\

\begin{figure}[!ht]
\centering
  \begin{subfigure}[b]{0.45\textwidth}
    \includegraphics[scale=0.2]{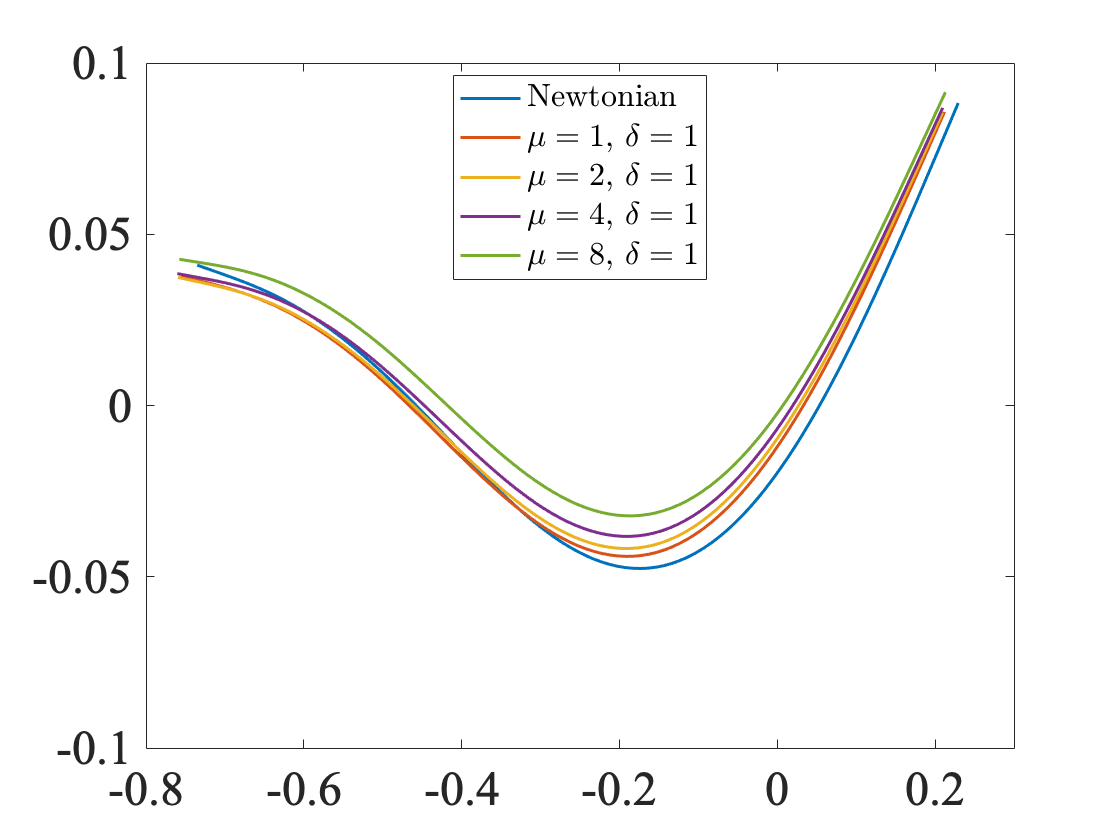}
    \caption{}
  \end{subfigure} 
  \begin{subfigure}[b]{0.45\textwidth}
    \includegraphics[scale=0.2]{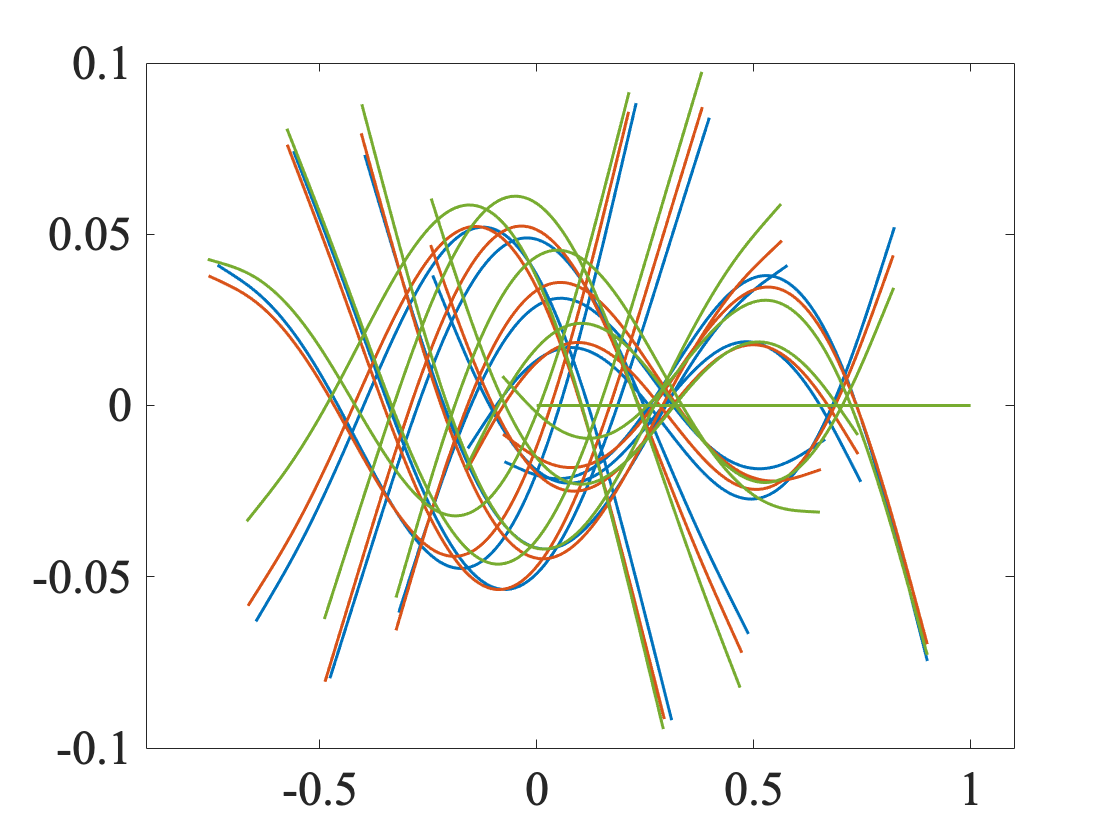}
    \caption{}
  \end{subfigure}
\caption{(a) Location of the fiber at time $t=2$ for fixed $\delta=1$ and five different values of $\mu$. All swimmers swim roughly the same distance, but note that the Newtonian swimmer (blue) is slightly slower. This is opposite from the viscoelastic effects pictured in Figure \ref{fig:BadSwim_BigDelta}. (b) Comparison of swimmer shapes at ten different snapshots in time for the Newtonian (blue), $\mu=\delta=1$ (orange), and $\mu=8,\delta=1$ (green) swimmers.  }
\label{fig:GoodSwim_BigDelta}
\end{figure}

For the case $\delta\gtrsim1$, we again start by fixing $\delta=1$ and compare the fiber displacement for 5 values of $\mu$ (see Figure \ref{fig:GoodSwim_BigDelta}a). In both the viscoelastic and Newtonian settings, the swimmers swim much further than in the case $F_1=F_2$ above, and the differences among their displacements is much smaller. However, we note that the Newtonian swimmer ($\mu=0$) has the smallest displacement between $t=1$ and $t=2$ of $-0.380$. This may be compared with each of the $\mu=1,2,4,8$ swimmers, which have a displacement of $-0.390$, $-0.391$, $-0.392$, and $-0.390$, respectively. Besides the jump in swimming speed between the Newtonian swimmer ($\mu=0$) and $\mu=1$, there is not much difference in displacement among different values of $\mu$, which is not surprising given the shape of the plot of $W_{1,\ell k }$ for $k=1$ (Figure \ref{fig:W1k_expr}a).
A similar result holds when $\mu=1$ is fixed and $\delta=1,2,4,8$ is varied. The displacement in each of these cases is $-0.390$. \\

\begin{figure}[!ht]
\centering
  \begin{subfigure}[b]{0.45\textwidth}
    \includegraphics[scale=0.19]{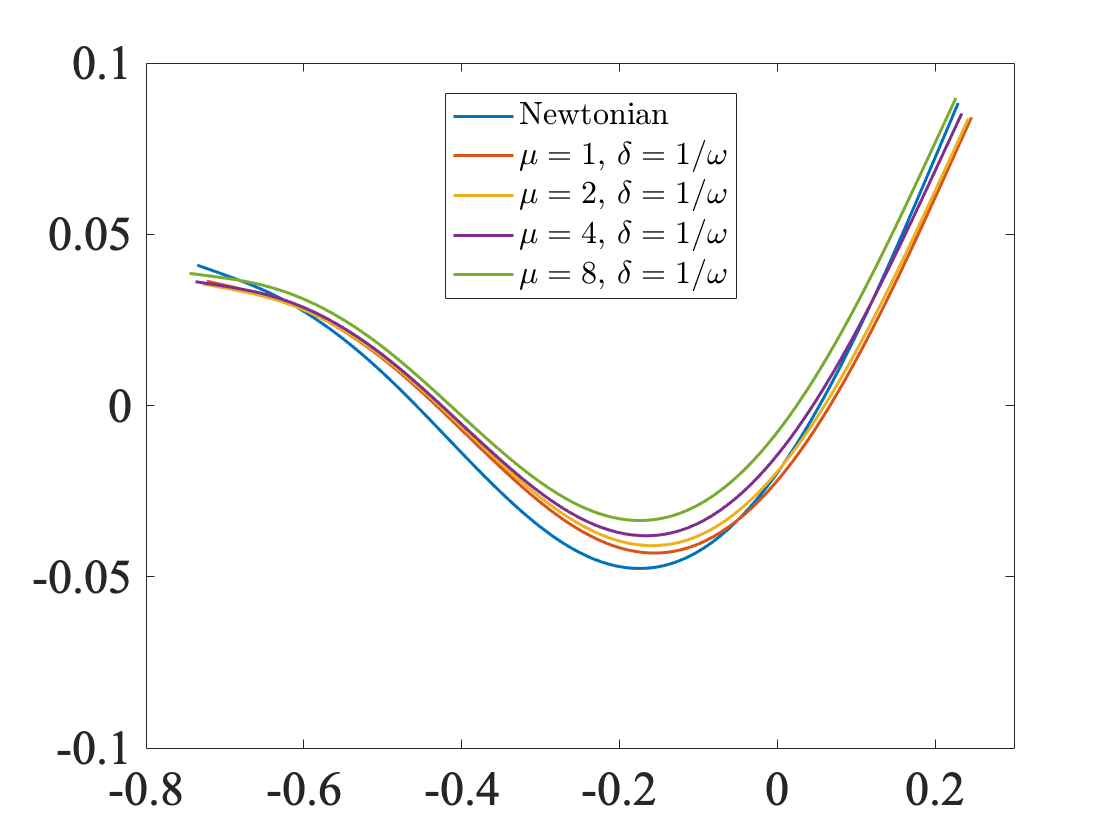}
    \caption{}
  \end{subfigure} 
  \begin{subfigure}[b]{0.45\textwidth}
    \includegraphics[scale=0.19]{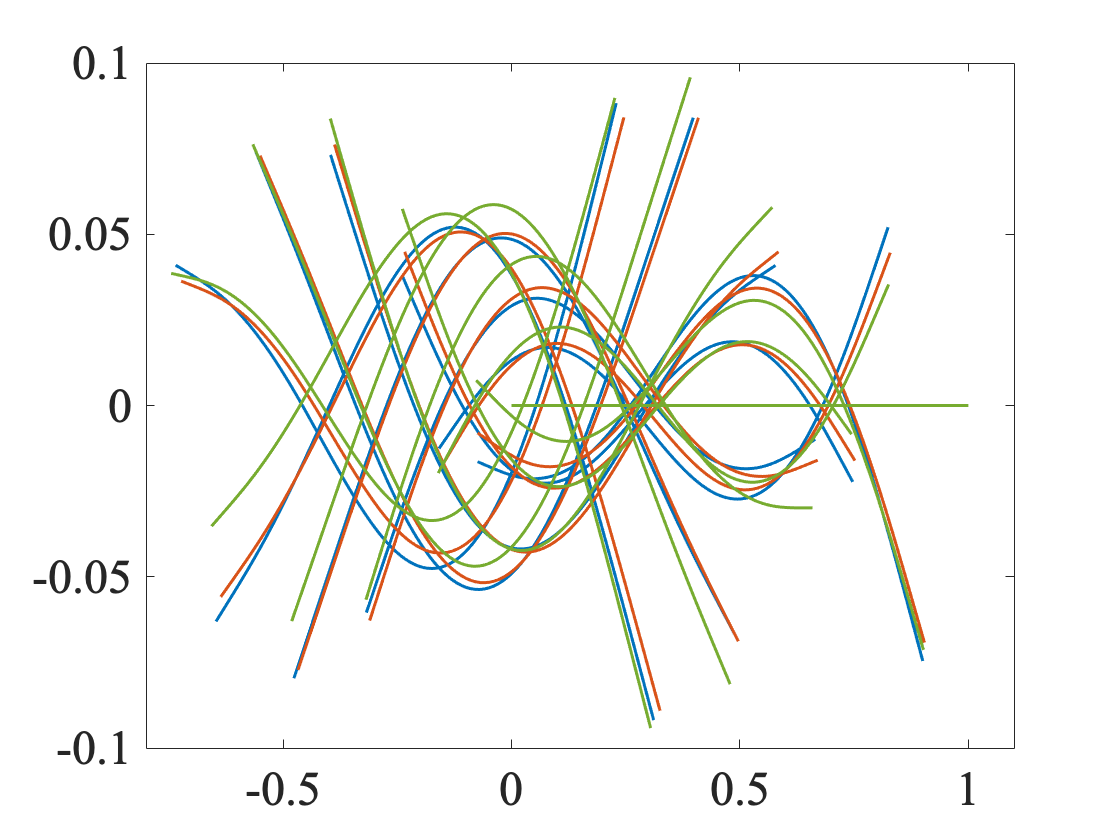}
    \caption{}
  \end{subfigure}
\caption{(a) Location of the fiber at time $t=2$ for fixed $\delta=1/\omega$ and five different values of $\mu$. Again, swimmers swim roughly the same distance, but now the Newtonian swimmer (blue) is faster than the $\mu=1,2$ viscoelastic swimmers and slower than the $\mu=4,8$ swimmers. (b) Comparison of swimmer shapes at ten different snapshots in time for the Newtonian (blue), $\mu=\delta=1$ (orange), and $\mu=8,\delta=1$ (green) swimmers.  }
\label{fig:GoodSwim_SmallDelta}
\end{figure}

As expected, the effect of varying $\mu$ is a bit more interesting at small $\delta$. Fixing $\delta=1/\omega$, we consider $\mu=0,1,2,4,8$. Recalling that the displacement of the Newtonian swimmer from $t=1$ to $t=2$ was $-0.380$, we note that for $\mu=1,2,4,8$, the swimmer's displacement was $-0.373$, $-0.375$, $-0.380$, and $-0.384$, respectively (see Figure \ref{fig:GoodSwim_SmallDelta}). The effect of varying $\mu$ is still relatively small, but more complex than at large $\delta$, varying from a slight inhibition of the swimming speed at smaller $\mu$ to a slight enhancement of the swimming speed at larger $\mu$. This behavior aligns with the predictions of Figure \ref{fig:W1k_expr}a.

\subsection{Discussion}
The numerical tests performed in this section cover a very small portion of the possible parameter space, and indeed an even smaller portion of the possible forcing functions $\kappa_0$. We hope, however, that the tests included here serve to emphasize the complexity of possible behaviors in this model over just a small range of the possible options. We believe this justifies studying the model \eqref{eq:VEoriginal1}-\eqref{eq:BCs_original} in more detail, and hopefully provides convincing evidence that linear viscoelasticity can have an interesting effect on small-amplitude undulatory swimming. \\

Our numerical experiments do not consider the possible effects of including higher modes in the forcing $\kappa_0(s,t)$ in both time and space, as we use only the temporal mode $m=1$ for all simulations and consider coefficients $F_1(s)$ and $F_2(s)$ mostly supported in a few low spatial modes. 
In the Newtonian setting, the $m=1$ mode in time and the $k=1,2$ modes in space result in the fastest swimming speed for a fixed bending energy $\norm{\kappa_0}_{L^2(I)}$ (see \cite[section 4.1]{RFTpaper}), while higher modes contribute less to the overall discplacement of the filament. We anticipate that the story is similar in the linear viscoelastic setting, which is why we choose to simulate only low modes. The effects of higher modes may be very different in a nonlinearly viscoelastic fluid environment and may perhaps contribute more to the overall swimming speed. It would be interesting to consider nonlinear viscoelastic effects of the surrounding fluid, although it is not immediately clear how to incorporate such effects into the reduced curve evolution model \eqref{eq:VEoriginal1}-\eqref{eq:BCs_original} in a physically meaningful way. The absence of such a reduced model would make the analysis much more challenging. For computational results on swimming filaments coupled with a bulk nonlinear viscoelastic fluid via the immersed boundary method, see \cite{li2017flagellar,thomases2014mechanisms,thomases2017role}. \\

Finally, we note that, while only planar deformations are considered in this paper, non-planar motions are an important consideration for real microswimmers. 
The PDE analysis of the model \eqref{eq:VEoriginal1}-\eqref{eq:BCs_original} for fully 3D centerline deformations is essentially the same as in 2D: using a Bishop frame \cite{bishop1975there} to parameterize the curve, we would need to consider the evolution of two curvature components $\kappa_1(s,t)$ and $\kappa_2(s,t)$ according to similar equations to \eqref{eq:kappadot}. 
However, the effects of 3D motions on the swimming speed could be much more complex. It is unclear whether, for a given bending energy, the filament can swim faster if it is allowed to deform out of plane than if it is confined to the plane. This is an important question in the Newtonian setting as well as the viscoelastic setting and merits further exploration. 
\\

The remainder of this paper is devoted to proving Theorems \ref{thm:wellposed}, \ref{thm:per_deb}, and \ref{thm:VEswimming} regarding the PDE behavior of the model \eqref{eq:VEoriginal1}-\eqref{eq:BCs_original}.


\section{Well-posedness and periodic solutions}\label{sec:wellpo}
In this section we prove Theorems \ref{thm:wellposed} and \ref{thm:per_deb}. We start by showing some preliminary bounds in sections \ref{subsec:semigrp} and \ref{subsec:tension}, and then proceed to the proof of Theorem \ref{thm:wellposed} in section \ref{subsec:evolution}. Sections \ref{subsec:periodic} and \ref{subsec:deborah} contain the proof of Theorem \ref{thm:per_deb}.

\subsection{Semigroup properties}\label{subsec:semigrp}
We begin by deriving the following estimates for the semigroup generated by the linear operator $\mc{A}$, given by \eqref{eq:Aop}. 

\begin{lemma}\label{lem:semigrp}
For any $(u,\phi)^{\rm T}\in L^2(I)\times L^2(I)$, for $0\le m+j\le4$, we have
\begin{align}
\norm{e^{t\mc{A}}\begin{pmatrix}
\p_s^ju\\
\p_s^j\phi
\end{pmatrix}}_{\dot H^m(I)\times \dot H^m(I)}
&\le 
c\,\max\{t^{-(m+j)/4},1\}\,e^{-t\Lambda} \norm{\begin{pmatrix}
 u\\
 \phi
 \end{pmatrix} }_{L^2(I)\times L^2(I)} \,, \label{est:semigrp1}
\end{align} 
where $\Lambda=\min\{\lambda_1,\frac{1}{\delta(1+\mu)}\}$ for $\lambda_1$ as in \eqref{eq:L_eigs}, and the constant $c$ is independent of $\delta$.
\end{lemma}

As a consequence of Lemma \ref{lem:semigrp}, we may also show the following small time estimate, which relies on approximating $(u,\phi)^{\rm T}\in L^2(I)\times L^2(I)$ by functions in $D(\mc{L}^r)\times D(\mc{L}^r)$, $0<r\le 1$. 
\begin{lemma}\label{lem:smalltime}
Fix $(u,\phi)^{\rm T}\in L^2(I)\times L^2(I)$ and let $0<r\le1$ and $\varepsilon>0$. There exists $T_\varepsilon>0$ depending on $u$ and $\phi$ such that
\begin{equation}\label{est:smalltime}
\sup_{t\in[0,T_\varepsilon]}\min\{t^r,1\}e^{t\Lambda}\norm{e^{t\mc{A}}\begin{pmatrix}
u\\
\phi
\end{pmatrix}}_{\dot H^{4r}(I)\times \dot H^{4r}(I)}
\le \varepsilon\,.
\end{equation} 
\end{lemma}

\begin{proof}[Proof of Lemma \ref{lem:semigrp}]
For $(w,\varphi)^{\rm T}\in D(\mc{L})\times D(\mc{L})$, the eigenfunction expansion of $\mc{A} \begin{pmatrix}
w\\
\varphi
\end{pmatrix}$ (see \eqref{eq:wt_expand}) may be written as $\sum_{k=1}^\infty
\wt{\mc{A}}_k \begin{pmatrix}
\wt w_k\\
\wt\varphi_k
\end{pmatrix}\psi_k$, where 
\begin{equation}\label{eq:Atildek}
\wt{\mc{A}}_k=\begin{pmatrix}
-(1+\mu)\lambda_k & \mu\lambda_k \\
\delta^{-1} & -\delta^{-1}
\end{pmatrix}\,.
\end{equation}

We study the properties of $\wt{\mc A}_k$. The eigenvalues of $\wt{\mc A}_k$ are given by
\begin{equation}\label{eq:A_evals}
\nu_k^\pm = \frac{1}{2\delta}\bigg(-\big(1+(1+\mu)\delta\lambda_k\big)\pm\sqrt{\big(1+(1+\mu)\delta\lambda_k\big)^2-4\delta\lambda_k} \bigg) \,,
\end{equation}\label{eq:A_evecs}
with corresponding eigenvectors
\begin{equation}
\bv_k^{-} = \begin{pmatrix}
1 \\
\frac{1}{1+\delta \nu_k^{-}}
\end{pmatrix} \,, \quad
\bv_k^{+} = \begin{pmatrix}
1+\delta \nu_k^{+}\\
1
\end{pmatrix}  \,.
\end{equation}
We note in particular that $\delta\nu_k^+\to-\frac{1}{1+\mu}$ monotonically as $k\to\infty$, while $\lambda_k^{-1}\delta\nu_k^{-}\to -\delta(1+\mu)$ as $k\to\infty$. 
We may decompose the unit vectors $(1,0)^{\rm T}$ and $(0,1)^{\rm T}$ in terms of the eigenvectors $\bv_k^\pm$ of $\wt{\mc{A}}_k$ as: 
\begin{equation}\label{eq:A_decomp}
\begin{aligned}
\begin{pmatrix}
1\\
0
\end{pmatrix} &= a_k^{-}\bv^{-}_k + a_k^{+}\bv^{+}_k\,,\qquad  a_k^{-}=-\frac{1+\delta \nu_k^{-}}{\delta(\nu_k^{+}-\nu_k^{-})} \,; \quad a_k^{+}=\frac{1}{\delta(\nu_k^{+}-\nu_k^{-})}\,, \\
\begin{pmatrix}
0\\
1
\end{pmatrix} &= b_k^{-}\bv_k^{-} + b_k^{+}\bv_k^{+}\,,\qquad  b_k^{-}=\frac{(1+\delta \nu_k^{+})(1+\delta \nu_k^{-})}{\delta(\nu_k^{+}-\nu_k^{-})}\, ; \quad b_k^{+}=-\frac{1+\delta \nu_k^{-}}{\delta(\nu_k^{+}-\nu_k^{-})}\,.
\end{aligned}
\end{equation}
Noting that
\begin{align*}
\delta(\nu_k^{+}-\nu_k^{-}) = \sqrt{(1+(1+\mu)\delta\lambda_k)^2-4\delta\lambda_k} \,,
\end{align*}
we have that there exist constants $c$ independent of both $k$ and $\delta$ such that 
\begin{equation}\label{eq:A_coeffbds}
\abs{a_k^{-}\bv_k^-} \le c\,,  \quad \abs{b_k^{-}\bv_k^-} \le c\,, \quad \text{and}
\quad (\delta\lambda_k)^r\abs{a_k^{+}\bv_k^+}\le c\,, \quad (\delta\lambda_k)^r\abs{b_k^{+}\bv_k^+}\le c\,, \quad 0\le r\le 1\,.
\end{equation}
Using the decomposition \eqref{eq:A_decomp} and the bounds \eqref{eq:A_coeffbds}, for any $0\le r\le 1$, we may estimate
\begin{align*}
\norm{\mc{L}^r e^{t\mc{A}}\begin{pmatrix}
u\\
0
\end{pmatrix}}_{L^2\times L^2} &= 
\norm{\sum_{k=1}^\infty\lambda_k^r e^{t\wt{\mc{A}}_k}\begin{pmatrix}
\wt u_k\psi_k\\
0
\end{pmatrix}}_{L^2\times L^2} 
= \norm{\sum_{k=1}^\infty\lambda_k^r \bigg(a_k^{-}\bv_k^{-}e^{t\nu_k^{-}}+a_k^{+}\bv_k^{+}e^{t\nu_k^{+}} \bigg)\wt{u}_k\psi_k}_{L^2\times L^2} \\
&\le c\,\sup_k\bigg(\lambda_k^r\,e^{t\nu_k^{-}}+\delta^{-r}\,e^{t\nu_k^{+}}\bigg)\norm{\sum_{k=1}^\infty\wt{u}_k\psi_k}_{L^2} \\
&\le c\,\bigg(\sup_k\bigg(\lambda_k^r\,e^{-t(\lambda_k-\lambda_1)}\bigg)e^{-t\lambda_1}+\delta^{-r}\,e^{-t/(\delta(1+\mu))}\bigg)\norm{u}_{L^2} \\
&\le c\,\max\{t^{-r},1\}\bigg(e^{-t\lambda_1}+e^{-t/(\delta(1+\mu))}\bigg)\norm{u}_{L^2} \,.
\end{align*}
By an analogous series of estimates, we may also show
\begin{align*}
\norm{\mc{L}^r e^{t\mc{A}}\begin{pmatrix}
0\\
\phi
\end{pmatrix}}_{L^2\times L^2} &= 
\norm{\sum_{k=1}^\infty\lambda_k^r e^{t\wt{\mc{A}}_k}\begin{pmatrix}
0 \\
\wt{\phi}_k\psi_k
\end{pmatrix}}_{L^2\times L^2} 
= \norm{\sum_{k=1}^\infty \bigg(b_k^{-}\bv_k^{-}e^{\nu_k^{-}t}+b_k^{+}\bv_k^{+}e^{\nu_k^{+}t} \bigg)\lambda_k^r\wt{\phi}_k\psi_k}_{L^2\times L^2} \\
&\le c\,\max\{t^{-r},1\}\bigg(e^{-t\lambda_1}+e^{-t/(\delta(1+\mu))}\bigg)\norm{\phi}_{L^2}\,.
\end{align*}
Recalling that $D(\mc{L}^r)\subseteq H^{4r}$ for each $0\le r\le1$ by \eqref{eq:domainLr}, we obtain estimate \eqref{est:semigrp1} for $j=0$. \\

For $0<j\le 4$, we proceed by a duality argument as in \cite{RFTpaper}. 
 In particular, for $(u,\phi),(w,\varphi)\in C_c^\infty(I)\times C_c^\infty(I)$, we have 
\begin{align*}
\sup_{\norm{(u,\phi)}_{L^2\times L^2}=1}\norm{e^{t\mc{A}}\begin{pmatrix}
\p_s^ju\\
\p_s^j\phi
 \end{pmatrix}}_{L^2\times L^2} &= 
 \sup_{\norm{(u,\phi)}_{L^2\times L^2}=\norm{(w,\varphi)}_{L^2\times L^2}=1}\bigg(\begin{pmatrix}
 w\\
 \varphi
 \end{pmatrix},
e^{t\mc{A}}\begin{pmatrix}
 \p_s^ju\\
 \p_s^j\phi
 \end{pmatrix} \bigg)_{L^2\times L^2} \\
 &=\sup_{\norm{(u,\phi)}_{L^2\times L^2}=\norm{(w,\varphi)}_{L^2\times L^2}=1}\bigg( \p_s^je^{t\mc{A}^*}\begin{pmatrix}
 w\\
 \varphi
 \end{pmatrix},
\begin{pmatrix}
u\\
\phi
 \end{pmatrix} \bigg)_{L^2\times L^2} \\
 &\le \sup_{\norm{(w,\varphi)}_{L^2\times L^2}=1}\norm{\p_s^je^{t\mc{A}^*}\begin{pmatrix}
 w\\
 \varphi
 \end{pmatrix} }_{L^2\times L^2} \,,
%
\end{align*}
where we are using the notation 
\begin{align*}
 \bigg(\begin{pmatrix}
 w\\
 \varphi
 \end{pmatrix},
\begin{pmatrix}
u\\
\phi
 \end{pmatrix} \bigg)_{L^2\times L^2} := (w,u)_{L^2}+(\varphi,\phi)_{L^2}\,.
 \end{align*} 

Now, the adjoint operator $\mc{A}^*$ satisfies 
\begin{align*}
\wt{\mc{A}^*}_k = (\wt{\mc{A}}_k)^{\rm T}= \begin{pmatrix}
-(1+\mu)\lambda_k & \delta^{-1} \\
\mu\lambda_k & -\delta^{-1}
\end{pmatrix}\,,
\end{align*}
with the same eigenvalues \eqref{eq:A_evals} as $\wt{\mc{A}}_k$ but with eigenvectors given by 
\begin{equation}\label{eq:Astar_evec}
\bv^{*+}_k = \begin{pmatrix}
\frac{1+\delta\nu_k^+}{\mu\delta\lambda_k}\\
1
\end{pmatrix} \,,\quad 
\bv^{*-}_k = \begin{pmatrix}
1\\
\frac{\mu\delta\lambda_k}{1+\delta\nu_k^-}
\end{pmatrix}\,.
\end{equation}

We may again decompose the unit vectors $(1,0)^{\rm T}$ and $(0,1)^{\rm T}$ in terms of the eigenvectors \eqref{eq:Astar_evec} of $\wt{\mc{A}^*}_k$ as 
\begin{equation}\label{eq:Astar_decomp}
\begin{aligned}
\begin{pmatrix}
1\\
0
\end{pmatrix} &= a^{*-}_k\bv^{*-}_k+ a^{*+}_k\bv^{*+}_k\,, \quad a^{*-}_k=-\frac{1+\delta\nu_k^-}{\delta(\nu_k^+-\nu_k^-)}\,; \quad a^{*+}_k=\frac{\mu\delta\lambda_k}{\delta(\nu_k^+-\nu_k^-)}\,,  \\
\begin{pmatrix}
0\\
1
\end{pmatrix} &= b^{*-}_k\bv^{*-}_k+ b^{*+}_k\bv^{*+}_k\,, \quad b^{*-}_k=\frac{(1+\delta\nu_k^+)(1+\delta\nu_k^-)}{\mu\delta\lambda_k\,\delta(\nu_k^+-\nu_k^-)}\,; \quad b^{*+}_k=-\frac{1+\delta\nu_k^-}{\delta(\nu_k^+-\nu_k^-)}\,.
\end{aligned}
\end{equation}

Note that
\begin{align*}
1+\delta\nu_k^+ &= \frac{1}{2} - \frac{1}{2}\sqrt{(1+(1+\mu)\delta\lambda_k)^2-4\delta\lambda_k} - \frac{1}{2}(1+\mu)\delta\lambda_k\\
&=-\frac{\delta\lambda_k}{2}\bigg(\frac{2(1+\mu)+(1+\mu)^2\delta\lambda_k+4}{1+\sqrt{(1+(1+\mu)\delta\lambda_k)^2-4\delta\lambda_k}}+(1+\mu)\bigg)\,;
\end{align*}
in particular, we may bound
\begin{align*}
\abs{\frac{1+\delta\nu_k^+}{\delta\lambda_k}}\le c
\end{align*}
for some $c$ independent of $k$ and $\delta$. Then, as in \eqref{eq:A_coeffbds}, we have that the components in \eqref{eq:Astar_decomp} satisfy
\begin{equation}\label{eq:Astar_ests}
\abs{a^{*-}_k\bv^{*-}_k}\le c, \quad \abs{b^{*-}_k\bv^{*-}_k}\le c, \quad (\delta\lambda_k)^r\abs{a^{*+}_k\bv^{*+}_k}\le c, \quad (\delta\lambda_k)^r\abs{b^{*+}_k\bv^{*+}_k}\le c\,, \quad 0\le r \le 1\,.
\end{equation}

Using the decomposition \eqref{eq:Astar_decomp} and the bounds \eqref{eq:Astar_ests} in the same way as above, we have that
\begin{align*}
\norm{\mc{L}^{j/4}e^{t\mc{A}^*}\begin{pmatrix}
 w\\
 0
 \end{pmatrix} }_{L^2\times L^2}
 &= 
%
\norm{\sum_{k=1}^\infty\lambda_k^{j/4} \bigg(a_k^{*-}\bv_k^{*-}e^{t\nu_k^{-}}+a_k^{*+}\bv_k^{*+}e^{t\nu_k^{+}} \bigg)\wt{w}_k\psi_k}_{L^2\times L^2} \\
&\le c\,\max\{t^{-j/4},1\}\bigg(e^{-t\lambda_1}+e^{-t/(\delta(1+\mu))}\bigg)\norm{w}_{L^2} 
\end{align*}
and
\begin{align*}
\norm{\mc{L}^{j/4}e^{t\mc{A}^*}\begin{pmatrix}
 0\\
 \varphi
 \end{pmatrix} }_{L^2\times L^2}
 &= \norm{\sum_{k=1}^\infty\lambda_k^{j/4} \bigg(b_k^{*-}\bv_k^{*-}e^{t\nu_k^{-}}+b_k^{*+}\bv_k^{*+}e^{t\nu_k^{+}} \bigg)\wt{\varphi}_k\psi_k}_{L^2\times L^2} \\
&\le c\,\max\{t^{-j/4},1\}\bigg(e^{-t\lambda_1}+e^{-t/(\delta(1+\mu))}\bigg)\norm{\varphi}_{L^2} \,.
\end{align*}
In particular, since $D(\mc{L}^{j/4})\subseteq H^j$ for $0\le j\le 4$, we have
\begin{align*}
\sup_{\norm{(u,\phi)}_{L^2\times L^2}=1}\norm{e^{t\mc{A}}\begin{pmatrix}
\p_s^ju\\
\p_s^j\phi
 \end{pmatrix}}_{L^2\times L^2} &=
\sup_{\norm{(w,\varphi)}_{L^2\times L^2}=1}\norm{\p_s^je^{t\mc{A}^*}\begin{pmatrix}
 w\\
 \varphi
 \end{pmatrix} }_{L^2\times L^2} \\
 &\le 
 c\,\max\{t^{-j/4},1\}\bigg(e^{-t\lambda_1}+e^{-t/(\delta(1+\mu))}\bigg)\,.
\end{align*}

The desired estimate then holds for $(u,\phi)\in L^2\times L^2$ by density.
\end{proof}

\begin{proof}[Proof of Lemma \ref{lem:smalltime}]
Let $u_n=\sum_{k=1}^n\wt u_k\psi_k$ and $\phi_n=\sum_{k=1}^n\wt\phi_k\psi_k$. Since these sums are finite, we have $u_n,\phi_n\in D(\mc{L}^r)$ (see \eqref{eq:domainLr}) and we may estimate
\begin{align*}
\norm{e^{t\mc{A}}\begin{pmatrix}
u_n\\
\phi_n
\end{pmatrix}}_{\dot H^{4r}\times \dot H^{4r}} 
&\le \norm{\mc{L}^re^{t\mc{A}}\begin{pmatrix}
u_n\\
\phi_n
\end{pmatrix}}_{L^2\times L^2}
\le c\,e^{-t\Lambda}\norm{\sum_{k=1}^n\lambda_k^r\psi_k\begin{pmatrix}
\wt u_k\\
\wt\phi_k
\end{pmatrix}}_{L^2\times L^2} \le c_n\,e^{-t\Lambda}\,.
\end{align*}

By Lemma \ref{lem:semigrp}, we then have
\begin{align*}
e^{t\Lambda}\min\{t^r,1 \}&
 \norm{e^{t\mc{A}}\begin{pmatrix}
u\\
\phi
\end{pmatrix}}_{\dot H^{4r}\times \dot H^{4r}} \\
&\le e^{t\Lambda} \min\{t^r,1 \}\bigg(\norm{e^{t\mc{A}}\begin{pmatrix}
u-u_n\\
\phi-\phi_n
\end{pmatrix}}_{\dot H^{4r}\times \dot H^{4r}} 
+ \norm{e^{t\mc{A}}\begin{pmatrix}
u_n\\
\phi_n
\end{pmatrix}}_{\dot H^{4r}\times \dot H^{4r}}\bigg) \\
&\le c\norm{\begin{pmatrix}
u-u_n\\
\phi-\phi_n
\end{pmatrix}}_{L^2\times L^2} 
+ 
\min\{t^r,1 \}\,c_n \,.
\end{align*}
Taking $n$ sufficiently large and $t$ sufficiently small (depending on $n$), we obtain Lemma \ref{lem:smalltime}.
\end{proof}

\subsection{Tension equation}\label{subsec:tension}
We next prove the following lemma regarding the elliptic equation \eqref{eq:taueq} for the tension $\overline\tau$.

\begin{lemma}\label{lem:tension}
Given $(\overline\kappa,\xi,\kappa_0)\in H^1(I)\times H^1(I)\times H^1(I)$, there exists a unique weak solution $\overline\tau\in H^1_0(I)$ to \eqref{eq:taueq} satisfying 
\begin{equation}\label{est:tauH1}
\norm{\overline\tau}_{H^1(I)} \le c\bigg(\norm{\overline\kappa}_{\dot H^1}^2(\norm{\overline\kappa}_{L^2}+1)+\norm{\xi}_{\dot H^1}(\norm{\overline\kappa}_{\dot H^1}+\norm{\kappa_0}_{H^1})+\norm{\overline\kappa}_{\dot H^1}\norm{\kappa_0}_{H^1}(\norm{\kappa_0}_{L^2}+1) \bigg)\,.
\end{equation}

Furthermore, given $(\overline\kappa_a,\xi_a),(\overline\kappa_b,\xi_b)\in H^1(I)\times H^1(I)$, define $\overline\tau_a,\overline\tau_b\in H^1_0(I)$ to be the corresponding unique weak solutions to \eqref{eq:taueq}. The difference $\overline\tau_a-\overline\tau_b$ then satisfies
\begin{equation}\label{est:tau_lip}
\begin{aligned}
\norm{\overline\tau_a-\overline\tau_b}_{H^1} &\le
c\norm{\overline\kappa_a-\overline\kappa_b}_{L^2}\bigg(
(\norm{\xi_a}_{\dot H^1}+\norm{\xi_b}_{\dot H^1})\big(\norm{\overline\kappa_a}_{\dot H^1}+\norm{\overline\kappa_b}_{\dot H^1}+\norm{\kappa_0}_{H^1}\big)\\
&\quad +\big( \norm{\overline\kappa_a}_{\dot H^1}^2+ \norm{\overline\kappa_b}_{\dot H^1}^2+\norm{\kappa_0}_{H^1}^2\big)\bigg)\big(\norm{\overline\kappa_a}_{L^2}+\norm{\overline\kappa_b}_{L^2}+\norm{\kappa_0}_{L^2}+1\big)^2 \\
&\qquad
 + c\norm{\overline\kappa_a-\overline\kappa_b}_{\dot H^1}\big(\norm{\overline\kappa_a}_{\dot H^1}+\norm{\overline\kappa_b}_{\dot H^1}+\norm{\xi_a}_{\dot H^1}+\norm{\kappa_0}_{H^1}\big) \\
 &\qquad +c\norm{\xi_a-\xi_b}_{\dot H^1}(\norm{\overline\kappa_b}_{\dot H^1}+\norm{\kappa_0}_{H^1}) \,.
\end{aligned}
\end{equation}
\end{lemma}

\begin{proof}
We begin by decomposing $\overline\tau$ into $\overline\tau=\overline\tau^{\rm nw}+\overline\tau^{\rm ve}$, where
\begin{align}
\overline\tau^{\rm nw}_{ss}-\frac{(\overline\kappa+\kappa_0)^2}{1+\gamma}\overline\tau^{\rm nw} &=\frac{1}{1+\gamma}\mc{T}[\overline\kappa,\kappa_0]  \label{eq:tau_nw}\\
\overline\tau^{\rm ve}_{ss}-\frac{(\overline\kappa+\kappa_0)^2}{1+\gamma}\overline\tau^{\rm ve} &= \frac{1}{1+\gamma}\frac{\mu}{1+\mu}\bigg( (2+\gamma)((\overline\kappa+\kappa_0)\xi_s)_s - (\overline\kappa+\kappa_0)_s\xi_s \bigg)\,. \label{eq:tau_ve}
\end{align}

From \cite{RFTpaper}, we have that there exists a unique $\overline\tau^{\rm nw}\in H^1_0(I)$ satisfying \eqref{eq:tau_nw} in a weak sense, with 
\begin{equation}\label{eq:tau_nw_bd}
\norm{\overline\tau^{\rm nw}}_{H^1} \le 
c\bigg(\norm{\overline\kappa}_{\dot H^1}^2(\norm{\overline\kappa}_{L^2}+1)+\norm{\overline\kappa}_{\dot H^1}\norm{\kappa_0}_{H^1}(\norm{\kappa_0}_{L^2}+1) \bigg)\,.
\end{equation}

It thus remains to consider \eqref{eq:tau_ve}. As in the Newtonian setting, we define the bilinear form
\begin{align*}
\mc{B}(\overline\tau,\phi) := \int_0^1\bigg(\overline\tau_s\phi_s+\frac{(\overline\kappa+\kappa_0)^2}{1+\gamma}\overline\tau\phi \bigg)\,ds\,,
\end{align*}
which is bounded and coercive on $H^1_0(I)$. A weak solution to \eqref{eq:tau_ve} may then be defined as $\overline\tau^{\rm ve}\in H^1_0$ satisfying 
\begin{align*}
\mc{B}(\overline\tau^{\rm ve},\phi) = \frac{1}{1+\gamma}\frac{\mu}{1+\mu}\int_0^1\bigg((2+\gamma)(\overline\kappa+\kappa_0)\xi_s\phi_s + (\overline\kappa+\kappa_0)_s\xi_s \phi \bigg)\,ds
\end{align*}
for all $\phi\in H^1_0(I)$, and existence and uniqueness follow from the Lax-Milgram lemma. Furthermore, we may estimate
\begin{align*}
\mc{B}(\overline\tau^{\rm ve},\overline\tau^{\rm ve}) &\le
c(\norm{\overline\kappa}_{\dot H^1}+\norm{\kappa_0}_{H^1})\norm{\xi}_{\dot H^1}\norm{\overline\tau^{\rm ve}}_{H^1} \,,
\end{align*}
and using that $\norm{\overline\tau^{\rm ve}}_{H^1}^2 \le c\mc{B}(\overline\tau^{\rm ve},\overline\tau^{\rm ve})$ along with Young's inequality, we obtain
\begin{equation}\label{eq:tau_ve_bd}
\norm{\overline\tau^{\rm ve}}_{H^1} \le c(\norm{\overline\kappa}_{\dot H^1}+\norm{\kappa_0}_{H^1})\norm{\xi}_{\dot H^1}\,.
\end{equation}
Combining estimates \eqref{eq:tau_nw_bd} and \eqref{eq:tau_ve_bd}, we obtain \eqref{est:tauH1}. \\

To show the Lipschitz estimate \eqref{est:tau_lip}, we first recall that, from \cite{RFTpaper}, we have 
\begin{equation}\label{eq:tau_nw_lip}
\begin{aligned}
\norm{\overline\tau^{\rm nw}_a-\overline\tau^{\rm nw}_b}_{H^1} &\le
c\norm{\overline\kappa_a-\overline\kappa_b}_{L^2}
 \big( \norm{\overline\kappa_a}_{\dot H^1}^2+ \norm{\overline\kappa_b}_{\dot H^1}^2+\norm{\kappa_0}_{H^1}^2\big)\big(\norm{\overline\kappa_a}_{L^2}+\norm{\overline\kappa_b}_{L^2}+\norm{\kappa_0}_{L^2}+1\big)^2\\
&\qquad
 + c\norm{\overline\kappa_a-\overline\kappa_b}_{\dot H^1}\big(\norm{\overline\kappa_a}_{\dot H^1}+\norm{\overline\kappa_b}_{\dot H^1}+\norm{\kappa_0}_{H^1}\big) \,.
\end{aligned}
\end{equation}

It remains to estimate the viscoelastic contribution, which (weakly) satisfies 
 \begin{align*}
(\overline\tau^{\rm ve}_a-\overline\tau^{\rm ve}_b)_{ss}&-\frac{(\overline\kappa_a+\kappa_0)^2+(\overline\kappa_b+\kappa_0)^2}{2(1+\gamma)}(\overline\tau^{\rm ve}_a-\overline\tau^{\rm ve}_b) \\
 &=\frac{1}{1+\gamma}\bigg(\frac{1}{2}(\overline\kappa_a-\overline\kappa_b)(\overline\kappa_a+\overline\kappa_b+2\kappa_0)(\overline\tau^{\rm ve}_a+\overline\tau^{\rm ve}_b)\\
 &\qquad +\frac{\mu}{1+\mu}\bigg( (2+\gamma)\big((\overline\kappa_a-\overline\kappa_b)(\xi_a)_s+(\overline\kappa_b+\kappa_0)(\xi_a-\xi_b)_s\big)_s \\
 &\qquad\qquad - (\overline\kappa_a-\overline\kappa_b)_s(\xi_a)_s+ (\overline\kappa_b+\kappa_0)_s(\xi_a-\xi_b)_s \bigg)\bigg)\,.
 \end{align*}

In particular, we have
 \begin{align*}
\norm{\overline\tau^{\rm ve}_a-\overline\tau^{\rm ve}_b}_{H^1}^2&\le c\int_0^1\bigg((\overline\tau^{\rm ve}_a-\overline\tau^{\rm ve}_b)_s^2+\frac{(\overline\kappa_a+\kappa_0)^2+(\overline\kappa_b+\kappa_0)^2}{2(1+\gamma)}(\overline\tau^{\rm ve}_a-\overline\tau^{\rm ve}_b)^2\bigg)\,ds \\
 &\le c\bigg(\norm{\overline\kappa_a-\overline\kappa_b}_{L^2}(\norm{\overline\kappa_a}_{L^2}+\norm{\overline\kappa_b}_{L^2}+\norm{\kappa_0}_{L^2})\big(
\norm{\overline\kappa_a}_{\dot H^1}\norm{\xi_a}_{\dot H^1} +\norm{\overline\kappa_b}_{\dot H^1}\norm{\xi_b}_{\dot H^1} \\
&\quad +(\norm{\xi_a}_{\dot H^1}+\norm{\xi_b}_{\dot H^1})\norm{\kappa_0}_{H^1}\big)\norm{\overline\tau^{\rm ve}_a-\overline\tau^{\rm ve}_b}_{L^2}\\
 &\quad +\big(\norm{\overline\kappa_a-\overline\kappa_b}_{\dot H^1}\norm{\xi_a}_{\dot H^1}+(\norm{\overline\kappa_b}_{\dot H^1}+\norm{\kappa_0}_{H^1})\norm{\xi_a-\xi_b}_{\dot H^1}\big) \norm{\overline\tau^{\rm ve}_a-\overline\tau^{\rm ve}_b}_{H^1}  \bigg) \,.
 \end{align*}

Applying Young's inequality, we obtain
\begin{equation}\label{eq:tau_ve_lip}
\begin{aligned}
\norm{\overline\tau^{\rm ve}_a-\overline\tau^{\rm ve}_b}_{H^1}  
 &\le c\bigg(\norm{\overline\kappa_a-\overline\kappa_b}_{L^2}\big(\norm{\overline\kappa_a}_{L^2}+\norm{\overline\kappa_b}_{L^2}+\norm{\kappa_0}_{L^2}\big)\big(
\norm{\overline\kappa_a}_{\dot H^1}+\norm{\overline\kappa_b}_{\dot H^1}\\
&\qquad +\norm{\kappa_0}_{H^1}\big)(\norm{\xi_a}_{\dot H^1}+\norm{\xi_b}_{\dot H^1}) \\
&\qquad  +\norm{\overline\kappa_a-\overline\kappa_b}_{\dot H^1}\norm{\xi_a}_{\dot H^1}+(\norm{\overline\kappa_b}_{\dot H^1}+\norm{\kappa_0}_{H^1})\norm{\xi_a-\xi_b}_{\dot H^1}   \bigg) \,.
\end{aligned}
\end{equation}
Combining \eqref{eq:tau_nw_lip} and \eqref{eq:tau_ve_lip} yields \eqref{est:tau_lip}.
\end{proof}

\subsection{Evolution equation}\label{subsec:evolution}
We now proceed to the proof of Theorem \ref{thm:wellposed}.

\begin{proof}[Proof of Theorem \ref{thm:wellposed}]
We will consider the map
\begin{equation}\label{eq:Psi_map}
\begin{aligned}
\Psi\bigg[\begin{pmatrix}
\overline\kappa \\
\xi 
\end{pmatrix}\bigg] &= e^{\mc{A}t}\begin{pmatrix}
\overline\kappa^{\rm in} \\
\xi^{\rm in} 
\end{pmatrix} - \int_0^te^{\mc{A}(t-t')}\begin{pmatrix}
 \dot\kappa_0 \\
0
\end{pmatrix}\, dt'+ (1+\mu)\int_0^te^{\mc{A}(t-t')}\begin{pmatrix}
 \big(\mc{N}[\overline\kappa,\kappa_0] \big)_s \\
0
\end{pmatrix}\, dt' \\
&\hspace{3cm}
 - \mu(1+\gamma)\int_0^te^{\mc{A}(t-t')}\begin{pmatrix}
\big((\overline\kappa+\kappa_0)^2\xi_s \big)_s \\
0
\end{pmatrix}\, dt'
\end{aligned}
\end{equation}
and show that $\Psi$ admits a unique fixed point in a suitable function space. To construct such a function space, we first define the spaces
\begin{equation}\label{eq:Y0Y1}
\begin{aligned}
\mc{Y}_0 &= \big\{ u\in C([0,T];L^2(I))\,:\, \norm{u}_{\mc{Y}_0}<\infty \big\}\,, \quad \norm{\cdot}_{\mc{Y}_0}:= \sup_{t\in [0,T]} \norm{\cdot}_{L^2(I)}\\
\mc{Y}_1 &= \big\{ u\in C((0,T];\dot H^1(I))\,:\, \norm{u}_{\mc{Y}_1}<\infty \big\} \,, \quad \norm{\cdot}_{\mc{Y}_1}:=\sup_{t\in [0,T]} \min\{t^{1/4},1\}\,\norm{\cdot}_{\dot H^1(I)}\,.
\end{aligned}
\end{equation}
We close our contraction mapping argument for $(\overline\kappa,\xi)$ in $(\mc{Y}_0\times \mc{Y}_0) \cap (\mc{Y}_1\times \mc{Y}_1)$. \\

Given a function space $\mc{X}\times\mc{X}$, we will use the notation $B_{M}(\mc{X}\times\mc{X})$ to denote the closed ball in $\mc{X}\times\mc{X}$ of radius $M$, i.e.
\begin{equation}\label{eq:ball}
B_{M}(\mc{X}\times\mc{X}) = \bigg\{ \begin{pmatrix}
u\\
\phi \end{pmatrix}\in \mc{X}\times\mc{X} \,:\, \norm{\begin{pmatrix}
u\\
\phi \end{pmatrix}}_{\mc{X}\times\mc{X}}\le M \bigg\}\,.
\end{equation}

We first show that $\Psi$ maps $B_{M_0}(\mc{Y}_0\times\mc{Y}_0)\cap B_{M_1}(\mc{Y}_1\times\mc{Y}_1)$ into itself for some $M_1,M_0>0$.\\

Since the nonlinear terms $\mc{N}$ have the same form as in the Newtonian setting, from \cite{RFTpaper}, we have 
\begin{equation}\label{eq:Nl2est}
\begin{aligned}
\norm{\mc{N}[\overline\kappa,\kappa_0]}_{L^2(I)} &\le c\big(\norm{\overline\kappa}_{\dot H^1}^2+\norm{\kappa_0}_{H^1}^2+\norm{\overline\tau}_{H^1} \big)\big( \norm{\overline\kappa}_{\dot H^1}+\norm{\kappa_0}_{H^1}\big) \\
&\le c\bigg((\norm{\overline\kappa}_{\dot H^1}^3+\norm{\kappa_0}_{H^1}^3)(\norm{\overline\kappa}_{L^2}+\norm{\kappa_0}_{L^2}+1)+\norm{\xi}_{\dot H^1}(\norm{\overline\kappa}_{\dot H^1}^2+\norm{\kappa_0}_{H^1}^2) \bigg)\,.
\end{aligned}
\end{equation}
Here we have used the new viscoelastic tension estimate (Lemma \ref{lem:tension}) in the second line. Then, using Lemma \ref{lem:semigrp}, for $(\overline\kappa,\xi)\in B_{M_0}(\mc{Y}_0\times\mc{Y}_0)\cap B_{M_1}(\mc{Y}_1\times\mc{Y}_1)$, we may estimate the second forcing term of \eqref{eq:Psi_map} in $\dot H^m(I)\times \dot H^m(I)$, $m=0,1$, as 
\begin{equation}\label{eq:H3term1}
\begin{aligned}
&\norm{\int_0^te^{\mc{A}(t-t')}\begin{pmatrix}
 \big(\mc{N}[\overline\kappa,\kappa_0] \big)_s \\
0
\end{pmatrix}\, dt'}_{\dot H^m\times \dot H^m} \\
&\hspace{1cm}\le c\int_0^t 
\max\{(t-t')^{-(m+1)/4},1\}\, e^{-(t-t')\Lambda}\norm{\mc{N}[\overline\kappa,\kappa_0]}_{L^2}\,dt' \\
&\hspace{1cm}\le c\int_0^t
\max\{(t-t')^{-(m+1)/4},1\}\, \max\{(t')^{-3/4},1\}\,e^{-(t-t')\Lambda}\,dt' \, M_1^3 \big(M_0+M_1+1\big) \\ 
&\hspace{1cm}\le c\,
\max\{t^{-m/4},1\}\, M_1^3 \big(M_0+M_1+1\big)\,.
\end{aligned}
\end{equation}
Here we have also taken $\sup_{t\in[0,T]}\norm{\kappa_0}_{H^1(I)}\le cM_1$. Furthermore, using Lemma \ref{lem:semigrp}, we may estimate the third forcing term of \eqref{eq:Psi_map} in $\dot H^m(I)\times \dot H^m(I)$, $m=0,1$, by
\begin{equation}\label{eq:H3term2}
\begin{aligned}
&\norm{\int_0^te^{\mc{A}(t-t')}\begin{pmatrix}
 \big((\overline\kappa+\kappa_0)^2\xi_s \big)_s \\
0
\end{pmatrix}\, dt'}_{\dot H^m\times\dot H^m} \\
&\hspace{1cm}\le c\int_0^t
\max\{(t-t')^{-(m+1)/4},1\}\, e^{-(t-t')\Lambda}\norm{\xi}_{\dot H^1}\big( \norm{\overline\kappa}_{L^2}^2 + \norm{\kappa_0}_{L^2}^2\big)\,dt' \\
&\hspace{1cm}\le c\int_0^t
\max\{(t-t')^{-(m+1)/4},1\}\,\max\{(t')^{-1/4},1\} \,e^{-(t-t')\Lambda}\,dt' \, M_1(M_0^2+M_1^2) \\
&\hspace{1cm}\le c\,M_1(M_0^2+M_1^2)\,.
\end{aligned}
\end{equation}
Finally, the forcing term involving $\dot\kappa_0$ may be estimated in $\dot H^m(I)\times \dot H^m(I)$, $m=0,1$, as 
\begin{align*}
\norm{\int_0^t e^{\mc{A}(t-t')}\begin{pmatrix}
 \dot\kappa_0 \\
0
\end{pmatrix}\, dt'}_{\dot H^m\times\dot H^m} &\le c\int_0^t
\max\{(t-t')^{-m/4},1 \}\,e^{-(t-t')\Lambda}\norm{\dot\kappa_0}_{L^2(I)}\,dt' \\
&\le c\,\bigg(\sup_{t\in[0,T]}\norm{\dot\kappa_0}_{L^2(I)}\bigg)\,.
\end{align*}

Combining the above three estimates and using Lemma \ref{lem:semigrp} to estimate the initial data, we obtain the following $\mc{Y}_0\times\mc{Y}_0$ bound: 
\begin{equation}\label{eq:Y0Y0bd}
\begin{aligned}
\norm{\Psi\bigg[\begin{pmatrix}
\overline\kappa \\
\xi 
\end{pmatrix}\bigg]}_{\mc{Y}_0\times\mc{Y}_0} &\le c\bigg( \norm{\begin{pmatrix}
\overline\kappa^{\rm in}\\
\xi^{\rm in}
\end{pmatrix}}_{L^2\times L^2} + M_1^3(M_0+M_1+1) \\
&\qquad + M_1(M_0^2+M_1^2) + \sup_{t\in[0,T]}\norm{\dot\kappa_0}_{L^2} \bigg) 
\le M_0\,,
\end{aligned}
\end{equation}
provided that we choose $M_0=c_0\norm{\begin{pmatrix}
\overline\kappa^{\rm in}\\
\xi^{\rm in}
\end{pmatrix}}_{L^2\times L^2}$ for $c_0$ small enough, $c\,\sup_{t\in[0,T]}\norm{\dot\kappa_0}_{L^2}\le M_0/4$, and $M_1$ small enough that $c\big(M_1^3(M_0+M_1+1)+M_1(M_0^2+M_1^2)\big)\le M_0/4$. \\

We may also obtain the following $\mc{Y}_1\times\mc{Y}_1$ bound for $\Psi$: 
\begin{equation}\label{eq:Y1Y1bd}
\begin{aligned}
\norm{\Psi\bigg[\begin{pmatrix}
\overline\kappa \\
\xi 
\end{pmatrix}\bigg]}_{\mc{Y}_1\times\mc{Y}_1} &\le 
\sup_{t\in[0,T]}\min\{t^{1/4},1\} \norm{e^{t\mc{A}}\begin{pmatrix}
\overline\kappa^{\rm in} \\
\xi^{\rm in} 
\end{pmatrix}}_{\dot H^1\times \dot H^1} + c\bigg(M_1^3(M_0+M_1+1) \\
&\qquad +M_1(M_0^2+M_1^2) +\sup_{t\in[0,T]}\norm{\dot\kappa_0}_{L^2} \bigg) \\
&\le \sup_{t\in[0,T]}\min\{t^{1/4},1\} \norm{e^{t\mc{A}}\begin{pmatrix}
\overline\kappa^{\rm in} \\
\xi^{\rm in} 
\end{pmatrix}}_{\dot H^1\times \dot H^1} + \frac{M_1}{2}\,,
\end{aligned}
\end{equation}
provided that $c\,\sup_{t\in[0,T]}\norm{\dot\kappa_0}_{L^2}\le M_1/4$, and $M_0$ and $M_1$ are small enough that $c\big(M_1^3(M_0+M_1+1)+M_1(M_0^2+M_1^2)\big)\le M_1/4$.\\

It remains to show that 
\begin{equation}\label{eq:init_data_bd}
\sup_{t\in[0,T]}\min\{t^{1/4},1\} \norm{e^{t\mc{A}}\begin{pmatrix}
\overline\kappa^{\rm in} \\
\xi^{\rm in} 
\end{pmatrix}}_{\dot H^1\times \dot H^1} \le \frac{M_1}{2}\,,
\end{equation}
which we may achieve by either choosing a small time interval $T$ or small initial data. 
For small time, we may use Lemma \ref{lem:smalltime} to find $T_{M_1}>0$ such that \eqref{eq:init_data_bd} holds. 
For small initial data, we may use Lemma \ref{lem:semigrp} to obtain 
\begin{align*}
\sup_{t\in[0,T]}\min\{t^{1/4},1\} \norm{e^{t\mc{A}}\begin{pmatrix}
\overline\kappa^{\rm in} \\
\xi^{\rm in} 
\end{pmatrix}}_{\dot H^1\times \dot H^1} 
\le c \,\norm{\begin{pmatrix}
\overline\kappa^{\rm in}\\
\xi^{\rm in}
\end{pmatrix}}_{L^2\times L^2}\,,
\end{align*}
and for sufficiently small $(\overline\kappa^{\rm in},\xi^{\rm in})$, we may take $M_1=c_1\norm{\begin{pmatrix}
\overline\kappa^{\rm in}\\
\xi^{\rm in}
\end{pmatrix}}_{L^2\times L^2}$ to obtain the bound \eqref{eq:init_data_bd}. \\

We next show that the map $\Psi$ is a contraction on $B_{M_0}(\mc{Y}_0\times\mc{Y}_0)\cap B_{M_1}(\mc{Y}_1\times\mc{Y}_1)$. Given two pairs $(\overline\kappa_a,\xi_a)$, $(\overline\kappa_b,\xi_b)$, we seek an estimate for $\Psi\bigg[\begin{pmatrix}
\overline\kappa_a \\
\xi_a 
\end{pmatrix}\bigg]-\Psi\bigg[\begin{pmatrix}
\overline\kappa_b \\
\xi_b 
\end{pmatrix}\bigg]$. \\

First, from \cite{RFTpaper}, we may borrow the estimate
\begin{equation}\label{eq:lip_est1}
\begin{aligned}
&\norm{\mc{N}[\overline\kappa_a(\cdot,t)]-\mc{N}[\overline\kappa_b(\cdot,t)]}_{L^2(I)} \\
&\le c\bigg(\norm{\overline\kappa_a-\overline\kappa_b}_{\dot H^1}\big(\norm{\overline\kappa_a}_{\dot H^1}^2+\norm{\overline\kappa_b}_{\dot H^1}^2+\norm{\kappa_0}_{H^1}^2\big)  \\
&\qquad  +\norm{\overline\tau_a-\overline\tau_b}_{H^1}\big(\norm{\overline\kappa_a}_{\dot H^1}+\norm{\kappa_0}_{H^1}\big) + \norm{\overline\kappa_a-\overline\kappa_b}_{\dot H^1}\norm{\overline\tau_b}_{H^1}\bigg) \\
&\le 
 c\bigg(\norm{\overline\kappa_a-\overline\kappa_b}_{\dot H^1}+\norm{\overline\kappa_a-\overline\kappa_b}_{L^2}\big(\norm{\overline\kappa_a}_{\dot H^1}+\norm{\kappa_0}_{H^1}\big)\bigg)\bigg( \norm{\overline\kappa_a}_{\dot H^1}^2+ \norm{\overline\kappa_b}_{\dot H^1}^2+\norm{\kappa_0}_{H^1}^2
\\
&\qquad +(\norm{\xi_a}_{\dot H^1}+\norm{\xi_b}_{\dot H^1})\big(\norm{\overline\kappa_a}_{\dot H^1}+\norm{\overline\kappa_b}_{\dot H^1}+\norm{\kappa_0}_{H^1}\big)\bigg)\big(\norm{\overline\kappa_a}_{L^2}+\norm{\overline\kappa_b}_{L^2}+\norm{\kappa_0}_{L^2}+1\big)^2 \\
&\quad +c\norm{\xi_a-\xi_b}_{\dot H^1}\big(\norm{\overline\kappa_a}_{\dot H^1}+\norm{\kappa_0}_{H^1}\big)(\norm{\overline\kappa_b}_{\dot H^1}+\norm{\kappa_0}_{H^1})\,.
\end{aligned}
\end{equation}
Here we have again used the new viscoelastic estimates of Lemma \ref{lem:tension} to bound the tension in the second inequality. Furthermore, we have the following Lipschitz bound for the new viscoelastic nonlinear term:  
\begin{equation}\label{eq:lip_est2}
\begin{aligned}
&\norm{(\overline\kappa_a+\kappa_0)^2(\xi_a)_s - (\overline\kappa_b+\kappa_0)^2(\xi_b)_s}_{L^2(I)} \\
&\qquad \le \norm{\xi_a-\xi_b}_{\dot H^1}(\norm{\overline\kappa_a}_{\dot H^1}^2+\norm{\kappa_0}_{H^1}^2) + \norm{\overline\kappa_a-\overline\kappa_b}_{\dot H^1}\norm{\xi_b}_{\dot H^1}(\norm{\overline\kappa_a}_{\dot H^1}+\norm{\overline\kappa_b}_{\dot H^1}+\norm{\kappa_0}_{H^1}) \,.
\end{aligned}
\end{equation}

Together, we may then obtain the Lipschitz estimate 
\begin{align*}
&\norm{\Psi\bigg[\begin{pmatrix}
\overline\kappa_a \\
\xi_a 
\end{pmatrix}\bigg]-\Psi\bigg[\begin{pmatrix}
\overline\kappa_b \\
\xi_b 
\end{pmatrix}\bigg]}_{\dot H^m\times \dot H^m} \\
&\quad \le c\int_0^t
\max\{(t-t')^{-(m+1)/4},1\}\, e^{-(t-t')\Lambda}\big(\norm{\mc{N}[\overline\kappa_a(\cdot,t)]-\mc{N}[\overline\kappa_b(\cdot,t)]}_{L^2(I)} \\
&\hspace{3cm}+ \norm{(\overline\kappa_a+\kappa_0)^2(\xi_a)_s - (\overline\kappa_b+\kappa_0)^2(\xi_b)_s}_{L^2(I)}\big)\,dt' \\
&\quad\le c\int_0^t
\max\{(t-t')^{-(m+1)/4},1\}\, e^{-(t-t')\Lambda}\bigg(\norm{\overline\kappa_a-\overline\kappa_b}_{\dot H^1}\\
&\qquad+\norm{\overline\kappa_a-\overline\kappa_b}_{L^2}\big(\norm{\overline\kappa_a}_{\dot H^1}+\norm{\kappa_0}_{H^1}\big)\bigg)\bigg( (\norm{\xi_a}_{\dot H^1}+\norm{\xi_b}_{\dot H^1})\big(\norm{\overline\kappa_a}_{\dot H^1}+\norm{\overline\kappa_b}_{\dot H^1}+\norm{\kappa_0}_{H^1}\big)
\\
&\qquad +\norm{\overline\kappa_a}_{\dot H^1}^2+ \norm{\overline\kappa_b}_{\dot H^1}^2+\norm{\kappa_0}_{H^1}^2\bigg)\big(\norm{\overline\kappa_a}_{L^2}+\norm{\overline\kappa_b}_{L^2}+\norm{\kappa_0}_{L^2}+1\big)^2\,dt' \\
&\qquad + c\int_0^t
\max\{(t-t')^{-(m+1)/4},1\}\, e^{-(t-t')\Lambda} \norm{\xi_a-\xi_b}_{\dot H^1}(\norm{\overline\kappa_a}_{\dot H^1}^2+\norm{\overline\kappa_b}_{\dot H^1}^2+\norm{\kappa_0}_{H^1}^2)\, dt' \\
&\quad \le c\int_0^t
\max\{(t-t')^{-(m+1)/4},1\}\,\max\{(t')^{-3/4},1\}\, e^{-(t-t')\Lambda}\,dt'\,M_1^2\bigg(\norm{\xi_a-\xi_b}_{\mc{Y}_1}\\
&\qquad + M_0^2\norm{\overline\kappa_a-\overline\kappa_b}_{\mc{Y}_1}+M_1M_0^2\norm{\overline\kappa_a-\overline\kappa_b}_{\mc{Y}_0}\bigg) \\
&\quad \le c\,
\max\{t^{-m/4},1 \}\,M_1^2\bigg(\norm{\xi_a-\xi_b}_{\mc{Y}_1}+M_0^2\norm{\overline\kappa_a-\overline\kappa_b}_{\mc{Y}_1}+M_1M_0^2\norm{\overline\kappa_a-\overline\kappa_b}_{\mc{Y}_0}\bigg)\,.
\end{align*}

We thus have 
\begin{align*}
&\norm{\Psi\bigg[\begin{pmatrix}
\overline\kappa_a \\
\xi_a 
\end{pmatrix}\bigg]-\Psi\bigg[\begin{pmatrix}
\overline\kappa_b \\
\xi_b 
\end{pmatrix}\bigg]}_{\mc{Y}_0\times \mc{Y}_0} \\
&\qquad \le c\,M_1^2\bigg(\norm{\xi_a-\xi_b}_{\mc{Y}_1}+M_0^2\norm{\overline\kappa_a-\overline\kappa_b}_{\mc{Y}_1}+M_1M_0^2\norm{\overline\kappa_a-\overline\kappa_b}_{\mc{Y}_0}\bigg)\,,\\ 
&\norm{\Psi\bigg[\begin{pmatrix}
\overline\kappa_a \\
\xi_a 
\end{pmatrix}\bigg]-\Psi\bigg[\begin{pmatrix}
\overline\kappa_b \\
\xi_b 
\end{pmatrix}\bigg]}_{\mc{Y}_1\times \mc{Y}_1} \\
&\qquad \le c\,M_1^2\bigg(\norm{\xi_a-\xi_b}_{\mc{Y}_1}+M_0^2\norm{\overline\kappa_a-\overline\kappa_b}_{\mc{Y}_1}+M_1M_0^2\norm{\overline\kappa_a-\overline\kappa_b}_{\mc{Y}_0}\bigg)\,.
\end{align*}
For sufficiently small $M_0,M_1<1$, we obtain a contraction on $B_{M_0}(\mc{Y}_0\times\mc{Y}_0)\cap B_{M_1}(\mc{Y}_1\times\mc{Y}_1)$, thus proving Theorem \ref{thm:wellposed} for $\kappa_0\not\equiv0$. \\

If $\kappa_0\equiv0$, we may replace the norms in the definition \eqref{eq:Y0Y1} of $\mc{Y}_0$ and $\mc{Y}_1$ with the exponentially weighted norms
\begin{equation}\label{eq:Y0Y1tilde}
\norm{\cdot}_{\wh{\mc{Y}_0}}:= \sup_{t\in [0,T]} e^{-t\Lambda}\norm{\cdot}_{L^2(I)} \,, \quad \norm{\cdot}_{\wh{\mc{Y}_1}}:=\sup_{t\in [0,T]} \min\{t^{1/4},1\}\,e^{-t\Lambda}\norm{\cdot}_{\dot H^1(I)}\,,
\end{equation}
where $\Lambda$ is given by Lemma \ref{lem:semigrp}. We obtain analogous estimates to \eqref{eq:Y0Y0bd} and \eqref{eq:Y1Y1bd} in $\wh{\mc{Y}_0}\times\wh{\mc{Y}_0}$ and $\wh{\mc{Y}_1}\times\wh{\mc{Y}_1}$, except, crucially, no term depending on $\dot\kappa_0$, allowing for the desired time decay. 
\end{proof}

\subsection{Existence of a unique periodic solution}\label{subsec:periodic}
We next consider solutions to the system \eqref{eq:kappadot}-\eqref{eq:BCs} when the internal fiber forcing $\kappa_0$ is $T$-periodic in time. We prove Theorem \ref{thm:per_deb} in two parts: in this section, we prove parts (a) and (b) on the existence of a unique periodic solution $(\overline\kappa,\xi)$, and in section \ref{subsec:deborah} we show part (c) concerning the limiting behavior of this periodic solution as $\delta\to 0$. \\

To prove parts (a) and (b) of Theorem \ref{thm:per_deb}, we begin by considering the following linear PDE, where $\kappa_0$ and $g$ are both given, $T$-periodic functions:
\begin{equation}\label{eq:linear_per}
\begin{aligned}
\dot{\overline\kappa}&= -(1+\mu)\overline\kappa_{ssss}+\mu\xi_{ssss} - \dot\kappa_0+g_s \\
\delta\dot\xi &= -\xi+\overline\kappa\,,
\end{aligned}
\end{equation}
with boundary conditions as in \eqref{eq:BCs}. For the system \eqref{eq:linear_per}, we show the following lemma. 

\begin{lemma}\label{lem:lin_per}
There exists a constant $\varepsilon>0$ such that, given a $T$-periodic $\kappa_0\in C^1([0,T];L^2(I))$ satisfying 
\begin{equation}\label{est:dotkap0}
\sup_{t\in[0,T]} \norm{\dot\kappa_0}_{L^2} = \varepsilon_1 \le \varepsilon
\end{equation}
and a $T$-periodic $g(s,t)\in C([0,T];L^2(I))$ satisfying 
\begin{equation}\label{est:gfun}
\sup_{t\in[0,T]} \norm{g}_{L^2} = \varepsilon_2 \le \varepsilon \,,
\end{equation}
there exists a unique $T$-periodic solution to \eqref{eq:linear_per} satisfying
\begin{equation}\label{eq:lin_per_bd}
\sup_{t\in[0,T]}\norm{\begin{pmatrix}
\overline\kappa \\
\xi
\end{pmatrix}}_{H^1\times H^1} \le c\bigg(\sup_{t\in[0,T]}\norm{\dot\kappa_0}_{L^2}+\sup_{t\in[0,T]}\norm{g}_{L^2} \bigg) \,.
\end{equation}
\end{lemma}

\begin{proof}
We consider the map $\Psi^T$ taking the initial data $(\overline\kappa^{\rm in},\xi^{\rm in})$ to the solution to \eqref{eq:linear_per} at time $T$, which may be written as 
\begin{align*}
\Psi^T\bigg[ \begin{pmatrix}
\overline\kappa^{\rm in} \\
\xi^{\rm in}
\end{pmatrix} \bigg] &= e^{T\mc{A}}\begin{pmatrix}
\overline\kappa^{\rm in} \\
\xi^{\rm in}
\end{pmatrix} -\int_0^Te^{(T-t')\mc{A}}\begin{pmatrix}
\dot\kappa_0 \\
0
\end{pmatrix}\, dt' + \int_0^Te^{(T-t')\mc{A}}\begin{pmatrix}
g_s \\
0
\end{pmatrix}\, dt'\,.
\end{align*}

We show that $\Psi^T$ maps $B_M(H^1\times H^1)$ to itself, where $B_M(H^1\times H^1)$ is as in \eqref{eq:ball}. Using Lemma \ref{lem:semigrp}, for $m=0,1$, we have 
\begin{equation}\label{eq:timeTest1}
\begin{aligned}
\norm{\Psi^T\bigg[ \begin{pmatrix}
\overline\kappa^{\rm in} \\
\xi^{\rm in}
\end{pmatrix} \bigg]}_{\dot H^m\times \dot H^m} 
&\le c\,e^{-T\Lambda}\norm{\begin{pmatrix}
\overline\kappa^{\rm in}\\
\xi^{\rm in}
\end{pmatrix}}_{\dot H^m\times \dot H^m}  +\int_0^T\norm{e^{(T-t')\mc{A}}\begin{pmatrix}
\dot\kappa_0 \\
0
\end{pmatrix}}_{\dot H^m\times \dot H^m}\, dt' \\
&\quad\qquad + \int_0^T\norm{e^{(T-t')\mc{A}}\begin{pmatrix}
g_s \\
0
\end{pmatrix}}_{\dot H^m\times \dot H^m}\, dt' \\
&\le c\,e^{-T\Lambda}\norm{\begin{pmatrix}
\overline\kappa^{\rm in}\\
\xi^{\rm in}
\end{pmatrix}}_{\dot H^m\times \dot H^m}\\
&\qquad +c\int_0^T
\max\{(T-t')^{-m/4},1\} \,
e^{-(T-t')\Lambda}\norm{\dot\kappa_0}_{L^2}\, dt'\\
&\qquad + \int_0^T
\max\{(T-t')^{-(m+1)/4},1\} \,e^{-(T-t')\Lambda}\norm{g}_{L^2}\, dt' \\
&\le c\bigg(e^{-T\Lambda}\norm{\begin{pmatrix}
\overline\kappa^{\rm in}\\
\xi^{\rm in}
\end{pmatrix}}_{\dot H^m\times \dot H^m} + \sup_{t\in [0,T]}\norm{\dot\kappa_0}_{L^2} + \sup_{t\in [0,T]}\norm{g}_{L^2} \bigg)\,.
\end{aligned}
\end{equation}

In particular, provided that the period $T$ is large enough that $c\,e^{-T\Lambda}\le \frac{1}{3}$, we may choose $\kappa_0$ and $g$ such that $c\,\sup_{t\in [0,T]}\norm{\dot\kappa_0}_{L^2}\le \frac{M}{3}$ and $c\,\sup_{t\in [0,T]}\norm{g}_{L^2}\le \frac{M}{3}$ to obtain
\begin{equation}\label{eq:timeTest2}
\norm{\Psi^T\bigg[ \begin{pmatrix}
\overline\kappa^{\rm in} \\
\xi^{\rm in}
\end{pmatrix} \bigg]}_{H^1\times H^1} \le \bigg(\frac{M}{3} + \frac{M}{3} +\frac{M}{3}\bigg) 
\le M \,.
\end{equation}
 
Furthermore, again using Lemma \ref{lem:semigrp}, we may obtain the following Lipschitz estimate:
\begin{align*}
&\norm{\Psi^T\bigg[ \begin{pmatrix}
\overline\kappa^{\rm in}_a \\
\xi^{\rm in}_a
\end{pmatrix} \bigg]-\Psi^T\bigg[ \begin{pmatrix}
\overline\kappa^{\rm in}_b \\
\xi^{\rm in}_b
\end{pmatrix} \bigg]}_{H^1\times H^1} 
= \norm{e^{T\mc{A}}\begin{pmatrix}
\overline\kappa^{\rm in}_a-\overline\kappa^{\rm in}_b \\
\xi^{\rm in}_a-\xi^{\rm in}_b
\end{pmatrix}}_{H^1\times H^1} \\
&\hspace{4cm} \le c\,e^{-T\Lambda}\norm{\begin{pmatrix}
\overline\kappa^{\rm in}_a-\overline\kappa^{\rm in}_b\\
\xi^{\rm in}_a-\xi^{\rm in}_b
\end{pmatrix}}_{H^1\times H^1}  
 \le \frac{1}{4}\norm{\begin{pmatrix}
\overline\kappa^{\rm in}_a-\overline\kappa^{\rm in}_b\\
\xi^{\rm in}_a-\xi^{\rm in}_b
\end{pmatrix}}_{H^1\times H^1} \,,
\end{align*}
as long as the period $T$ is sufficiently large. By the contraction mapping theorem, there exists a unique fixed point of the map $\Psi^T$, i.e. $\Psi^T\bigg[ \begin{pmatrix}
\overline\kappa^{\rm in} \\
\xi^{\rm in}
\end{pmatrix}\bigg] = \begin{pmatrix}
\overline\kappa^{\rm in} \\
\xi^{\rm in}
\end{pmatrix}$, corresponding to a unique $T$-periodic solution $(\overline\kappa,\xi)$ to \eqref{eq:linear_per}. \\

In addition, using \eqref{eq:timeTest1} and \eqref{eq:timeTest2}, the $T$-periodic solution $(\overline\kappa,\xi)$ satisfies
\begin{equation}\label{eq:inbound}
\norm{\begin{pmatrix}
\overline\kappa^{\rm in} \\
\xi^{\rm in} 
\end{pmatrix}}_{H^1\times H^1} \le c\bigg(\sup_{t\in[0,T]}\norm{\dot\kappa_0}_{L^2}+\sup_{t\in[0,T]}\norm{g}_{L^2} \bigg) \,.
\end{equation}

To obtain the bound \eqref{eq:lin_per_bd}, we may use Duhamel's formula to write $(\overline\kappa,\xi)$ as
\begin{align*}
\begin{pmatrix}
\overline\kappa \\
\xi 
\end{pmatrix} &= e^{\mc{A}t}\begin{pmatrix}
\overline\kappa^{\rm in} \\
\xi^{\rm in} 
\end{pmatrix} - \int_0^te^{\mc{A}(t-t')}\begin{pmatrix}
 \dot\kappa_0 \\
0
\end{pmatrix}\, dt'+ \int_0^te^{\mc{A}(t-t')}\begin{pmatrix}
 g_s \\
0
\end{pmatrix}\, dt' \,.
\end{align*}

Then, as for the time-$T$ map \eqref{eq:timeTest1}, but now for any $t\in[0,T]$ and $m=0,1$, we have
\begin{align*}
\norm{\begin{pmatrix}
\overline\kappa \\
\xi 
\end{pmatrix}}_{\dot H^m\times \dot H^m} 
&\le c\,e^{-t\Lambda}\norm{\begin{pmatrix}
\overline\kappa^{\rm in} \\
\xi^{\rm in} 
\end{pmatrix}}_{\dot H^m\times \dot H^m} 
 +c\int_0^t\max\{(t-t')^{-m/4},1\} \,e^{-(t-t')\Lambda}\norm{\dot\kappa_0}_{L^2}\, dt' \\
&\qquad + \int_0^t\max\{(t-t')^{-(m+1)/4},1\}\,e^{-(t-t')\Lambda}\norm{g}_{L^2}\, dt' \\
&\le c\bigg(e^{-t\Lambda}\norm{\begin{pmatrix}
\overline\kappa^{\rm in} \\
\xi^{\rm in} 
\end{pmatrix}}_{\dot H^m\times \dot H^m} + \sup_{t\in [0,T]}\norm{\dot\kappa_0}_{L^2} + \sup_{t\in [0,T]}\norm{g}_{L^2} \bigg)\,.
\end{align*}
Using \eqref{eq:inbound}, we obtain \eqref{eq:lin_per_bd}.
\end{proof}

We now show parts (a) and (b) of Theorem \ref{thm:per_deb}. 

\begin{proof}[Proof of Theorem \ref{thm:per_deb}, parts (a) \& (b)]
We will use Lemma \ref{lem:lin_per}.
Let $\mc{A}_T^{\rm per}$ denote the solution operator mapping $(-\dot\kappa_0+g_s,0)^{\rm T}$ to the unique periodic $(\overline\kappa,\xi)$:
\begin{align*}
\begin{pmatrix}
\overline\kappa \\
\xi
\end{pmatrix} &= \mc{A}^{\rm per}_T\bigg[\begin{pmatrix}
-\dot\kappa_0+g_s\\
0
\end{pmatrix}\bigg]\,.
\end{align*}

We will consider $g(\overline\kappa,\xi,\kappa_0)= (1+\mu)\mc{N}[\overline\kappa,\kappa_0] - \mu(1+\gamma)(\overline\kappa+\kappa_0)^2\xi_s$, i.e. the nonlinear terms from \eqref{eq:kappadot}, and show that, given $\kappa_0$, the operator $\mc{A}^{\rm per}_T$ admits a unique fixed point in the space $\mc{X}_T\times \mc{X}_T$, where 
\begin{align*}
\mc{X}_T := \big\{ u\in C([0,T];H^1(I)) \, : \, u \text{ is $T$-periodic} \big\}\,,\quad \norm{\cdot}_{\mc{X}_T}:= \sup_{t\in[0,T]}\norm{u}_{H^1(I)}\,.
\end{align*}

We show that $\mc{A}_T^{\rm per}$ maps the ball $B_M(\mc{X}_T,\mc{X}_T)$ \eqref{eq:ball} to itself for some $M>0$. For $g$ as above, we first note that, by \eqref{eq:Nl2est}, we have
\begin{equation}\label{eq:g_bound}
\norm{g}_{L^2(I)} \le c\bigg((\norm{\overline\kappa}_{H^1}^3+\norm{\kappa_0}_{H^1}^3)(\norm{\overline\kappa}_{L^2}+\norm{\kappa_0}_{L^2}+1)+\norm{\xi}_{H^1}(\norm{\overline\kappa}_{H^1}^2+\norm{\kappa_0}_{H^1}^2) \bigg)\,.
\end{equation}

Then, using Lemma \ref{lem:lin_per}, taking $\kappa_0$ such that $\norm{\kappa_0}_{\mc{X}_T}=c_1 M$, for $(\overline\kappa,\xi)\in B_M(\mc{X}_T,\mc{X}_T)$ we have
\begin{equation}\label{eq:ATper_est}
\begin{aligned}
\norm{\mc{A}^{\rm per}_T\bigg[\begin{pmatrix}
-\dot\kappa_0+g_s\\
0
\end{pmatrix}\bigg]}_{\mc{X}_T\times\mc{X}_T} 
&\le c\bigg(\sup_{t\in[0,T]}\norm{\dot\kappa_0}_{L^2}+\norm{\overline\kappa}_{\mc{X}_T}^4+\norm{\kappa_0}_{\mc{X}_T}^4+ \norm{\overline\kappa}_{\mc{X}_T}^3+\norm{\kappa_0}_{\mc{X}_T}^3\\
&\qquad+\norm{\xi}_{\mc{X}_T}(\norm{\overline\kappa}_{\mc{X}_T}^2+\norm{\kappa_0}_{\mc{X}_T}^2) \bigg)\\
&\le \bigg(\frac{M}{2} + c(M^4+M^3)\bigg) \le M\,,
\end{aligned}
\end{equation}
where we have taken $c\,\sup_{t\in[0,T]}\norm{\dot\kappa_0}_{L^2}= c_2M$ for $c_2\le \frac{1}{2}$ and $M$ sufficiently small. \\

To show that $\mc{A}_T^{\rm per}$ is a contraction on $B_M(\mc{X}_T,\mc{X}_T)$, we note that from \eqref{eq:lip_est1} and \eqref{eq:lip_est2}, given two pairs $(\overline\kappa_a,\xi_a)$, $(\overline\kappa_b,\xi_b)$ and defining $g_a=g(\overline\kappa_a,\xi_a,\kappa_0)$, $g_b=g(\overline\kappa_b,\xi_b,\kappa_0)$, we have 
\begin{align*}
\norm{g_a - g_b}_{L^2} &\le
 c\bigg(\norm{\overline\kappa_a-\overline\kappa_b}_{H^1}+\norm{\overline\kappa_a-\overline\kappa_b}_{L^2}\big(\norm{\overline\kappa_a}_{H^1}+\norm{\kappa_0}_{H^1}\big)\bigg)\bigg( \norm{\overline\kappa_a}_{H^1}^2+ \norm{\overline\kappa_b}_{H^1}^2+\norm{\kappa_0}_{H^1}^2
\\
&\quad +(\norm{\xi_a}_{H^1}+\norm{\xi_b}_{H^1})\big(\norm{\overline\kappa_a}_{H^1}+\norm{\overline\kappa_b}_{H^1}+\norm{\kappa_0}_{H^1}\big)\bigg)\big(\norm{\overline\kappa_a}_{L^2}+\norm{\overline\kappa_b}_{L^2}+\norm{\kappa_0}_{L^2}+1\big)^2 \\
&\quad +c\norm{\xi_a-\xi_b}_{H^1}\big(\norm{\overline\kappa_a}_{H^1}^2+\norm{\overline\kappa_b}_{H^1}^2+\norm{\kappa_0}_{H^1}^2)\,.
\end{align*}

We then have
\begin{align*}
\norm{\mc{A}^{\rm per}_T\bigg[\begin{pmatrix}
(g_a-g_b)_s\\
0
\end{pmatrix}\bigg]}_{\mc{X}_T\times\mc{X}_T} 
%
%
%
%
&\le c\bigg(M^2\big(M^3+1\big)\norm{\overline\kappa_a-\overline\kappa_b}_{\mc{X}_T} + M^2\norm{\xi_a-\xi_b}_{\mc{X}_T}\bigg) \\
&\le \frac{1}{4}\norm{\begin{pmatrix}
\overline\kappa_a-\overline\kappa_b\\
\xi_a-\xi_b
\end{pmatrix}}_{\mc{X}_T\times\mc{X}_T}
\end{align*}
for $M$ sufficiently small, yielding a contraction on $B_M(\mc{X}_T,\mc{X}_T)$.\\

For part (b) of Theorem \ref{thm:per_deb}, we note that $(\overline\kappa^{\rm lin},\xi^{\rm lin})$ is the solution to \eqref{eq:linear_per} with $g=0$. In particular, using the bound \eqref{eq:g_bound} on $g(\overline\kappa,\xi,\kappa_0)= (1+\mu)\mc{N}[\overline\kappa,\kappa_0] - \mu(1+\gamma)(\overline\kappa+\kappa_0)^2\xi_s$ and the estimate \eqref{eq:ATper_est}, for the periodic solution $(\overline\kappa,\xi)$ of part (a), we have 
\begin{align*}
\norm{\begin{pmatrix}
\overline \kappa- \overline\kappa^{\rm lin}\\
\xi - \xi^{\rm lin}
\end{pmatrix}}_{\mc{X}_T\times\mc{X}_T} &=\norm{\mc{A}^{\rm per}_T\bigg[\begin{pmatrix}
-\dot\kappa_0+g_s\\
0
\end{pmatrix}-\mc{A}^{\rm per}_T\bigg[\begin{pmatrix}
-\dot\kappa_0\\
0
\end{pmatrix}\bigg]}_{\mc{X}_T\times\mc{X}_T} \\
&\le c\bigg(\norm{\overline\kappa}_{\mc{X}_T}^4+\norm{\kappa_0}_{\mc{X}_T}^4+ \norm{\overline\kappa}_{\mc{X}_T}^3+\norm{\kappa_0}_{\mc{X}_T}^3+\norm{\xi}_{\mc{X}_T}(\norm{\overline\kappa}_{\mc{X}_T}^2+\norm{\kappa_0}_{\mc{X}_T}^2) \bigg)\\
&\le c\,\varepsilon^3\,.
\end{align*}
\end{proof}

\subsection{Small relaxation time limit}\label{subsec:deborah}
We next show part (c) of Theorem \ref{thm:per_deb} concerning the behavior of the periodic solution $(\overline\kappa,\xi)$ of part (a) as the relaxation time $\delta\to 0$. To show that $(\overline\kappa,\xi)$ satisfies the estimate \eqref{est:small_delta}, we need the following lemma.

\begin{lemma}\label{lem:deltadiff}
For $f\in L^2(I)$, let 
\begin{equation}\label{eq:uphi_form}
\begin{pmatrix}
u_j\\
\phi_j
\end{pmatrix} = e^{t\mc{A}}\begin{pmatrix}
\p_s^jf\\
0
\end{pmatrix}\, \quad j=0,1\,.
\end{equation}
Then for $0\le m\le 4-j$ and $t\in(0,T]$, we have 
\begin{equation}
\norm{u_j-\phi_j}_{\dot H^m} \le c\, \delta^{1-(j+m)/4}\,\sup_{k}\bigg(\abs{\nu_k^-}e^{t\nu_k^-} + \abs{\nu_k^+}e^{t\nu_k^+} \bigg) \norm{f}_{L^2(I)}
\end{equation}
where 
\begin{equation}\label{eq:A_eval2}
\nu_k^\pm = \frac{1}{2\delta}\bigg(-\big(\delta(1+\mu)\lambda_k+1\big)\pm\sqrt{\big(\delta(1+\mu)\lambda_k+1\big)^2-4\delta\lambda_k} \bigg)
\end{equation}
are the eigenvalues of the matrix $\wt{\mc{A}}_k$ defined in \eqref{eq:Atildek}.
\end{lemma}

\begin{proof}
It suffices to show that 
\begin{align*}
\begin{pmatrix}
w_j\\
\varphi_j
\end{pmatrix}&=\mc{L}^{j/4}e^{t\mc{A}}\begin{pmatrix}
f\\
0
\end{pmatrix}
\end{align*}
satisfies
\begin{align*}
\norm{w_j-\varphi_j}_{\dot H^m} \le c\, \delta^{1-(j+m)/4}\,\sup_{k}\bigg(\abs{\nu_k^-}e^{t\nu_k^-} + \abs{\nu_k^+}e^{t\nu_k^+} \bigg) \norm{f}_{L^2(I)}\,;
\end{align*}
Lemma \ref{lem:deltadiff} then follows by a duality argument as in the proof of Lemma \ref{lem:semigrp}.\\

We begin by recalling the decomposition \eqref{eq:A_decomp} in terms of eigenvectors of $\wt{\mc{A}}_k$; in particular
\begin{align*}
\begin{pmatrix}
1\\
0
\end{pmatrix} = a_k^{-}\bv^{-}_k + a_k^{+}\bv^{+}_k\,,\quad  a_k^{-}\bv_k^{-}&=-\frac{1}{\delta(\nu_k^{+}-\nu_k^{-})}\begin{pmatrix}
1+\delta \nu_k^{-} \\
1
\end{pmatrix} \,, \quad a_k^{+}\bv_k^{+}=\frac{1}{\delta(\nu_k^{+}-\nu_k^{-})}\begin{pmatrix}
1+\delta \nu_k^{+}\\
1
\end{pmatrix}\,,
\end{align*}
where
\begin{align*}
\delta(\nu_k^{+}-\nu_k^{-}) = \sqrt{(\delta(1+\mu)\lambda_k+1)^2-4\delta\lambda_k} \,.
\end{align*}

Note that for $0\le r\le 1$, we have
\begin{align*}
\frac{(\delta\lambda_k)^r}{\delta(\nu_k^{+}-\nu_k^{-})} 
= \frac{(\delta\lambda_k)^r}{\sqrt{(1-\delta\lambda_k)^2+2\mu\delta\lambda_k+(2\mu+\mu^2)\delta^2\lambda_k^2}} 
\le c
\end{align*}
for $c$ independent of both $\delta$ and $\lambda_k$. 
Letting $a_k^\pm v^\pm_{k,(1)}, a_k^\pm v^\pm_{k,(2)}$ denote the first and second component, respectively, of the vectors $a_k^\pm \bv_k^\pm$, for $0\le r\le 1$, we then have
\begin{align*}
\abs{a_k^{-}(v^-_{k,(1)} - v^-_{k,(2)})}\lambda_k^r &= \abs{\frac{\delta\nu_k^-}{\delta(\nu_k^{+}-\nu_k^{-})}}\lambda_k^r
=\le c\,\delta^{1-r}\abs{\nu_k^-}\,, \\
\abs{a_k^{+}(v^+_{k,(1)} - v^+_{k,(2)})}\lambda_k^r &= \abs{\frac{\delta\nu_k^+}{\delta(\nu_k^{+}-\nu_k^{-})}}\lambda_k^r 
\le c\,\delta^{1-r}\abs{\nu_k^+}\,.
\end{align*}

Using the decomposition \eqref{eq:A_decomp}, we then have
\begin{align*}
\norm{w_j-\varphi_j}_{\dot H^m} &\le
\norm{\sum_{k=1}^\infty\bigg(a_k^{-}(v^-_{k,(1)} - v^-_{k,(2)})e^{t\nu_k^-} + a_k^{+}(v^+_{k,(1)} - v^+_{k,(2)})e^{t\nu_k^+} \bigg)\lambda_k^{(j+m)/4}\wt{f}_k\psi_k}_{L^2} \\
&\le c\, \delta^{1-(j+m)/4}\,\sup_{k}\bigg(\abs{\nu_k^-}e^{t\nu_k^-} + \abs{\nu_k^+}e^{t\nu_k^+} \bigg)\norm{\sum_{k=1}^\infty \wt{f}_k\psi_k}_{L^2} \,.
\end{align*}
\end{proof}

Using Lemma \ref{lem:deltadiff}, we may now show part (c) of Theorem \ref{thm:per_deb}.
\begin{proof}[Proof of Theorem \ref{thm:per_deb}, part (c)]
We consider the $T$-periodic solution $(\overline\kappa,\xi)$ of part (a), which satisfies the bound \eqref{est:periodic}. Note that $c$ in \eqref{est:periodic} is bounded independent of $\delta$ as $\delta\to 0$, due to the $\delta$-independence of the constant $c$ in Lemma \ref{lem:semigrp}.\\

By $T$-periodicity, we have that $(\overline\kappa,\xi)$ at time $t\in[0,T]$ may be written
\begin{align*}
\begin{pmatrix}
\overline\kappa\\
\xi
\end{pmatrix} = e^{(t+NT)\mc{A}}\begin{pmatrix}
\overline\kappa^{\rm in}\\
\xi^{\rm in}
\end{pmatrix} +\begin{pmatrix}
\overline\kappa^{\rm f}\\
\xi^{\rm f}
\end{pmatrix}\,, \qquad
\begin{pmatrix}
\overline\kappa^{\rm f}\\
\xi^{\rm f}
\end{pmatrix}&:= \int_0^{t+NT} e^{(t+NT-t')\mc{A}} \begin{pmatrix}
-\dot\kappa_0+g_s\\
0
\end{pmatrix}\, dt'
\end{align*}
for any $N\in \N$, where $g(\overline\kappa,\xi,\kappa_0)= (1+\mu)\mc{N}[\overline\kappa,\kappa_0] - \mu(1+\gamma)(\overline\kappa+\kappa_0)^2\xi_s$. Recall that by \eqref{eq:g_bound} and the estimate \eqref{est:periodic} on $(\overline\kappa,\xi)$, we have that $\sup_{t\in[0,T]}\norm{g}_{L^2(I)}\le c\big(\sup_{t\in[0,T]}\norm{\dot\kappa_0}_{L^2} +\sup_{t\in[0,T]}\norm{\kappa_0}_{H^1} \big)$.\\

By Lemma \ref{lem:semigrp}, for $m=0,1$ we have 
\begin{align*}
\norm{e^{(t+NT)\mc{A}}\begin{pmatrix}
\overline\kappa^{\rm in}\\
\xi^{\rm in}
\end{pmatrix}}_{\dot H^m\times \dot H^m} 
&\le c\,e^{-(t+NT)\Lambda}\norm{\begin{pmatrix}
\overline\kappa^{\rm in}\\
\xi^{\rm in}
\end{pmatrix}}_{\dot H^m\times \dot H^m} 
\end{align*}
for $c$ independent of $\delta$ and $\Lambda=\min\{\lambda_1,\frac{1}{\delta(1+\mu)}\}$. In particular, for sufficiently small $\delta$, we have $\Lambda=\lambda_1$. Note that since $(\overline\kappa,\xi)\in H^1\times H^1$, we use $\dot H^m\times \dot H^m$, $m=0,1$ on the right hand side. For any (small) $\delta$, we may choose $N=N_\delta$ large enough that
\begin{equation}\label{eq:del_bd1}
\norm{e^{(t+N_\delta T)\mc{A}}\begin{pmatrix}
\overline\kappa^{\rm in}\\
\xi^{\rm in}
\end{pmatrix}}_{H^1\times H^1} 
\le \delta\,.
\end{equation}

Furthermore, for $N=N_\delta$ as above, by Lemma \ref{lem:deltadiff} we may estimate the difference $\overline\kappa^{\rm f}-\xi^{\rm f}$ as
\begin{equation}\label{eq:del_bd2}
\begin{aligned}
\norm{\overline\kappa^{\rm f}-\xi^{\rm f}}_{\dot H^m(I)} &\le c\,\delta^{1-m/4}\int_0^{t+N_\delta T} \,\sup_{k}\bigg(\abs{\nu_k^-}e^{t\nu_k^-} + \abs{\nu_k^+}e^{t\nu_k^+} \bigg) \norm{\dot\kappa_0}_{L^2(I)}\,dt' \\
&\qquad
+ c\,\delta^{1-(1+m)/4}\int_0^{t+N_\delta T} \sup_{k}\bigg(\abs{\nu_k^-}e^{t\nu_k^-} + \abs{\nu_k^+}e^{t\nu_k^+} \bigg) \norm{g}_{L^2(I)}\,dt' \\
&\le c\,\delta^{1-(1+m)/4}\bigg((1+\delta^{1/4})\sup_{t\in[0,T]}\norm{\dot\kappa_0}_{L^2} +\sup_{t\in[0,T]}\norm{\kappa_0}_{H^1} \bigg)\,.
\end{aligned}
\end{equation}
Here we have integrated in time and used the $T$-periodicity of both $\kappa_0$ and $g$ to take the supremum only over time $t\in[0,T]$. 
Combining \eqref{eq:del_bd1} and \eqref{eq:del_bd2}, as $\delta\to 0$ we obtain 
\begin{align*}
\norm{\overline\kappa-\xi}_{H^1(I)} &\le \delta + \norm{\overline\kappa^{\rm f}-\xi^{\rm f}}_{H^1(I)} 
\le \delta + c\,\delta^{1/2}\big(\sup_{t\in[0,T]}\norm{\dot\kappa_0}_{L^2} +\sup_{t\in[0,T]}\norm{\kappa_0}_{H^1} \big)\,.
\end{align*}
\end{proof}

\section{Viscoelastic swimming}\label{sec:VEswim}
In this section we give a proof of the fiber swimming expressions in Theorem \ref{thm:VEswimming}. We will first need a brief lemma. A version of this lemma also appears in the Newtonian case \cite{RFTpaper} and states that, given some additional regularity on our (small) $\kappa_0$, we can ensure that the fiber frame $(\be_{\rm t},\be_{\rm n})$ is not varying much over time. 

\begin{lemma}\label{lem:frame}
Suppose that $\kappa_0\in C^1([0,T];H^3(I))$ is $T$-periodic and satisfies
\begin{align*}
 \sup_{t\in[0,T]} \norm{\dot\kappa_0}_{L^2} = \varepsilon_1 \le \varepsilon\,, \qquad \sup_{t\in[0,T]} \norm{\kappa_0}_{H^1} = \varepsilon_2 \le \varepsilon\,,
 \end{align*} 
 for some $0<\varepsilon<1$, and let $(\overline\kappa,\xi)$ be the corresponding $T$-periodic solution to \eqref{eq:kappadot}-\eqref{eq:BCs}. The evolution of the fiber tangent vector $\be_{\rm t}$ \eqref{eq:frame_ev} then satisfies
 \begin{equation}
 \sup_{t\in[0,T]}\norm{\be_{\rm t}(\cdot,t)-\be_{\rm t}(0,0)}_{L^2(I)} \le c\,\varepsilon\,.
 \end{equation}
\end{lemma}

\begin{proof}
Since $\kappa_0\in H^3(I)$, we may use estimates \eqref{eq:H3term1} and \eqref{eq:H3term2} for the Duhamel formula \eqref{eq:duhamel} for $(\overline\kappa,\xi)$ along with Lemma \ref{lem:semigrp} to show  
\begin{align*}
\sup_{t\in[0,T]}\min\{t^{m/4},1\} \norm{\begin{pmatrix}
\overline\kappa \\
\xi
\end{pmatrix}}_{\dot H^m\times \dot H^m} 
\le c \,\norm{\begin{pmatrix}
\overline\kappa^{\rm in}\\
\xi^{\rm in}
\end{pmatrix}}_{L^2\times L^2}\,, \quad 0\le m\le 3\,.
\end{align*}
Due to the $T$-periodicity of $\kappa_0$, we in fact have
\begin{align*}
\sup_{t\in[0,T]} \norm{\begin{pmatrix}
\overline\kappa \\
\xi
\end{pmatrix}}_{\dot H^m\times \dot H^m} 
\le c \,\norm{\begin{pmatrix}
\overline\kappa^{\rm in}\\
\xi^{\rm in}
\end{pmatrix}}_{L^2\times L^2}\le c\,\varepsilon\,, \quad 0\le m\le 3\,.
\end{align*}

Using the equation \eqref{eq:theta_recover} for $\dot\theta$ in the frame evolution \eqref{eq:frame_ev}, we thus have
\begin{align*}
\sup_{t\in[0,T]}\norm{\be_{\rm t}(\cdot,t)-\be_{\rm t}(\cdot,0)}_{L^2(I)} &\le c\sup_{t\in[0,T]}\|\dot\theta\|_{L^2(I)}\\
&\le c\sup_{t\in[0,T]}\big(\norm{\overline\kappa_{sss}}_{L^2(I)}+\norm{\dot\kappa_0}_{L^2(I)} + \norm{\xi_{sss}}_{L^2(I)}+ \varepsilon^2\big) \\
&\le c\, \varepsilon\,.
\end{align*}

In addition, since $(\be_{\rm t})_s=\kappa\be_{\rm n}$, we have
\begin{align*}
\norm{\be_{\rm t}(\cdot,0)-\be_{\rm t}(0,0)}_{L^2(I)} &\le c\norm{\kappa}_{L^2(I)} \le c\,\varepsilon\,.
\end{align*}
Together, these two estimates give Lemma \ref{lem:frame}.
\end{proof}

Equipped with Lemma \ref{lem:frame}, we may now prove Theorem \ref{thm:VEswimming}.

\begin{proof}[Proof of Theorem \ref{thm:VEswimming}]
It will be convenient to define the difference $z:=\overline\kappa-\xi$ and work in terms of $\overline\kappa$ and $z$ rather than $\overline\kappa$ and $\xi$. Using the definition of $z$ and the equation \eqref{eq:VEoriginal1} for $\frac{\p\X}{\p t}$, we may calculate the velocity of the swimming fiber as 
\begin{align*}
\bm{V}(t) &= \int_0^1\frac{\p\X}{\p t}(s,t)\,ds = \bm{V}_{\rm vis}(t)+\bm{V}_{\rm ve}(t)\,,
\end{align*}
where $\bm{V}_{\rm vis}$ and $\bm{V}_{\rm ve}$ are given by
\begin{align*}
\bm{V}_{\rm vis}(t) &:= -\int_0^1\big({\bf I}+\gamma\be_{\rm t}\be_{\rm t}^{\rm T}\big) \big(-(\overline\kappa^2+2\overline\kappa\kappa_0)\be_{\rm t}+\overline\kappa_s\be_{\rm n}-\overline\tau\be_{\rm t}\big)_s \, ds \\
&=\gamma\int_0^1 \big(3\overline\kappa\overline\kappa_s+3\overline\kappa_s\kappa_0+2\overline\kappa(\kappa_0)_s+\overline\tau_s \big)\be_{\rm t} \, ds \,;\\
\bm{V}_{\rm ve}(t) &:=-\mu\int_0^1\big({\bf I}+\gamma\be_{\rm t}\be_{\rm t}^{\rm T}\big) \big(-(\overline\kappa^2+2\overline\kappa\kappa_0)\be_{\rm t}-\overline\tau\be_{\rm t}+z_s\be_{\rm n}\big)_s \, ds \\
&= \gamma\mu\int_0^1 \big((\overline\kappa^2+2\overline\kappa\kappa_0)_s+\overline\tau_s+z_s\overline\kappa
+z_s\kappa_0\big)\be_{\rm t} \, ds \,.
\end{align*}

Then, using Lemma \ref{lem:frame} along with the vanishing boundary conditions for $\overline\kappa$, $\overline\tau$, and $z$, for small $\kappa_0$ we may write
\begin{align*}
\bm{V}_{\rm vis}(t) &= \gamma\int_0^1 \kappa_0\overline\kappa_s\be_{\rm t}(0,0) \, ds + \bm{r}_{\rm vis}(t) = -\gamma\int_0^1 (\kappa_0)_s\overline\kappa \, ds\, \be_{\rm t}(0,0) + \bm{r}_{\rm vis}(t)\,,\\
\bm{V}_{\rm ve}(t) &= -\gamma\mu\int_0^1 z(\overline\kappa+\kappa_0)_s \, ds \, \be_{\rm t}(0,0)+ \bm{r}_{\rm ve}(t)\,,
\end{align*}
where both 
\begin{align*}
\abs{\bm{r}_{\rm vis}(t)} \le c\varepsilon^3\,, \quad \abs{\bm{r}_{\rm ve}(t)} \le c\varepsilon^3\,.
\end{align*}

We thus have 
\begin{align*}
\bm{V}(t)= \bigg( -\gamma\int_0^1 (\kappa_0)_s\overline\kappa\, ds -\gamma\mu\int_0^1z(\overline\kappa+\kappa_0)_s \, ds\bigg)\be_{\rm t}(0,0) + \bm{r}_{\rm vis}(t)+\bm{r}_{\rm ve}(t)\,,
\end{align*}
from which we obtain the first swimming expression \eqref{est:swimspeed1}-\eqref{est:Ut}. \\

To obtain the second expression \eqref{est:swimspeed2}, we first note that by Theorem \ref{thm:per_deb}, part (b), we may approximate $\overline\kappa$ and $z=\overline\kappa-\xi$ by $\overline\kappa^{\rm lin}$ and $z^{\rm lin}=\overline\kappa^{\rm lin}-\xi^{\rm lin}$, the $T$-periodic solutions to the linear equations
\begin{equation}\label{eq:kapz_lin}
\begin{aligned}
\dot{\overline\kappa}^{\rm lin}&= -\mc{L}\overline\kappa^{\rm lin}-\mu\mc{L}z^{\rm lin} - \dot\kappa_0 \\
\dot z^{\rm lin} &=\dot{\overline\kappa}^{\rm lin}-\delta^{-1}z^{\rm lin} \,.
\end{aligned}
\end{equation}
More specifically, defining 
\begin{equation}\label{eq:Ulin}
U^{\rm lin} = -\gamma\int_0^1 (\kappa_0)_s\overline\kappa^{\rm lin}\, ds -\gamma\mu\int_0^1z^{\rm lin}(\overline\kappa^{\rm lin}+\kappa_0)_s \, ds\,,
\end{equation}
we have that 
\begin{align*}
|U-U^{\rm lin}| &\le c\int_0^1 \abs{(\kappa_0)_s}|\overline\kappa-\overline\kappa^{\rm lin}|\, ds +\int_0^1|z-z^{\rm lin}|\big(\abs{\overline\kappa_s}+\abs{(\kappa_0)_s}\big) \, ds \\
&\qquad + \int_0^1|z^{\rm lin}|\,|(\overline\kappa-\overline\kappa^{\rm lin})_s| \, ds \\
&\le c\norm{\kappa_0}_{H^1}\|\overline\kappa-\overline\kappa^{\rm lin}\|_{L^2} +\|z-z^{\rm lin}\|_{L^2}\big(\norm{\overline\kappa}_{H^1}+\norm{\kappa_0}_{H^1}\big) + \|z^{\rm lin}\|_{L^2}\|\overline\kappa-\overline\kappa^{\rm lin}\|_{H^1} \\
&\le c\,\varepsilon^4\,. 
\end{align*}

Therefore, it suffices to use $\overline\kappa^{\rm lin}$ and $z^{\rm lin}$ to compute a more detailed expression for the time-averaged swimming speed $\langle U \rangle$. We being by solving for $\overline\kappa^{\rm lin}$ and $z^{\rm lin}$ in terms of $\kappa_0$. 
Defining $\omega=\frac{2\pi}{T}$, we expand each of $\kappa_0$, $\overline\kappa^{\rm lin}$, and $z^{\rm lin}$ as a Fourier series in time:
\begin{align*}
\kappa_0=\sum_{m=1}^\infty A_m(s)\cos(\omega m \,t)-&B_m(s)\sin(\omega m \,t)\,, \quad 
\overline\kappa^{\rm lin}=\sum_{m=1}^\infty C_m(s)\cos(\omega m \,t)-D_m(s)\sin(\omega m \,t)\,,\\
z^{\rm lin }&=\sum_{m=1}^\infty E_m(s)\cos(\omega m \,t)-F_m(s)\sin(\omega m \,t)\,.
\end{align*}

Using \eqref{eq:kapz_lin}, the coefficients of this expansion then satisfy the following system of equations:
\begin{align*}
-\omega mC_m&= \mc{L}D_m +\mu\mc{L}F_m+\omega m A_m\,, \qquad
-\omega mE_m =-\omega mC_m+\delta^{-1}F_m\,, \\
-\omega mD_m&= -\mc{L}C_m-\mu\mc{L}E_m+\omega mB_m \,, \qquad
-\omega mF_m =-\omega mD_m-\delta^{-1}E_m\,.
\end{align*}

Further expanding each of these coefficients in eigenfunctions \eqref{eq:L_eigs} of the operator $\mc{L}$, i.e.
\begin{align*}
A_m &= \sum_{k=1}^\infty a_{m,k}\psi_k\,,\quad B_m = \sum_{k=1}^\infty b_{m,k}\psi_k \,, \quad
C_m = \sum_{k=1}^\infty c_{m,k}\psi_k\,,\\
D_m &= \sum_{k=1}^\infty d_{m,k}\psi_k \,,\quad E_m=\sum_{k=1}^\infty e_{m,k}\psi_k\,,\quad F_m = \sum_{k=1}^\infty f_{m,k}\psi_k\,,
\end{align*}
we may solve for the coefficients $c_{m,k}$ and $d_{m,k}$ as
\begin{align*}
c_{m,k}&= Q_{m,k}b_{m,k} - H_{m,k}a_{m,k} \,, \qquad 
d_{m,k}= - Q_{m,k}a_{m,k}- H_{m,k}b_{m,k}\,,
\end{align*}
where
\begin{align*}
Q_{m,k}&= \frac{\lambda_k\omega m(1+(1+\mu)(\delta\omega m)^2)}{\lambda_k^2(1+(1+\mu)^2(\delta\omega m)^2)+\omega^2 m^2(2\mu\delta\lambda_k+1+(\delta\omega m)^2)}\,,
\\
H_{m,k}&=\frac{\omega^2 m^2(\mu\delta\lambda_k+ 1+(\delta\omega m)^2)}{\lambda_k^2(1+(1+\mu)^2(\delta\omega m)^2)+\omega^2 m^2(2\mu\delta\lambda_k+1+(\delta\omega m)^2)}\,.
\end{align*}

Additionally we may solve for $e_{m,k}$ and $f_{m,k}$ as
\begin{align*}
e_{m,k} &= \frac{\delta\omega m}{1+(\delta\omega m)^2}\bigg(Q_{m,k}(a_{m,k}+\delta\omega mb_{m,k})
+ H_{m,k}(b_{m,k}-\delta\omega ma_{m,k})\bigg) \\
f_{m,k} &=\frac{\delta\omega m}{1+(\delta\omega m)^2}\bigg(Q_{m,k}(b_{m,k} -\delta\omega ma_{m,k})
+ H_{m,k}(-a_{m,k}-\delta\omega mb_{m,k})\bigg) \,.
\end{align*}

We now need to use the above expansions in the expression \eqref{eq:Ulin} to calculate the average swimming speed $\langle U^{\rm lin}\rangle$. We first calculate
\begin{align*}
-\gamma&\int_0^1 \langle(\kappa_0)_s \overline\kappa^{\rm lin}\rangle \, ds  = 
-\gamma\sum_{m,k,\ell=1}^\infty\bigg(\frac{1}{2}a_{m,\ell}c_{m,k}+\frac{1}{2}b_{m,\ell}d_{m,k} \bigg)\int_0^1\psi_k(\psi_\ell)_s\,ds \\
&= \frac{\gamma}{2}\sum_{m,k,\ell=1}^\infty \bigg(Q_{m,k}(a_{m,k}b_{m,\ell}-a_{m,\ell}b_{m,k})+ H_{m,k}(a_{m,k}a_{m,\ell}+b_{m,k}b_{m,\ell})\bigg) \int_0^1\psi_k(\psi_\ell)_s\,ds \, .
\end{align*}

Furthermore, noting that 
\begin{align*}
(\overline\kappa^{\rm lin}+\kappa_0)_s = \sum_{m,\ell}(\psi_\ell)_s\big((a_{m,\ell}+c_{m,\ell})\cos(\omega mt)-(b_{m,\ell}+d_{m,\ell})\sin(\omega m t)\big)\,,
\end{align*}
we may also calculate
\begin{align*}
-\gamma\mu\int_0^1&\langle z^{\rm lin}(\overline\kappa^{\rm lin}+\kappa_0)_s\rangle \, ds =-\frac{\gamma\mu}{2}\sum_{m,k,\ell}\bigg((a_{m,\ell}+c_{m,\ell})e_{m,k}+ (b_{m,\ell}+d_{m,\ell})f_{m,k}\bigg)\int_0^1\psi_k(\psi_\ell)_s\,ds\\
&=-\frac{\gamma\mu}{2}\sum_{m,k,\ell}\frac{\delta\omega m}{1+(\delta\omega m)^2}\bigg((Q_{m,\ell}b_{m,\ell} +(1- H_{m,\ell})a_{m,\ell})\bigg(Q_{m,k}(a_{m,k}+\delta\omega mb_{m,k})\\
&\qquad 
+ H_{m,k}(b_{m,k}-\delta\omega ma_{m,k})\bigg)
+ ((1- H_{m,\ell})b_{m,\ell}- Q_{m,\ell}a_{m,\ell})\bigg(Q_{m,k}(b_{m,k} -\delta\omega ma_{m,k})\\
&\qquad - H_{m,k}(a_{m,k}+\delta\omega mb_{m,k})\bigg)\bigg)\int_0^1\psi_k(\psi_\ell)_s\,ds\,.
\end{align*}

Rearranging the above expression and combining the two components of $\langle U^{\rm lin}\rangle$, we obtain the form of the swimming speed reported in \eqref{est:swimspeed2} and \eqref{eq:Wexpr}.
\end{proof}

\appendix
\section{Numerical method}\label{app:nummeth}
For the numerical simulations of section \ref{sec:numerics}, we use the formulation introduced in \cite{RFTpaper}, which readily adapts to the viscoelastic setting. The formulation is based on a combination of works by Moreau, et al. \cite{moreau2018asymptotic} and Maxian, et al. \cite{maxian2021integral}. For convenience, we recall the original formulation \eqref{eq:VEoriginal1}-\eqref{eq:BCs_original} of the viscoelastic resistive force theory equations:
\begin{equation}\label{eq:VEoriginal}
\begin{aligned}
\frac{\p\X}{\p t}(s,t) &= -(1+\mu)\big({\bf I}+\gamma\X_s\X_s^{\rm T}\big)\big(\X_{sss}-\tau\X_s-(\kappa_0)_s\be_{\rm n}-\frac{\mu}{1+\mu}\xi_s\be_{\rm n}\big)_s \\
\delta\dot\xi &= -\xi + \kappa - \kappa_0 \\
\abs{\X_s}^2&=1 \\
(\X_{ss}-\kappa_0\be_{\rm n})\big|_{s=0,1}&=0\,, \quad (\X_{sss}-\tau\X_s-(\kappa_0)_s\be_{\rm n})\big|_{s=0,1}=0\,, \quad \xi\big|_{s=0,1}=\xi_s\big|_{s=0,1}=0\,.
\end{aligned}
\end{equation}
The formulation used in numerical simulations will be derived from \eqref{eq:VEoriginal}.
We begin by parameterizing the filament using the tangent angle description \eqref{eq:tangent}. In particular, we write
\begin{equation}\label{eq:param}
\X(s,t)=\X_0(t) + \int_0^s\be_{\rm t}(s',t)ds' \, ,\quad
\be_{\rm t}= \begin{pmatrix}
\cos\theta\\
\sin\theta
\end{pmatrix}
\,, \quad
\be_{\rm n}=\begin{pmatrix}
-\sin\theta\\
\cos\theta
\end{pmatrix}
\end{equation}

As in \cite{RFTpaper,moreau2018asymptotic,maxian2021integral}, we exploit that, due to the inextensibility constraint, only the normal components of the hydrodynamic force along the filament actually contribute to the fiber motion. In particular, using the parameterization \eqref{eq:param}, we may rewrite the first three equations of \eqref{eq:VEoriginal} as a closed system:
\begin{align}
\dot \X_0 +\int_0^s \dot\be_{\rm t}(s')\,ds' &= -(1+\mu)({\bf I}+\gamma\be_{\rm t}\be_{\rm t}^{\rm T})\bm{h}(s) \label{eq:form1} \\
\big({\bf I}-\be_{\rm t}(s)\be_{\rm t}(s)^{\rm T}\big)\int_0^s\bm{h}(s')ds' &= \big({\bf I}-\be_{\rm t}(s)\be_{\rm t}(s)^{\rm T}\big)\bigg(\X_{sss}-(\kappa_0)_s\be_{\rm n}-\frac{\mu}{1+\mu}\xi_s\be_{\rm n}\bigg) \label{eq:form2} \\
\delta\dot\xi &= -\xi + \kappa - \kappa_0\,.
\end{align}
Note that in \eqref{eq:form2}, by projecting away from the tangential direction along the filament, we have eliminated the need to solve for the unknown fiber tension. Instead, inextensibility is enforced directly via the parameterization \eqref{eq:param}. \\

Solving \eqref{eq:form1} directly for $\bm{h}$ and inserting this expression in \eqref{eq:form2}, we obtain the system
\begin{align}
\be_{\rm n}(s,t)\cdot\int_0^s({\bf I}-\frac{\gamma}{1+\gamma}\be_{\rm t}\be_{\rm t}^{\rm T})\bigg(\dot\X_0 +\int_0^{s'} \dot\be_{\rm t}(\bars)\,d\bars\bigg)ds' &= -(1+\mu)\theta_{ss}+(1+\mu)(\kappa_0)_s + \mu\xi_s  \label{eq:nummeth1}\\
\delta\dot\xi &= -\xi + \theta_s - \kappa_0  \label{eq:nummeth2}
\end{align}
for unknowns $\X_0(t)$, $\theta(s,t)$, and $\xi(s,t)$. Equations \eqref{eq:nummeth1} and \eqref{eq:nummeth2} serve as the basis for our numerical method. 
The boundary conditions $(\theta_s-\kappa_0)\big|_{s=0,1}=0$ are enforced directly in the discretization of $\theta_{ss}$ on the right hand side of \eqref{eq:nummeth1}, while $\xi\big|_{s=0,1}=\xi_s\big|_{s=0,1}=0$ is enforced in the discretization of $\xi$ in \eqref{eq:nummeth2}.
 To enforce the boundary condition $(-\theta_{ss}+(\kappa_0)_s)\big|_{s=1}=0$, we will also need to require 
\begin{equation}\label{eq:nummeth3}
\int_0^1({\bf I}-\frac{\gamma}{1+\gamma}\be_{\rm t}\be_{\rm t}^{\rm T})\bigg(\dot\X_0 +\int_0^s \dot\be_{\rm t}(s')\,ds'\bigg) \,ds =0.
\end{equation}
The analogous condition at $s=0$ is then satisfied automatically via the formulation \eqref{eq:nummeth1}.\\

We discretize the arclength coordinate $s\in[0,1]$ into $N+1$ equally spaced points $s_i$, $i=0,\dots,N$ and define $\X_i=\X(s_i)$. We consider the fiber as $N$ straight segments between each $\X_i$, and define $\theta_i$, $i=1,\dots,N$, to be the angle between segment $i$ and the $x$-axis. \\

The evolution equation \eqref{eq:nummeth1} is enforced at the midpoint of each segment $\X_{i-\frac{1}{2}}:=\frac{\X_{i-1}+\X_i}{2}$, $i=1,\dots,N$. In particular, we parameterize the evolution $\dot{\X}_{i-\frac{1}{2}}$ of each fiber segment as
\begin{align*}
\dot{\X}_{i-\frac{1}{2}} = \begin{pmatrix}
\dot x_0\\
\dot y_0
\end{pmatrix} 
+ \frac{1}{2N}\begin{pmatrix}
-\sin\theta_i\\
\cos\theta_i
\end{pmatrix}\dot\theta_i 
+\frac{1}{N}\sum_{k=1}^{i} \begin{pmatrix}
-\sin\theta_k\\
\cos\theta_k
\end{pmatrix}\dot\theta_k\,, \quad i=1,\dots,N\,.
\end{align*}
We also define $\kappa_{0,i}=\kappa_0(s_{i-\frac{1}{2}})$ where $s_{i-\frac{1}{2}}=\frac{s_{i-1}+s_i}{2}$, $i=1,\dots,N$, and we define $\xi_i$ similarly.\\ 

For the middle segments $j=2,\dots,N-1$, we obtain $2(N-2)$ equations from the discretization of \eqref{eq:nummeth1} and \eqref{eq:nummeth2}: 
\begin{align}
\frac{1}{N}\begin{pmatrix}
-\sin\theta_j \\
\cos\theta_j
\end{pmatrix}\cdot\sum_{i=1}^j\bm{M}_{\rm RFT}(\theta_i)\dot{\X}_{i-\frac{1}{2}} &= - N^2(1+\mu)(\theta_{j-1}-2\theta_j+\theta_{j+1}) + (1+\mu)(\kappa_0)_{s,j} \label{eq:disc1}\\
&\qquad +\mu \frac{N}{2}(\xi_{j+1}-\xi_{j-1}) \,, \quad j=2,\dots,N-1   \nonumber\\
\delta\dot\xi_j &= -\xi_j+ 2N(\theta_{j+1}-\theta_{j-1})-\kappa_{0,j} \,, \quad j=2,\dots,N-1  \label{eq:disc2}\,.
\end{align}
Here the $2N\times2N$ matrix $\bm{M}_{\rm RFT}(\theta_i)$ is given by 
\begin{align*}
\bm{M}_{\rm RFT}(\theta_i) = 
\begin{pmatrix}
1-\frac{\gamma}{1+\gamma}\cos^2\theta_i & -\frac{\gamma}{1+\gamma}\cos\theta_i\sin\theta_i \\
-\frac{\gamma}{1+\gamma}\cos\theta_i\sin\theta_i & 1-\frac{\gamma}{1+\gamma}\sin^2\theta_i
\end{pmatrix}\,.
\end{align*}

At the fiber endpoints, we set $\xi_1=\xi_N=0$ and enforce the boundary conditions $(-\theta_s+\kappa_0)_s\big|_{s=0,1}=0$ via the following two equations:
\begin{align}
\frac{1}{N}\begin{pmatrix}
-\sin\theta_1 \\
\cos\theta_1
\end{pmatrix}\cdot\bm{M}_{\rm RFT}(\theta_1)\dot{\X}_{\frac{1}{2}} &= -N^2(2\theta_2 -2\theta_1) + 2N\kappa_{0,1}    \label{eq:endpt0}\\
\frac{1}{N}\sum_{i=1}^{N-1}\bm{M}_{\rm RFT}(\theta_i)\dot{\X}_{i-\frac{1}{2}} &=0 \,,  \label{eq:endpt1}
\end{align}

To enforce $(\theta_s-\kappa_0)\big|_{s=0,1}=0$, we discretize $\theta_{ss}$ near the fiber endpoints as 
\begin{equation}\label{eq:theta_ss}
\theta_{ss}\big|_{s=0}\approx N^2(2\theta_2 -2\theta_1) - 2N\kappa_{0,1}\,,\qquad 
\theta_{ss}\big|_{s=1}\approx N^2(2\theta_{N-1} -2\theta_N) + 2N\kappa_{0,N}\, .
\end{equation}

At the $s=1$ endpoint, equations \eqref{eq:nummeth1} and \eqref{eq:nummeth3} coincide to give the boundary condition $\theta_{ss}\big|_{s=1}=(\kappa_0)_s\big|_{s=1}$, which, using \eqref{eq:theta_ss}, becomes an equation for $\theta_N$: 
\begin{equation}\label{eq:thetaN}
\theta_N = \theta_{N-1}+ \frac{1}{N}\kappa_{0,N} - \frac{1}{2N^2}(\kappa_0)_{s,N} \,.  
\end{equation}

Counting equations, we have $2(N-2)$ equations from \eqref{eq:disc1} and \eqref{eq:disc2}, 1 equation from \eqref{eq:endpt0}, 2 equations from \eqref{eq:endpt1}, and 1 equation from \eqref{eq:thetaN} for a total of $2N$ equations. These uniquely determine the $2N$ unknowns $x_0$, $y_0$, $\theta_1$, \dots, $\theta_N$, $\xi_2$, \dots, $\xi_{N-1}$. \\

The equations \eqref{eq:disc1}-\eqref{eq:thetaN} are evolved in time using a built-in ODE solver in Matlab.\\

\vspace{0.5cm}
{\bf Acknowledgments.} 
L.O. acknowledges support from NSF Postdoctoral Fellowship DMS-2001959 and thanks Becca Thomases, Yoichiro Mori, and Dallas Albritton for helpful discussion. \\

{\bf Declarations.}
Data sharing not applicable to this article as no datasets were generated or analysed during the current study. The author declares no conflict of interest.


\bibliographystyle{abbrv} 
\bibliography{VE_RFTbib}

\end{document}